\documentclass[10pt,reqno]{amsart}
\usepackage{amssymb,mathrsfs,graphicx}
\usepackage{ifthen}
\usepackage{hyperref}
\usepackage[margin=1in]{geometry}
\usepackage{caption}
\usepackage{sidecap}
\usepackage{rotating,dsfont}
\usepackage{colortbl}
\definecolor{black}{rgb}{0.0, 0.0, 0.0}
\definecolor{red}{rgb}{1.0, 0.5, 0.5}
\provideboolean{shownotes} 
\setboolean{shownotes}{true} 
\newcommand{\margnote}[1]{
\ifthenelse{\boolean{shownotes}}%
{\marginpar{\raggedright\tiny\texttt{#1}}}%
{}%
}
\newcommand{\hole}[1]{
\ifthenelse{\boolean{shownotes}}%
{\begin{center} \fbox{ \rule {.25cm}{0cm} \rule[-.1cm]{0cm}{.4cm}
\parbox{.85\textwidth}{\begin{center} \texttt{#1}\end{center}} \rule
{.25cm}{0cm}}\end{center}} {} }

\graphicspath{{pics/}}
\numberwithin{equation}{section}

\title[Analysis of Vlasov--Poisson--Navier--Stokes system]{On the dynamics of charged particles in an incompressible flow: from kinetic-fluid to fluid-fluid models}

\author[Choi]{Young-Pil Choi}
\address[Young-Pil Choi]{\newline Department of Mathematics\newline
Yonsei University, 50 Yonsei-Ro, Seodaemun-Gu, Seoul 03722, Republic of Korea}
\email{ypchoi@yonsei.ac.kr}

\author[Jung]{Jinwook Jung}
\address[Jinwook Jung]{\newline Research Institute of Basic Sciences \newline Seoul National University, Seoul  08826, Republic of Korea}
\email{warp100@snu.ac.kr}

\numberwithin{equation}{section}

\newtheorem{theorem}{Theorem}[section]
\newtheorem{lemma}{Lemma}[section]
\newtheorem{corollary}{Corollary}[section]
\newtheorem{proposition}{Proposition}[section]
\newtheorem{remark}{Remark}[section]
\newtheorem{definition}{Definition}[section]

\newcommand{\R}{\mathbb R}

\newcommand{\om}{\Omega}
\newcommand{\bbn}{\mathbb N}

\newcommand{\T}{\mathbb T}

\newcommand{\mc}{\mathcal C}

\newcommand{\bq}{\begin{equation}}
\newcommand{\eq}{\end{equation}}
\newcommand{\e}{\varepsilon}
\newcommand{\lt}{\left}
\newcommand{\rt}{\right}

\newcommand{\pa}{\partial}

\newcommand{\mh}{\mathcal{H}}
\newcommand{\me}{\mathcal{E}}

\newcommand{\ml}{\mathcal{L}}

\begin{document}
\allowdisplaybreaks

\date{\today}

\subjclass[]{}
\keywords{Hydrodynamic limit, Euler--Poisson system, kinetic-fluid models, incompressible Navier--Stokes system, global existence of solutions.}

\begin{abstract} In this paper, we are interested in the dynamics of charged particles interacting with the incompressible viscous flow. More precisely, we consider the Vlasov--Poisson or Vlasov--Poisson--Fokker--Planck equation coupled with the incompressible Navier--Stokes system through the drag force. For the proposed kinetic-fluid model, we study the asymptotic regime corresponding to strong local alignment and diffusion forces. Under suitable assumptions on the initial data, we rigorously derive a coupled isothermal/pressureless Euler--Poisson system and incompressible Navier--Stokes system(in short, EPNS system). For this hydrodynamic limit, we employ the modulated kinetic energy estimate together with the relative entropy method and the bounded Lipschitz distance. We also construct a global-in-time strong solvability for the isothermal/pressureless EPNS system. In particular, this global-in-time solvability gives the estimates of hydrodynamic limit hold for all time. 
\end{abstract}

\maketitle \centerline{\date}


%
%
%
%
\section{Introduction}\label{sec:1}
In this work, we are interested in the dynamics of charged particles immersed in an incompressible viscous fluid. The motion of numerous small solid charged particles (resp. affected by a white noise random force) can be described by the Vlasov--Poisson system (resp. Vlasov--Poisson--Fokker--Planck) and the incompressible viscous fluid can be modeled by the Navier--Stokes system \cite{AIK14,AKS10}, see also \cite{BT06,BT13,HDRD13,KC93} for the mathematical modeling on a charged particle in fluids. Specifically, let $f = f(x,\xi,t)$ be the number density of charged particles at position $x \in \T^d$ with velocity $\xi \in \R^d$ at time $t \in \R_+$ and $v = v(x,t)$ be the bulk velocity of the incompressible viscous fluid, respectively. Then, the dynamics of a pair $(f, v)$ is governed by the following Vlasov--Poisson--Navier--Stokes (in short, VPNS) system\footnote{To be more precise, if there is no diffusion in the kinetic equation in \eqref{main_eq}, i.e., $\sigma=0$, then the system \eqref{main_eq} would be called the Vlasov--Poisson--Navier--Stokes system. On the other hand, if $\sigma > 0$, then the system \eqref{main_eq} can be called the Vlasov--Poisson--Fokker--Planck--Navier--Stokes system. However, for simplicity, in the present work, we call \eqref{main_eq} Vlasov--Poisson--Navier--Stokes system.}:
\begin{align}\label{main_eq}
\begin{aligned}
&\partial_t f + \xi \cdot \nabla_x f + \nabla_\xi \cdot ( (F_d- \nabla_x U) f)  = \tau^{-1}\nabla_\xi \cdot (\sigma\nabla_\xi f - (u-\xi)f), \quad (x,\xi) \in \T^d \times \R^d, \ t > 0,\\
&-\Delta_x U = (\rho -c),\\
&\partial_t v + (v \cdot \nabla_x) v + \nabla_x p - \mu\Delta_x v =- \int_{\R^d} m_p F_d f\,d\xi,\\
&\nabla_x \cdot v = 0
\end{aligned}
\end{align}
subject to initial data:
\[
(f(x,\xi,0), v(x,0)) = (f_0(x,\xi), v_0(x)), \quad (x,\xi) \in \T^d \times \R^d,
\]
where $\rho = \rho(x,t)$ and $u = u(x,t)$ denote the local averaged particle density and velocity, respectively:
\[
\rho(x,t) := \int_{\R^d} f(x,\xi,t)\,d\xi \quad \mbox{and} \quad u(x,t) := \frac1{\rho(x,t)}\int_{\R^d} \xi f(x,\xi,t) \,d\xi.
\]
Here $\sigma \geq 0$, $\mu > 0$, and $\tau > 0$ represent the diffusion and viscosity coefficients,  relaxation time, respectively.  $c > 0$ denotes the background state which is chosen as
\[
c := \int_{\T^d} \rho\,dx.
\]
Throughout this paper, without loss of generality, we may assume that $c = 1$ due to the conservation of mass.  

The coupled kinetic-fluid models describe the motion of dispersed particles in an underlying gas.  It has received considerable attention due to its wide range of applications in the modeling of chemical engineering, atmospheric pollution, aerosols, and sprays \cite{BBJM05,BGLM15,OR81,Wi58}. We refer to \cite{Des10, OR81} for more physical backgrounds and modeling issues for the interactions between particles and fluids. Our main system \eqref{main_eq} consists of the Vlasov--Poisson/Vlasov--Poisson--Fokker--Planck system and the incompressible Navier--Stokes system, and these two systems are coupled through the drag force $F_d$ which is typically given by $m_p F_d = c_d \mu(v - \xi)$, which is also often called the friction or Brinkman force \cite{B49}. Here $m_p, c_d > 0$ denote the mass of one single particle and the coefficient of drag force, respectively. In \cite{DGR08}, this force is rigorously derived from rigid spherical particles immersed in the fluid by means of homogenization arguments. The term with $u - \xi$ on the right-hand side of \eqref{main_eq} describes the velocity-relaxation towards the local particle velocity $u$ which can also be derived from nonlinear velocity alignment forces \cite{KMT14}. If there is no coupling between particles and fluid, i.e., without drag forces, the kinetic part in \eqref{main_eq} becomes the nonlinear Vlasov--Poisson--Fokker--Planck system, see \cite{Vill02} for general discussion on the Vlasov--Fokker--Planck-type equations. For the safe of simplicity, we assume that the constants $m_p$, $\mu$, and $c_d$ equal $1$ in the sequel.

In the current work, we study an asymptotic analysis for the VPNS system \eqref{main_eq} corresponding to both the local alignment and the Brownian effects that are very strong, i.e. $\tau=\e \ll 1$. More precisely, for each $\e > 0$ we consider 
\begin{align}\label{A-2}
\begin{aligned}
&\partial_t f^\e + \xi \cdot \nabla_x f^\e + \nabla_\xi \cdot ( ((v^\e-\xi)-  \nabla_x U^\e) f^\e)  = \frac{1}{\e}\nabla_\xi \cdot (\sigma \nabla_\xi f^\e -(u^\e-\xi)f^\e)\\
&-\Delta_x  U^\e = \rho^\e - 1,\\
&\partial_t v^\e + (v^\e \cdot \nabla_x) v^\e + \nabla_x p^\e -\Delta _x v^\e =- \int_{\R^d} (v^\e-\xi)f^\e\,d\xi,\\
&\nabla_x \cdot v^\e = 0.
\end{aligned}
\end{align}
Formally, if we send the singular parameter $\e \to 0$, then the right hand side of the kinetic equation in \eqref{A-2} becomes zero, and subsequently this yields
\[
f^\e(x,\xi,t) \simeq 
\left\{\begin{array}{lcl} \rho(x,t) \otimes \delta_{u(x,t)}(\xi) & \mbox{ if } & \sigma = 0,\\[2mm]
\frac{\rho(x,t)}{(2\pi\sigma)^{d/2}}e^{-\frac{|u-\xi|^2}{2\sigma}} & \mbox{ if } & \sigma >0.
\end{array}\right.
\]
This formal observation together with estimating velocity moments of $f^\e$ gives that the system \eqref{A-2} can be well approximated by the isothermal/pressureless Euler--Poisson system coupled with the incompressible Navier--Stokes system (in short, EPNS system):
\begin{align}
\begin{aligned}\label{A-3}
&\partial_t \rho + \nabla_x \cdot (\rho u) = 0, \quad x \in \T^d, \ t > 0,\\
&\partial_t (\rho u) + \nabla_x \cdot (\rho u \otimes u) + \sigma \nabla_x \rho= -\rho(u-v)  - \rho\nabla_x U,\\
&-\Delta_x U = \rho - 1,\\
&\partial_t v + (v \cdot\nabla_x )v + \nabla_x p - \Delta_x v = \rho(u-v),\\
&\nabla_x \cdot v =0.
\end{aligned}
\end{align} 

Our first main contribution is to make the above formal observation rigorous. To be more specific, we establish quantitative error estimates between solutions $(f^\e, v^\e)$ to \eqref{A-2} and $(\rho,u,v)$ to \eqref{A-3}. In particular, this shows the unique strong solution to the system \eqref{A-3} can be well approximated by weak solutions to the system \eqref{main_eq}. There is little literature on the rigorous asymptotic analysis for kinetic-fluid models. In \cite{MV08}, the Vlasov--Fokker--Planck equation coupled with the compressible Navier--Stokes system is considered, and the asymptotic analysis under the regime of strong diffusion and drag forces is investigated. In this limit, a two--phase macroscopic model is rigorously derived. Later, in \cite{CCK16}, the kinetic-fluid model consisting of the nonlinear Vlasov--Fokker--Planck equation and the incompressible Navier--Stokes system, i.e., our main system \eqref{main_eq} with $U \equiv 0$, is taken into account, and the isothermal Euler coupled with incompressible Navier--Stokes system is derived under the same asymptotic regime as ours. We also refer to \cite{CJpre, CJpre2, GJV04,GJV04_2} for other types of hydrodynamic limits of kinetic-fluid models. In those works, the presence of diffusion plays an important role in analyzing the modulated total energy or relative entropy based on the weak-strong uniqueness principle. In fact, it provides the strict convexity of the modulated total energy with respect to $\rho$, and it is crucially used to handle the strong coupling between kinetic and fluid equations in the framework of relative entropy method, see Remark \ref{rmk_comm}. It is well known that $L^1$ norm can be controlled by the relative entropy, for instance, Csisz\'ar--Kullback--Pinsker inequality. In the present work, we propose a unified approach for the hydrodynamic limit of the VPNS system \eqref{A-2}; our analyses do hold regardless of the presence of the diffusion force. We find that the modulated interaction energy arising from Coulomb interaction and the smoothing effect from viscosity in the fluid part of \eqref{A-2} enable us to control the strong coupling between kinetic and fluid equations. This remarkable observation enables us to have the hydrodynamic limit even in the absence of the diffusion in the kinetic equation in \eqref{A-2}.

Before stating the theorem precisely, let us introduce the notions of weak and strong solutions to the equations \eqref{main_eq} and \eqref{A-3}. For the weak solutions to \eqref{main_eq}, we consider two function spaces:
\[
\mathsf{H} := \{v \in L^2(\T^d) \ | \ \nabla_x  \cdot v =0\} \quad \mbox{and} \quad \mathsf{V} := \{v \in H^1(\T^d) \ | \ \nabla_x  \cdot v =0\}.
\]
We also write $\mathsf{V} '$ as the dual space of $\mathsf{V}$.
\begin{definition}\label{D2.1}
For $T \in (0,+\infty)$, we say a pair $(f,v)$ is a weak entropy solution to the system \eqref{main_eq} on the time interval $[0,T]$ if it satisfies the following:
\begin{enumerate}
\item[(i)]
$f \in L^\infty(0,T; (L_+^1 \cap L^\infty)(\T^d \times \R^d))$\footnote{$L^1_+(\T^d \times \R^d)$ denotes the set of non-negative $L^1(\T^d \times \R^d)$ functions.}, \quad $|\xi|^2 f \in L^\infty(0,T;L^1(\T^d\times\R^d))$,
\item[(ii)]
$v \in L^\infty(0,T;\mathsf{H}) \cap L^2(0,T;\mathsf{V})$, \quad $\partial_t v \in \mathcal{C}([0,T];\mathsf{V}')$,
\item[(iii)]
for every $\Phi\in \mathcal{C}_c^\infty(\T^d\times\R^d\times[0,T])$ with $\Phi(\cdot, \cdot, T)=0$,
\begin{align*}
&\int_0^T \iint_{\T^d \times \R^d} f \partial_t \Phi \,dxd\xi dt  + \iint_{\T^d \times \R^d} f_0 \Phi(\cdot, \cdot, 0)\,dxd\xi\\
&+ \int_0^T \iint_{\T^d \times \R^d} (\xi f) \cdot \nabla_x  \Phi + ((v-\xi)f + \tau^{-1}(u-\xi)f - (\nabla_x  U) f - \sigma \tau^{-1}\nabla_\xi f) \cdot \nabla_\xi \Phi\,dxd\xi dt=0,
\end{align*}
\item[(iv)]
for every $\phi \in H^1(\T^d)$ and a.e. $t>0$,
\[
\int_{\T^d} \nabla_x  U \cdot \nabla_x  \phi\,dx = \int_{\T^d} (\rho-1)\phi \,dx,
\]
\item[(v)]
for every $\Psi \in \mathcal{C}_c^\infty(\T^d\times[0,T])$ with $\nabla_x  \cdot \Psi=0$ and a.e. $t>0$,
\begin{align*}
\int_{\T^d}& v(x,t) \cdot \Psi(x,t)\,dx -\int_{\T^d} v_0(x)\cdot\Psi(x,0)\,dx -\int_0^t \int_{\T^d} \left( v \cdot \partial_t \Psi + (v\cdot \nabla_x )\Psi \cdot v -\nabla_x  v : \nabla_x\Psi \right) dxds \\
&= \int_0^t\int_{\T^d} \rho(u-v)\Psi\,dxds.
\end{align*}
\item[(vi)] $(f,v)$ satisfies the following entropy inequality: 
\begin{align}\label{entropy}
\begin{aligned}
&\iint_{\T^d \times \R^d} \left(\frac{|\xi|^2}{2} + \sigma \log f \right) f \,dxd\xi + \frac{1}{2}\int_{\T^d} |\nabla U|^2 \,dx + \frac{1}{2}\int_{\T^d} |v|^2\,dx\\
&\quad + \int_0^t\iint_{\T^d \times \R^d}\left(\frac{1}{\tau f} \left|\sigma \nabla_\xi f - (u-\xi)f\right|^2  + |v-\xi|^2 f\right) dxd\xi ds + \int_0^t\int_{\T^d}|\nabla v|^2  \,dxds\\
&\qquad  \le \iint_{\T^d \times \R^d} \left(\frac{|\xi|^2}{2} + \sigma \log f_0 \right) f_0 \,dxd\xi + \frac{1}{2}\int_{\T^d} |\nabla U_0|^2 \,dx + \frac{1}{2}\int_{\T^d} |v_0|^2\,dx+\sigma dt \|f_0\|_{L^1}.\\
\end{aligned}
\end{align} 
\end{enumerate}
\end{definition}
Since we deal with two cases; isothermal/pressureless cases, we introduce two different notions of strong solutions to the system \eqref{A-3}.

\begin{definition}[Isothermal EPNS system]\label{def_strong1} Let $d\geq 2$ and $s>d/2+1$. For a given $T \in (0,+\infty]$, we call $(\rho, u, v)$ a strong solution of \eqref{A-3} with $\sigma > 0$ on the time interval $[0,T)$ if $(\rho, u, v)$ satisfies the following conditions:
\begin{itemize}
\item[(i)] $(\rho, u) \in \mc([0,T); H^s(\T^d)) \times \mc([0,T); H^s(\T^d))$,
\item[(ii)] $v \in \mc([0,T); H^s(\T^d)) \cap L^2(0,T; H^{s+1}(\T^d))$,
\item[(iii)] $(\rho, u,v)$ satisfies the system \eqref{A-3} with $\sigma > 0$ in the sense of distributions.
\end{itemize}
\end{definition}

\begin{definition}[Pressureless EPNS system]\label{def_strong2} Let $d\geq 2$ and $s>d/2+1$. For a given $T \in (0,+\infty]$, we call $(\rho, u, v)$ a strong solution of \eqref{A-3} with $\sigma = 0$ on the time interval $[0,T)$ if $(\rho, u, v)$ satisfies the following conditions:
\begin{itemize}
\item[(i)] $(\rho, u) \in \mc([0,T); H^s(\T^d)) \times \mc([0,T); H^{s+1}(\T^d))$,
\item[(ii)] $v \in \mc([0,T); H^{s+1}(\T^d)) \cap L^2(0,T; H^{s+2}(\T^d))$,
\item[(iii)] $(\rho, u,v)$ satisfies the system \eqref{A-3} with $\sigma = 0$ in the sense of distributions.
\end{itemize}
\end{definition}

\begin{remark}As stated in Definition \ref{def_strong2}, compared to the case with pressure the different regularities for $\rho$ and $u$ are taken due to the absence of pressure. 
\end{remark}

We also introduce our main assumptions for the quantified hydrodynamic limit from the VPNS system \eqref{A-2} to the EPNS system \eqref{A-3} below.
\begin{itemize}
\item[{\bf (H1)}] The initial total energies of the systems \eqref{A-2} and \eqref{A-3} satisfy
\[
\begin{split}
\iint_{\T^d \times \R^d}& \left(\frac{|\xi|^2}{2} + \sigma \log f_0^\e\right)f_0^\e \,dxd\xi + \frac{1}{2}\int_{\T^d} |v_0^\e|^2\,dx \\
&- \int_{\T^d} \left(\frac{|u_0|^2}{2} + \sigma \log\rho_0\right)\rho_0\,dx - \frac{1}{2}\int_{\T^d} |v_0|^2\,dx = \mathcal{O}(\sqrt\e).
\end{split}
\]
\item[{\bf(H2)}] The initial modulated energies of the systems \eqref{A-2} and \eqref{A-3} satisfy
\[
\begin{split}
\int_{\T^d} &\rho_0^\e |u_0^\e - u_0|^2\,dx + \int_{\T^d} |v_0^\e - v_0|^2\,dx +\sigma \int_{\T^d} \int_{\rho_0}^{\rho_0^\e} \frac{\rho_0^\e-z}{z}\,dzdx + \int_{\T^d} |\nabla_x (U_0^\e-U_0)|^2\,dx = \mathcal{O}(\sqrt\e).
\end{split}
\]
\end{itemize}

\begin{theorem}\label{T2.2}
For $T>0$ and $d\ge 2$,  let $(f^\e, v^\e)$ be weak entropy solutions to the system \eqref{A-2} in the sense of Definition \ref{D2.1} on the time interval $[0,T]$ corresponding to initial data $(f_0^\e, v_0^\e)$. Let $(\rho, u, v)$ be the unique strong solution to the system \eqref{A-3} in the sense of Definition \ref{def_strong1} or \ref{def_strong2} on the time interval $[0,T]$ corresponding to initial data $(\rho_0, u_0, v_0)$. Assume that the initial data $(f^\e, v^\e)$ are well-prepared in the sense that {\bf (H1)}-{\bf (H2)} hold.
\begin{itemize}
\item[(i)] (Isothermal pressure case) If $\sigma>0$, we have the following convergences:
\begin{align*}
\begin{aligned}
(\rho^\e, \rho^\e u^\e) &\to (\rho,\rho u), \quad \rho^\e u^\e \otimes u^\e \to \rho u\otimes u \quad \mbox{a.e.} \quad \mbox{and} \quad L^\infty(0,T;L^1(\T^d)),\cr
\int_{\R^d} f^\e \xi\otimes \xi\,d\xi &\to \rho u\otimes u + \rho\, \mathbb{I}_d \quad \mbox{a.e.} \quad \mbox{and} \quad L^p(0,T;L^1(\T^d)) \quad \mbox{for} \quad 1 \leq p \leq 2, \quad \mbox{and}\cr
v^\e &\to v \quad \mbox{a.e.} \quad \mbox{and} \quad L^\infty(0,T;L^2(\T^d))
\end{aligned}
\end{align*}
as $\e \to 0$, where $\mathbb{I}_d$ is the $d\times d$ identity matrix. Furthermore if the relative entropy between $f_0^\e$ and $M_{\rho_0, u_0}$ is well-prepared in the following sense:
\[
\iint_{\T^d \times \R^d} \int_{M_{\rho_0,u_0}}^{f_0^\e} \frac{f_0^\e-z}{z}\,dzdxd\xi \to 0
\]
as $\e \to 0$, we have
\[
f^\e \to M_{\rho,u} := \frac{\rho}{(2\pi\sigma)^{d/2}}e^{-\frac{|u-\xi|^2}{2\sigma}} \quad \mbox{a.e.} \quad \mbox{and} \quad L^\infty(0,T;L^1(\T^d\times\R^d))
\]
as $\e \to 0$.
\item[(ii)] (Pressureless case) If $\sigma=0$, we have
\begin{align*}
(\rho^\e, \rho^\e u^\e) &\rightharpoonup  (\rho,\rho u),   \quad \rho^\e u^\e \otimes u^\e \rightharpoonup \rho u \otimes u, \quad \mbox{weakly in } L^\infty(0,T;\mathcal{M}(\T^d)), \cr
\int_{\R^d} f^\e \xi \otimes \xi\,d\xi &\rightharpoonup \rho u \otimes u \quad \mbox{weakly in } L^1(0,T;\mathcal{M}(\T^d)),\cr
f^\e &\rightharpoonup \rho \otimes \delta_{u}(\xi)  \quad \mbox{weakly in } L^p(0,T;\mathcal{M}(\T^d\times\R^d)) \quad\mbox{for} \quad 1\le p \le 2,\quad\mbox{and}\cr
v^\e &\to v \quad \mbox{a.e.} \quad \mbox{and} \quad L^\infty(0,T;L^2(\T^d))
\end{align*}
as $\e \to 0$, where $\mathcal{M}(\om)$ denotes the space of (signed) Radon measures on $\om$. 
\end{itemize}
\end{theorem}

\begin{remark} The assumptions {\bf (H1)}-{\bf (H2)} for the well-prepared initial data can be replaced by
\begin{itemize}
\item[{\bf (H1)$'$}]
\[
\iint_{\T^d \times \R^d} \left(\frac{|\xi|^2}{2} + \sigma \log f_0^\e\right)f_0^\e \,dxd\xi - \int_{\T^d} \left(\frac{|u_0|^2}{2} + \sigma \log\rho_0\right)\rho_0\,dx = \mathcal{O}(\sqrt\e).
\]
\item[{\bf(H2)$'$}]
\[
\|\rho^\e_0 - \rho_0\|_{L^2(\T^d)} = \mathcal{O}(\sqrt\e), \quad \|u_0^\e - u_0\|_{L^\infty(\T^d)} = \mathcal{O}(\sqrt\e), \quad \mbox{and} \quad \|v_0^\e - v_0\|_{L^\infty(\T^d)} = \mathcal{O}(\sqrt\e).
\]
\end{itemize}
Indeed, we notice that
\[
\|\nabla_x (U^\e - U)\|_{L^2(\T^d)} \leq C \|\rho^\e - \rho\|_{H^{-1}(\T^d)}  \le C \|\rho^\e - \rho\|_{L^2(\T^d)},
\]
where $C>0$ is independent of $\e > 0$. For the other replacements, we can use a similar argument as in \cite[Remark 1.5]{CJpre}.
\end{remark}

\begin{remark} In order to make all the computations rigorous, we may need slightly less regularity than that in Definitions \ref{def_strong1} and \ref{def_strong2}, see \cite{Cpre,LT17} for instance. More precisely, we only need to have the solutions $(\rho, u,v)$ for the system \eqref{A-3} satisfying 
\begin{enumerate}
\item[(i)] $(\rho, u,v) \in \mc([0,T];L^1(\T^d)) \times \mathcal{C}([0,T];W^{1,\infty}(\T^d)) \times \mathcal{C}([0,T];W^{1,\infty}(\T^d))$,
\item[(ii)] $(\rho, u,v)$ satisfies the following free energy estimate in the sense of distributions:
\begin{align*}
&\frac{d}{dt}\lt(\frac12\int_{\T^d} \rho|u|^2\,dx  + \sigma \int_{\T^d} \rho \log \rho\,dx + \frac12\int_{\T^d} U\rho\,dx +\frac{1}{2} \int_{\T^d} |v|^2\,dx\rt) \cr
&\quad + \int_{\T^d} |\nabla_x v|^2\,dx +\int_{\T^d} \rho|u-v|^2\,dx=0,
\end{align*}
\item[(iii)] $(\rho, u,v)$ satisfies the system \eqref{A-3} in the sense of distributions.
\end{enumerate}
\end{remark}

\begin{remark}\label{rmk_conv}It follows from the convergence $\|\nabla_x (U^\e - U)\|_{L^\infty(0,T;L^2(\T^d))}$ in Theorem \ref{T2.2} that
\[
\rho^\e \to \rho \quad \mbox{in}\quad  L^\infty(0,T;H^{-1}(\T^d))
\]
as $\e \to 0$.  Indeed, it can be easily checked that
\[
\|\rho^\e - \rho\|_{H^{-1}(\T^d)} \leq \|\nabla_x (U^\e - U)\|_{L^2(\T^d)}.
\]
Notice that $\rho^\e$ and $\rho$ do not belong to $H^{-1}(\T^d)$ but $\rho^\e-1$ and $\rho-1$ do.
\end{remark}

\begin{remark} Theorem \ref{T2.2} can be extended to the whole space case if we add additional assumption on the bounded Lipschitz distance between the initial densities $\rho_0^\e$ and $\rho_0$. See Remarks \ref{rmk_rel} and \ref{rmk_whole} for detailed discussions on that.
\end{remark}

As stated in Theorem \ref{T2.2}, in order to make the hydrodynamic limit from \eqref{A-2} to \eqref{A-3} fully rigorous, it requires to obtain the existence of weak solutions to the system \eqref{A-2} and the existence of a unique strong solution to the limiting system \eqref{A-3}. The existence theories for weak solutions of Vlasov--Navier--Stokes system and Vlasov--Fokker--Planck--Navier--Stokes system are well developed in \cite{BDGM09,Yu13} and  \cite{CCK16,CKL11,CJpre2,MV07}, respectively. More recently, in \cite{CY20, YY18}, the global-in-time existence of weak solutions for Vlasov--Navier--Stokes-type system is also established when the collisional interactions for the particles are also taken into account. The Coulomb interactions between particles are dealt with in \cite{AIK14,AKS10}, and the global-in-time weak solutions for Vlasov--Poisson(--Fokker--Planck)--Navier--Stokes system are constructed. Compared to those works, we include the local alignment force $u - \xi$ in the kinetic equation, and this introduces new difficulties. In particular, due to the lack of regularity of the local particle velocity $u = \int_{\R^d} \xi f\,d\xi / \int_{\R^d} f\,d\xi$, the method of characteristics can not be applied directly, nor is the weak compactness of the product $u f$ obvious. This requires additional technical arguments, for instance, regularizations, weak/strong compactness arguments, velocity averaging lemma, and some uniform entropy inequalities. We  would like to emphasize that we also show that our weak solutions satisfy the entropy inequality \eqref{entropy} in Definition \ref{D2.1} (vi), which is crucially used in the proof of Theorem \ref{T2.2}. In fact, showing the entropy inequality demands a proper way of regularizing the system \eqref{main_eq}. One can simply mollify the local particle density $u$ and the Coulomb interactions by using the standard mollifier, then this would give the existence of weak solutions. However, this strategy grinds out some technical difficulties in establishing the entropy inequality \eqref{entropy} in Definition \ref{D2.1} (vi). Since we cannot specify appropriate references on it, we provide details on the global-in-time existence theory for the system \eqref{main_eq} in Appendix \ref{sec:weak}. 

\vspace{0.4cm}

The global-in-time existence of weak solutions to the system \eqref{main_eq} implies that the quantified hydrodynamic limit estimates in Theorem \ref{T2.2} hold as long as there exists a strong solution to the system \eqref{A-3}. Thus, our second main result is concerned with the global-in-time existence and uniqueness of strong solutions to the EPNS system \eqref{A-3},  investigated in Section \ref{sec:strong_ext}. As mentioned before, the asymptotic analysis stated in the previous section is based on the weak-strong uniqueness principle, thus we need a strong regularity of solutions to the limiting system \eqref{A-3}. 

The theorems below provide the global-in-time solvability of strong solutions to the system \eqref{A-3} under suitable smoothness and smallness assumptions on the initial data. These together with Theorem \ref{T2.2} assert that the hydrodynamic limit from \eqref{A-2} to \eqref{A-3} holds for all time $t \geq 0$.  

\begin{theorem}\label{T2.3}
Let $d\geq2$ and $s>d/2+1$. Suppose that the initial data $(\rho_0, u_0, v_0)$ satisfy
\begin{itemize}
\item[(i)] $\inf_{x\in\T^d}\rho_0(x)>0$ and 
\item[(ii)] $(\rho_0, u_0, v_0) \in H^s(\T^d) \times H^s(\T^d) \times H^s(\T^d)$.
\end{itemize}
If 
\[
\|\log \rho_0\|_{H^{s}(\T^d)}^2 + \|u_0\|_{H^{s}(\T^d)}^2 + \|v_0\|_{H^{s}(\T^d)}^2 <\tilde\e_0^2,
\]
for some $\tilde\e_0>0$ sufficiently small, then the system \eqref{A-3} with $\sigma > 0$ admits a global-in-time unique strong solution in the sense of Definition \ref{def_strong1} with $T = +\infty$ satisfying 
\[
\sup_{t \ge 0}\lt(\|\log \rho(\cdot, t)\|_{H^{s}(\T^d)}^2 + \|u(\cdot, t)\|_{H^{s}(\T^d)}^2 + \|v(\cdot, t)\|_{H^{s}(\T^d)}^2 \rt)<\infty. 
\]
\end{theorem}

\begin{theorem}\label{T2.3_2}
Let $d\geq2$ and $s>d/2+1$. Suppose that the initial data $(\rho_0, u_0, v_0)$ satisfy
\begin{itemize}
\item[(i)] $\inf_{x\in\T^d}\rho_0>0$ and 
\item[(ii)] $(\rho_0, u_0, v_0) \in H^s(\T^d) \times H^{s+1}(\T^d) \times H^{s+1}(\T^d)$.
\end{itemize}
If 
\[
 \|\rho_0 - 1\|_{H^s(\T^d)}^2 + \|u_0\|_{H^{s+1}(\T^d)}^2 + \|v_0\|_{H^{s+1}(\T^d)}^2 <\tilde\e_0^2,
\]
for some $\tilde\e_0>0$ sufficiently small, then the system \eqref{A-3} with $\sigma = 0$ admits a global-in-time unique strong solution in the sense of Definition \ref{def_strong2} with $T = +\infty$ satisfying 
\[
\sup_{t \ge 0}\lt(\|(\rho-1)(\cdot, t)\|_{H^s(\T^d)}^2 + \|u(\cdot, t)\|_{H^{s+1}(\T^d)}^2 + \|v(\cdot, t)\|_{H^{s+1}(\T^d)}^2\rt)<\infty. 
\]
\end{theorem}

Note that the limiting system \eqref{A-3} consists of the isothermal/pressureless Euler--Poisson system and the incompressible Navier--Stokes system. It is well-known that the Euler--Poisson equations develop a formation of singularities in finite time no matter how smooth the initial data are \cite{CCZ16, CW96, En98, ELT01, LL13, LT02, LT03, Pe90}. For that reason, it is not obvious to expect the global-in-time smooth regularity of solutions to the isothermal/pressureless EPNS system \eqref{A-3}. On the other hand, it is worth noticing that the incompressible Navier--Stokes system has a dissipative structure; it includes a smoothing effect of the viscous term. Thus it is required to use that to prevent the formation of finite-time singularities. 

Without the Poisson interaction, i.e. the isothermal/pressureless Euler system coupled with Navier--Stokes system, the global-in-time well-posedness theory is now well-established. In the presence of pressure, the drag force is properly used to transfer the smoothing effect in the viscous fluid to the isothermal Euler part and we can have a dissipative term for $\rho$ from the pressure. Thus the global-in-time unique strong solution can be constructed (see \cite{Choi15,Choi16}). In the pressureless case, it is easily noticeable that there is no direct dissipative effect for the fluid density $\rho$, thus a large-time behavior estimate of solutions to pressureless Euler--Navier--Stokes-type system is used in \cite{CK16,CJpre3,HKK14} to control the fluid density of the pressureless Euler part, and the global-in-time unique strong solution is obtained. More precisely, in those works, the large-time behavior estimate gives some time-integrability of $\|u\|_{H^{s+2}}$ for some $s>0$. This subsequently yields the uniform-in-time bound on $\rho$ in $H^s$-space. Here the regularity of $u$ in $L^2$ should be higher than that of $\rho$ by two. 

For the isothermal EPNS system, i.e. \eqref{A-3} with $\sigma > 0$, we use a similar argument as in \cite{Choi15,Choi16}. We take into account the drag force in the Euler--Poisson system in \eqref{A-3} as the relative damping to have a dissipation of the fluid velocity $u$. By this approach, we can handle the higher-order estimate of $u$ in the Euler--Poisson part. We then use the viscous term in the incompressible Navier--Stokes system in \eqref{A-3} through the drag force to control the growth of the incompressible fluid velocity $v$ in the Euler--Poisson system. Finally we can get a dissipative term for $\rho$ from the pressure, and thus the solutions $(\rho,u,v)$ belong to the same Sobolev spaces stated in Definition \ref{def_strong1}. On the other hand, for the pressureless EPNS system, as mentioned above, it is not clear how to have the dissipative term for the fluid density $\rho$ to construct the global-in-time solutions. Moreover, according our estimates, $u$ can only have higher $L^2$-regularity than $\rho$ by one, not two. This prevents us from using the large-time behavior of solutions. Instead, we analyze the Poisson term appropriately to have some dissipative effect on the density. Thus as stated in Definition \ref{def_strong2}, we were able to have solutions $\rho$ and $u$ in $H^s$ and $H^{s+1}$ spaces, respectively; $u$ only needs to have one more derivative in $L^2$ than $\rho$ has. For both cases, we use some estimates of crossing terms with a different order of derivatives of solutions and the mathematical induction for the homogeneous Sobolev spaces. 

Here we introduce several notations used throughout the current work. For functions, $f(x,v)$ and $g(x)$, $\|f\|_{L^p}$ and $\|g\|_{L^p}$ represent the usual $L^p(\T^d \times \R^d)$- and $L^p(\T^d)$-norms, respectively. We denote by $C$ a generic positive constant which may differ from line to line. $C = C(\alpha, \beta, \cdots)$ stands for a positive constant depending on $\alpha, \beta, \cdots$. For simplicity, we often omit $x$-dependence of differential operators, i.e. $\nabla f:= \nabla_x f$ and $\Delta f := \Delta_x f$. $\nabla^k$ represents any partial derivative $\pa^\alpha$ with multi-index $\alpha$, $|\alpha| = k$. $C = C(\alpha, \beta, \cdots)$ stands for a positive constant depending on $\alpha, \beta, \cdots$. Finally, we write $f\lesssim g$ if there exists a constant $C>0$ satisfying $f \le Cg$. 

The rest of this paper is organized as follows. In Section \ref{sec:h_limit}, we provide the details of proof of Theorem \ref{T2.2}. We first get the uniform-in-$\e$ upper bound for weak solutions and modified entropy inequality. Then, we investigate the modulated energy functional  and obtain its dissipation estimate, which yields the desired convergences.  Section \ref{sec:strong_ext} is devoted to present the details of proofs of Theorems \ref{T2.3} and \ref{T2.3_2}. Under the smallness assumption on the initial data, we can obtain the uniform-in-time upper bounds for strong solutions, leading to the global existence results.  Appendix \ref{sec:weak} addresses the existence of weak entropy solutions to system \eqref{main_eq} in the sense of Definition \ref{D2.1}. In Appendix \ref{app_conv}, we provide the proof of Proposition \ref{prop_h1}.  Finally, in Appendix \ref{app_local}, the local well-posedness theory for the EPNS system \eqref{A-3} is established. 
%
%
%
%
%
%

\section{Hydrodynamic limit: from VPNS to EPNS}\label{sec:h_limit}
 In this section, we present the rigorous derivation of the EPNS system \eqref{A-3} from the VPNS \eqref{main_eq} with $\sigma \geq 0$. We first notice that the potential $U$ can be represented by using the interaction potential $K$ which satisfies the following conditions (see \cite{BR93, Tit58}):
\begin{enumerate}
\item[(i)] The potential $K$ is an even function explicitly written as
\[
K(x) = \left\{\begin{array}{lcl}\displaystyle  -c_0 \log |x| + G_0(x) & \mbox{ if } & d=2, \\
c_1|x|^{2-d} + G_1(x) & \mbox{ if } & d\ge 3,\end{array}\right.
\]
where $c_0>0$ and $c_1>0$ are normalization constants and $G_0$ and $G_1$ are smooth functions over $\T^2$ and $\T^d$ ($d\ge 3$), respectively. 
\vspace{0.2cm}

\item[(ii)] For any $h \in L^2(\T^d)$ with $\int_{\T^d} h \,dx =0$, $U := K\star h \in H^1(\T^d)$ is the unique function that satisfies the following condition:
\[
\int_{\T^d} U\,dx = 0 \quad \mbox{and} \quad \int_{\T^d} \nabla U \cdot \nabla \psi\,dx = \int_{\T^d} h\,\psi\,dx \quad \forall\, \psi \in H^1(\T^d),
\]
i.e., $U$ is the unique weak solution to $-\Delta U = h$.
\end{enumerate}
 Thus, we rewrite the system \eqref{main_eq} as
\begin{align*}
&\partial_t f + \xi \cdot \nabla f + \nabla_\xi \cdot ( ((v-\xi)-\nabla K\star(\rho-1)) f)  = \nabla_\xi \cdot (\sigma\nabla_\xi f - (u-\xi)f),\\
&\partial_t v + (v \cdot \nabla) v + \nabla p -\Delta v =- \int_{\R^d} (v-\xi)f\,d\xi,\\
&\nabla \cdot v = 0.
\end{align*}

  As mentioned before, our proof is based on the modulated kinetic energy estimate and the relative entropy method. For this, we first provide the entropy inequality estimate. Set 
\[
\begin{split}
&\mathcal{F}(f,v):= \iint_{\T^d \times \R^d} \left(\frac{|\xi|^2}{2} + \sigma\log f\right) f\,dxd\xi + \frac{1}{2}\int_{\T^d} |\nabla K\star (\rho-1)|^2 \,dx + \int_{\T^d} \frac{1}{2}|v|^2\,dx,\\
&\mathcal{D}_1(f) := \iint_{\T^d \times \R^d}\frac{1}{f}|\sigma\nabla_\xi f - (u-\xi)f|^2\,dxd\xi, \quad \mbox{and}\\
&\mathcal{D}_2(f,v):= \iint_{\T^d \times \R^d} |v-\xi|^2 f\,dxd\xi + \int_{\T^d} |\nabla v|^2\,dx.
\end{split}
\]
Recall from Definition \ref{D2.1} (vi) that $\mathcal{F}(f^\e, v^\e)$ satisfies
\bq\label{D-1}
\mathcal{F}(f^\e, v^\e) + \frac{1}{\e}\int_0^t \mathcal{D}_1(f^\e) \, ds +\int_0^t \mathcal{D}_2(f^\e,v^\e) ds \le \mathcal{F}(f_0^\e, v_0^\e) + \sigma d t \|f_0^\e\|_{L^1}.
\eq
Without loss of generality, we assume $\|f_0^\e \|_{L^1} = 1$. In the lemma below, we show the uniform bound for the system \eqref{A-2}.

\begin{lemma}\label{L4.1}
For $T>0$, let $(f^\e, v^\e)$ be a weak entropy solution to the system \eqref{A-2} on the time interval $[0,T)$ corresponding to initial data $(f_0^\e, v_0^\e)$. Then we have
\begin{align*}
&\iint_{\T^d \times \R^d} \left( \frac{|x|^2}{4} +\frac{|\xi|^2}{4}+ \sigma |\log f^\e|\right)f^\e\,dxd\xi + \frac12\int_{\T^d} |\nabla K \star(\rho^\e-1)|^2\,dx+ \int_{\T^d} \frac{1}{2} |v^\e|^2 \,dx \\
&\quad +  \frac{1}{\e}\int_0^T \mathcal{D}_1(f^\e) \, ds +\int_0^T \mathcal{D}_2(f^\e,v^\e)\, ds \le C(\mathcal{F}(f_0^\e, v_0^\e) +1)e^{CT},
\end{align*}
where $C>0$ is a constant independent of $\e$.
\end{lemma}
\begin{proof}
First, we note that
\[
\iint_{\T^d \times \R^d} f^\e \log_{-}f^\e \,dx d\xi  \le \frac{1}{4\sigma}\iint_{\T^d \times \R^d} f^\e \left( 1 + |\xi|^2 \right) dx d\xi  + C,
\]
where $C=C(\sigma)$ is independent of $\e$. Then, we combine the above estimate with \eqref{D-1} and Gr\"onwall's lemma to get the desired estimate.
\end{proof}

Next, we investigate a modified entropy inequality.
\begin{lemma}\label{L4.2}
For $T>0$,  let $(f^\e, v^\e)$ be a weak entropy solution to \eqref{A-2} on the time interval $[0,T]$ corresponding to initial data $(f_0^\e, v_0^\e)$. Then we have
\[\begin{aligned}
\mathcal{F}(f^\e, v^\e) + \int_0^t \int_{\T^d} |\nabla v^\e|^2 \,dxds + \int_0^t \int_{\T^d} \rho^\e |u^\e - v^\e|^2 \, dxds + \frac{1}{2\e} \int_0^t \mathcal{D}_1(f^\e)\, ds\le \mathcal{F}(f_0^\e, v_0^\e) + C\e,
\end{aligned}\]
where $C >0$ is a constant independent of $\e$.
\end{lemma}
\begin{proof}
From a direct computation, we get
\begin{align*}\begin{aligned}
&\frac{1}{2}\iiiint_{\T^{2d} \times \R^{2d}} f^\varepsilon(x,\xi)f^\varepsilon(y,\xi_*)|\xi-\xi_*|^2 \,dxdyd\xi d\xi_*  + \int_{\T^d} \rho^\varepsilon|u^\varepsilon -v^\varepsilon|^2 \,dx\\
&\quad =\frac{1}{2} \iint_{\T^d \times \T^d} \rho^\varepsilon(x)\rho^\varepsilon (y)|u^\varepsilon(x) - u^\varepsilon(y)|^2 \,dxdy  +  \iint_{\T^d \times \R^d} |v^\varepsilon - \xi|^2 f^\varepsilon \,dxd\xi.
\end{aligned}\end{align*}
Then, we estimate the first term on the right hand side of the above equality in the same way as \cite[Lemma B.3]{KMT15}:
\begin{align}\label{est_en1}
\begin{aligned}
\frac{1}{2}& \iint_{\T^d \times \T^d} \rho^\varepsilon(x)\rho^\varepsilon (y)|u^\varepsilon(x) - u^\varepsilon(y)|^2 \,dxdy\\
& = \iint_{\T^d \times \T^d} \rho^\e (x) \rho^\e (y) (u^\e(x) - u^\e(y))\cdot u^\e(x) \,dxdy\\
&=\iiiint_{\T^{2d} \times \R^{2d}} f^\e(x,\xi)f^\e(y,\xi_*)(\xi-\xi_*)\cdot u^\e(x)\,dxdyd\xi d\xi_* \\
&= \iiiint_{\T^{2d} \times \R^{2d}}f^\e(y,\xi_*)f^\e(x,\xi)(\xi-\xi_*)\cdot  \xi \,dxdyd\xi d\xi_* \\
&\quad + \iiiint_{\T^{2d} \times \R^{2d}} f^\e(y,\xi_*)(\xi-\xi_*) \cdot (f^\e(x,\xi)(u^\e(x)-\xi)-\sigma \nabla_\xi f(x,\xi)) \,dxdyd\xi d\xi_* \\
& \quad +\iiiint_{\T^{2d} \times \R^{2d}}f^\e(y,\xi_*)(\xi-\xi_*)\cdot \sigma \nabla_\xi f(x,\xi)\,dxdyd\xi d\xi_* \\
& =: \mathcal{I}_1 + \mathcal{I}_2 + \mathcal{I}_3.
\end{aligned}
\end{align}
For $\mathcal{I}_1$, we use the change of variables $(x,\xi) \leftrightarrow (y,\xi_*)$ to get
\[
\mathcal{I}_1 = \frac{1}{2}\iiiint_{\T^{2d} \times \R^{2d}}f^\e(y,\xi_*)f^\e(x,\xi)|\xi-\xi_*|^2 \,dxdyd\xi d\xi_* . 
\]
For the estimate of $\mathcal{I}_2$, similarly to \cite{KMT15}, we set 
\[
V^\e(x,\xi) := \frac{1}{\sqrt{f^\e(x,\xi)}}(f^\e(x,\xi)(u^\e(x)-\xi)-\sigma \nabla_\xi f(x,\xi)).
\]
Then we obtain
\begin{align*}
\mathcal{I}_2 &=\int_{\T^d \times \T^d \times \R^d}\sqrt{f^\e(x,\xi)}\rho^\e(y)(\xi - u^\e(y))\cdot V^\e(x,\xi)\,dxdy d\xi \\
&=  \left(\int_{\T^d} \rho^\e(y)\,dy\right)\iint_{\T^d \times \R^d} \xi\sqrt{f^\e(x,\xi)} \cdot V^\e(x,\xi)\,dxd\xi\\
&\quad - \left(\int_{\T^d} (\rho^\e u^\e)(y)\,dy\right)  \iint_{\T^d \times \R^d} \sqrt{f^\e(x,\xi)} \cdot V^\e(x,\xi)\,dxd\xi\\
&\le 2\|f_0^\e\|_{L^1}\left(\iint_{\T^d \times \R^d} |\xi|^2 f^\e\,dxd\xi\right)^{1/2}\left(\iint_{\T^d \times \R^d} (V^\e(x,\xi))^2\,dxd\xi\right)^{1/2}\\
&\le \frac{1}{2\e} D_1(f^\e) + C\e,
\end{align*}
where $C>0$ is independent of $\e$, and we used Cauchy--Schwarz inequality and
\[
\rho^\e |u^\e|^2 \le \int_{\R^d}|\xi|^2 f^\e\,d\xi.
\]
For $\mathcal{I}_3$, we directly estimate
\[
\mathcal{I}_3= -\sigma d \int_{\T^{2d}\times\R^{2d}}f^\e(y,\xi_*)f^\e(x,\xi)\,dxdyd\xi d\xi_* = -d\sigma.
\]
We combine the estimates for $\mathcal{I}_i, i=1,2,3$, and put it into \eqref{est_en1} to complete the proof.
\end{proof}

\subsection{Modulated energy estimates}\label{sec:rel}
In this subsection, we provide several estimates regarding the modulated energy associated to the systems \eqref{A-2} and \eqref{A-3}. For this purpose, we define
\[
U := \left(\begin{array}{c}\rho \\ m \\ v\end{array}\right) \quad \mbox{and} \quad A(U) := \left(\begin{array}{ccc}m & 0 & 0 \\ \frac{m\otimes m}{\rho} & \sigma \rho \mathbb{I}_d & 0 \\ 0 & 0 & v\otimes v \end{array}\right),
\]
where $m = \rho u$. Then, the system \eqref{A-3} can be recast in the form of hyperbolic conservation law:
\[
\partial_t U + \nabla \cdot A(U) = F(U),
\]
where $F = F(U)$ is given by
\[
F(U) := \left(\begin{array}{c}0 \\ -\rho(u-v) -\rho \nabla K \star (\rho-1) \\ -\nabla p +\Delta v + \rho(u-v)\end{array}\right).
\]
Then, the associated total energy can be written as
\[
E(U) := \frac{|m|^2}{2\rho} + \frac{|v|^2}{2} +\sigma \rho \log \rho.
\]
With these newly defined macroscopic variables, we define a modulated energy functional as follows:
\[
\me(\bar{U}| U) = E({\bar U}) - E(U) - DE(U)(\bar{U} -U), \quad {\bar U} := \left(\begin{array}{c}{\bar \rho} \\ {\bar m} \\ {\bar v}\end{array}\right),
\]
which can be rewritten as
\[
\me(\bar{U}| U) = \frac{{\bar\rho}}{2}|{\bar u} - u|^2 + \frac{1}{2}|{\bar v} - v|^2 + \sigma \mh(\bar\rho | \rho),
\]
where the relative entropy $\mh(x|y)$ is given by
\[
\mh(x|y) := x\log x - y\log y - (1+\log y)(x-y) \ge \frac{1}{2}\min\left\{\frac{1}{x}, \frac{1}{y} \right\}|x-y|^2.
\]
Note that if $\sigma = 0$, then the modulate energy $\me$ does not include the relative entropy $\mh$, and thus it is not convex with respect to the fluid density $\rho$.  As mentioned in Introduction, this makes a huge problem with the estimate of a term that appears due to the drag force. Regarding this issue, we do not employ the relative entropy $\mh$ to handle that term but use the Coulomb interactions to overcome this difficulty. See Remark \ref{rmk_comm} for more details.

Now, we provide the estimate for the modulated energy functional.

\begin{lemma}\label{L4.3}
The modulated energy functional $\me$ satisfies the following differential equality:
\begin{align*}
\frac{d}{dt}&\int_{\T^d} \me({\bar U}|U)\,dx + \int_{\T^d} |\nabla ({\bar v}-v)|^2\,dx + \int_{\T^d} {\bar \rho}|({\bar u} - {\bar v}) - (u-v)|^2\,dx\\
&= \int_{\T^d} \partial_t E({\bar U})\,dx +\int_{\T^d} {\bar \rho}|{\bar u}-u|^2\,dx + \int_{\T^d} |\nabla {\bar v}|^2\,dx +\int_{\T^d} {\bar \rho}{\bar u}\cdot \nabla K \star ({\bar \rho}-1)\,dx\\
&\quad -\int_{\T^d} \nabla DE(U):A({\bar U}|U)\,dx-\int_{\T^d} DE(U)(\pa_t {\bar U} + \nabla \cdot A({\bar U}) -F({\bar U}))\,dx\\
&\quad + \int_{\T^d} ({\bar \rho}-\rho)({\bar v}-v)(u-v)\,dx -\int_{\T^d} {\bar \rho}({\bar u}-u)\cdot \nabla K \star ({\bar \rho}-\rho)\,dx,
\end{align*}
where $A({\bar U}|U) := A({\bar U})-A(U) -DA(U)({\bar U} - U)$ is the modulated flux functional.
\end{lemma}
\begin{proof}
A straightforward computation yields
\begin{align*}
\frac{d}{dt}\int_{\T^d} \me({\bar U}|U)\,dx &= \int_{\T^d} \partial_t E({\bar U})\,dx -\int_{\T^d} DE(U)(\pa_t {\bar U} + \nabla \cdot A({\bar U}) -F({\bar U}))\,dx\\
&\quad + \int_{\T^d} D^2E(U)\nabla \cdot A(U)({\bar U}- U) + DE(U) \nabla \cdot A({\bar U})\,dx\\
&\quad -\int_{\T^d} D^2E(U)F(U)({\bar U}-U) + DE(U)F({\bar U})\,dx\\
&=: \sum_{i=1}^4 \mathcal{J}_i.
\end{align*}
For $\mathcal{J}_3$, we use the same estimate in \cite[Appendix A]{CCK16} to get 
\[
\mathcal{J}_3 = -\int_{\T^d} \nabla DE(U):A({\bar U}|U)\,dx.
\]
For $\mathcal{J}_4$, we note that
\[
DE(U) = \left(\begin{array}{c}-m^2/(2\rho^2) + \sigma(\log\rho+1)\\ m/\rho \\ v\end{array}\right) \quad \mbox{and} \quad D^2E(U) = \left(\begin{array}{ccc} * & -m/\rho^2 & 0 \\ * & 1/\rho & 0 \\ 0 & 0 & 1\end{array}\right).
\]
Then we find
\begin{align*}
D^2E(U)F(U)({\bar U}- U) &=({\bar \rho} - \rho)u \cdot \nabla K \star (\rho-1) - u \cdot (v-u)({\bar \rho} - \rho)\\
&\quad - ({\bar m}-m)\cdot\nabla K \star (\rho-1) + ({\bar m} - m)\cdot (v-u)\\
&\quad -\nabla p \cdot ({\bar v} - v) + \Delta v\cdot({\bar v} - v) + \rho(u-v)\cdot({\bar v}-v)
\end{align*} 
and
\[
DE(U)F({\bar U}) = -{\bar \rho}u \cdot \nabla K \star ({\bar\rho}-1) + {\bar \rho}u \cdot ({\bar v} - {\bar u}) -\nabla {\bar p} \cdot v + \Delta {\bar v} \cdot v + {\bar \rho} v \cdot ({\bar u} - {\bar v}). 
\]
Thus we obtain
\begin{align*}
&\int_{\T^d} D^2E(U)F(U)({\bar U}- U) + DE(U)F({\bar U})\,dx \\
&\quad = \int_{\T^d} {\bar \rho}(u-{\bar u}) \cdot \nabla K \star({\bar \rho} - \rho) -{\bar \rho}{\bar u}\cdot \nabla K \star ({\bar \rho}-1)\,dx+ \int_{\T^d} {\bar \rho} (u^2 - uv +2{\bar u}v - 2{\bar u}u + u{\bar v} - {\bar v}v)\,dx\\
&\qquad +\int_{\T^d} \rho(u{\bar v} - {\bar v}v -uv + v^2)\,dx + \int_{\T^d} \Delta v \cdot ({\bar v}-v) +\Delta {\bar v}\cdot v\,dx.
\end{align*}
This together with \cite[Appendix A]{CCK16} yields
\begin{align*}
\mathcal{J}_4 &= -\int_{\T^d} {\bar \rho}({\bar u}-u) \cdot \nabla K \star({\bar \rho} - \rho) dx +\int_{\T^d} {\bar \rho}{\bar u}\cdot \nabla K \star ({\bar \rho}-1)\,dx\\
&\quad + \int_{\T^d} {\bar \rho} |{\bar u}-u|^2\,dx + \int_{\T^d} |\nabla {\bar v}|^2\,dx-\int_{\T^d} {\bar \rho}|({\bar u}-{\bar v}) -(u-v)|^2\,dx\\
&\quad -\int_{\T^d} |\nabla({\bar v}-v)|^2\,dx +\int_{\T^d} ({\bar \rho}-\rho)({\bar v}-v)(u-v)\,dx.
\end{align*}
This completes the proof.
\end{proof}

In the proposition below, we present the estimate of the modulated energy associated to the systems \eqref{A-2} and \eqref{A-3}.
\begin{proposition}\label{P2.2}
For $T>0$ and $d\ge 2$, let $(f^\e, v^\e)$ be weak entropy solutions to the system \eqref{A-2} on the interval $[0,T]$ corresponding to initial data $(f_0^\e, v_0^\e)$. Let $(\rho, u, v)$ be the unique strong solution to the system \eqref{A-3} on the interval $[0,T]$ corresponding to initial data $(\rho_0, u_0, v_0)$. Assume that the initial data $(f^\e, v^\e)$ satisfy the assumptions {\bf (H1)}-{\bf (H2)} in Theorem \ref{T2.2}. Then we have
\begin{align*}
&\int_{\T^d} \rho^\e|u^\e-u|^2dx  + \int_{\T^d} |v^\e-v|^2\,dx + \sigma \int_{\T^d} \int_\rho^{\rho^\e} \frac{\rho^\e-z}{z}\,dzdx + \int_{\T^d} |\nabla K\star (\rho^\e-\rho)|^2\,dx \\
&\quad +\int_0^t \int_{\T^d} |\nabla (v^\e-v)|^2\,dxds + \int_0^t \int_{\T^d}  \rho^\e |(u^\e - v^\e) - (u-v)|^2\,dxds\\
&\qquad \le C\sqrt{\e}
\end{align*}
for almost every $t \in [0,T]$, where $C$ is independent of $\e$.
\end{proposition}
\begin{proof} First, we set
\[
U:= \left(\begin{array}{c} \rho \\ \rho u \\ v \end{array}\right) \quad \mbox{and} \quad U^\e := \left(\begin{array}{c} \rho^\e \\ \rho^\e u^\e \\ v^\e \end{array}\right), 
\]
where
\[
\rho^\e :=\int_{\R^d} f^\e \,d\xi \quad \mbox{and} \quad \rho^\e u^\e := \int_{\R^d}\xi f^\e \,d\xi.
\]
We replace ${\bar U}$ with $U^\e$ in Lemma \ref{L4.3} to get
\begin{align*}
&\int_{\T^d} \me(U^\e |U)\,dx + \int_0^t \int_{\T^d} |\nabla (v^\e-v)|^2\,dxds + \int_0^t \int_{\T^d}  \rho^\e |(u^\e - v^\e) - (u-v)|^2\,dxds\\
&= \int_{\T^d} \me(U_0^\e| U_0)\,dx + \int_0^t \int_{\T^d} \left(\partial_s E( U^\e )+\rho^\e |u^\e-u|^2\,+  |\nabla v^\e |^2 + \rho^\e  u^\e \cdot \nabla K \star (\rho^\e-1)\right) \,dxds\\
&\quad -\int_0^t \int_{\T^d} \nabla DE(U):A(U^\e|U)\,dxds -\int_0^t\int_{\T^d} DE(U)(\pa_t U^\e+ \nabla \cdot A( U^\e) -F(U^\e))\,dxds\\
&\quad + \int_0^t \int_{\T^d} (\rho^\e-\rho)(v^\e-v)(u-v)\,dxds -\int_0^t \int_{\T^d} \rho^\e(u^\e-u)\cdot \nabla K \star (\rho^\e-\rho)\,dxds\\
&=: \sum_{i=1}^6 \mathcal{K}_i.
\end{align*}
From now on, we separately estimate $\mathcal{K}_i, i=2,\dots,6$ as follows:\\

\noindent $\diamond$ (Estimates for $\mathcal{K}_2$): Note that
\[
\frac{1}{2}\frac{d}{dt}\int_{\T^d} |\nabla K\star(\rho^\e-1)|^2 \,dx = \int_{\T^d} \rho^\e u^\e \cdot \nabla K \star (\rho^\e-1) \,dx.
\]
This observation entails
\begin{align*}
\mathcal{K}_2 &= \mathcal{F}(f^\e, v^\e) + \int_0^t \int_{\T^d} \rho^\e |u^\e-u|^2\,dxds + \int_0^t \int_{\T^d}  |\nabla v^\e|^2\,dxds \\
&\quad - \int_{\T^d} \left(E(U_0^\e) + \frac{1}{2}|\nabla K\star (\rho_0^\e-1)|^2 \right)\,dx + \int_{\T^d}\rho^\e \frac{|u^\e|^2}{2} \,dx - \int_{\T^d} \frac{|\xi|^2}{2} f^\e\,dxd\xi\\
&\le C\e + \mathcal{F}(f_0^\e, v_0^\e) -\int_{\T^d} \left(E(U_0^\e) + \frac{1}{2}|\nabla K\star (\rho_0^\e-1)|^2 \right)\,dx\\
&\le \mathcal{O}(\sqrt{\e}),
\end{align*}
where we used the fact $\rho^\e |u^\e|^2 \le \int_{\R^d} |\xi|^2 f^\e \,d\xi$, {\bf (H1)}, and Lemma \ref{L4.2}.\\

\noindent $\diamond$ (Estimates for $\mathcal{K}_3$):  By the definition of $A(U^\e| U)$, we get
\[
A(U^\e|U) = \left(\begin{array}{ccc}0 & 0 & 0 \\ \rho^\e (u^\e - u)\otimes (u^\e - u) & 0 & 0 \\ 0 & 0 & (v^\e-v)\otimes (v^\e-v) \end{array}\right).
\]
This together with 
\[
DE(U) = \left(\begin{array}{c}*\\ u \\ v\end{array}\right)
\]
yields
\[
\mathcal{K}_3 \le \|\nabla (u,v)\|_{L^\infty} \int_0^t \int_{\T^d} \rho^\e |u^\e-u|^2 + |v^\e - v|^2\,dx \le C\int_0^t \int_{\T^d} \me(U^\e|U)\,dxds,
\]
where $C>0$ is a positive constant independent of $\e$.\\

\noindent $\diamond$ (Estimates for $\mathcal{K}_4$): We integrate the kinetic equation in \eqref{A-2} with respect to $\xi$ to find
\begin{align*}
&\partial_t \rho^\e + \nabla \cdot (\rho^\e u^\e) = 0,\\
&\partial_t (\rho^\e u^\e) + \nabla \cdot (\rho^\e u^\e\otimes u^\e) +\sigma\nabla\rho^\e +  \rho^\e( (u^\e - v^\e) +\nabla K \star(\rho^\e-1)) \\
&\hspace{4.5cm}= \nabla \cdot \left(\int_{\R^d} (u^\e \otimes u^\e -\xi \otimes \xi +\sigma \mathbb{I}_d)f^\e\,d\xi\right)\\
&\partial_t v^\e + (v^\e \cdot \nabla) v^\e + \nabla p^\e -\Delta v^\e = \rho^\e(u^\e-v^\e),\\
&\nabla \cdot v^\e =0
\end{align*}
in the sense of distributions. Thus we use Cauchy--Schwarz inequality to get
\begin{align*}
\mathcal{K}_4 &=-\int_0^t\int_{\T^d} DE(U)(\pa_t U^\e+ \nabla \cdot A( U^\e) -F(U^\e))\,dxds\\
&= \int_0^t \int_{\T^d} \nabla u : \left(\int_{\R^d} (u^\e \otimes u^\e - \xi\otimes \xi + \sigma\mathbb{I}_d)f^\e\,d\xi\right)\,dxds.
\end{align*}
Then, by \cite[Lemma 4.4]{KMT15}, we estimate
\begin{align*}
&\int_{\R^d} (u^\e \otimes u^\e - \xi\otimes\xi +\sigma \mathbb{I}_d)f^\e \,d\xi\\
&\quad=\int_{\R^d}(u^\e \otimes (u^\e-\xi) + (u^\e- \xi)\otimes\xi +\sigma \mathbb{I}_d)f^\e \,d\xi\\
&\quad= \int_{\R^d}\bigg[ (u^\e \sqrt{f^\e} \otimes ((u^\e-\xi)\sqrt{f^\e} -2\sigma \nabla_\xi \sqrt{f^\e}) + \sigma u^\e \otimes \nabla_\xi f^\e \\
&\hspace{5cm} + ((u^\e-\xi)\sqrt{f^\e} -2\sigma \nabla_\xi \sqrt{f^\e})\otimes\xi\sqrt{f^\e} + \sigma\nabla_\xi f^\e \otimes \xi + \sigma f^\e\mathbb{I}_d\bigg] \,d\xi\\
&\quad= \int_{\R^d}\bigg[ (u^\e \sqrt{f^\e} \otimes ((u^\e-\xi)\sqrt{f^\e} -2\sigma \nabla_\xi \sqrt{f^\e}) + ((u^\e-\xi)\sqrt{f^\e} -2\sigma \nabla_\xi \sqrt{f^\e})\otimes\xi\sqrt{f^\e} \bigg] \,d\xi\\
&\quad \le 2\left(\int_{\R^d}|\xi|^2 f^\e\,d\xi\right)^{1/2} \left(\int_{\R^d} \frac{1}{f^\e}|\sigma \nabla_\xi f^\e - (u^\e-\xi)f^\e|^2\,d\xi\right),
\end{align*}
and this yields
\[
\mathcal{K}_4 \le C \left(\int_0^t \iint_{\T^d \times\R^d}|\xi|^2f^\e \,dxd\xi ds\right)^{1/2}\left(\int_0^t \mathcal{D}_1(f^\e)\,ds\right)^{1/2} \le C\sqrt{\e}.
\]
Here $C=C(T, \|\nabla u\|_{L^\infty}) > 0$ is a constant independent of $\e$ and we used Lemma \ref{L4.1}.\\

\noindent $\diamond$ (Estimates for $\mathcal{K}_5$): Since $-\Delta K\star(\rho^\e-\rho) = \rho^\e - \rho$ in the weak sense, by the integration by parts, we obtain
\begin{align*}
\mathcal{K}_5 &= \int_0^t \int_{\T^d} -\Delta K \star(\rho^\e - \rho) (v^\e-v)(u-v)\,dxds\\
&= \int_0^t \int_{\T^d} \nabla K\star(\rho^\e -\rho) \cdot \nabla ( (v^\e-v) (u-v))\,dxds\\
&\le C\int_0^t \int_{\T^d} |\nabla K \star (\rho^\e - \rho)|( |v^\e-v| + |\nabla (v^\e - v)|)\,dxds\\
&\le C\int_0^t \int_{\T^d} \me(U^\e | U)\,dxds + C\int_0^t \int_{\T^d} |\nabla K\star (\rho^\e-\rho)|^2\,dxds + \frac{1}{4} \int_0^t \int_{\T^d} |\nabla (v^\e-v)|^2\,dxds.
\end{align*}

\noindent $\diamond$ (Estimates for $\mathcal{K}_6$): A straightforward computation gives
\begin{align*}
\frac12\frac{d}{dt}\int_{\T^d} |\nabla K \star (\rho^\e - \rho)|^2\,dx &= \int_{\T^d} (\nabla K \star (\rho^\e - \rho))\cdot (\nabla K \star (\pa_t \rho^\e - \pa_t \rho))\,dx\cr
&=-\int_{\T^d} (\Delta K \star (\rho^\e - \rho)) \lt(K\star (\pa_t \rho^\e - \pa_t \rho)\rt)dx\cr
&= \int_{\T^d}   (\rho^\e - \rho) \lt(K\star (\pa_t \rho^\e - \pa_t \rho)\rt)dx.
\end{align*}
We then use the symmetry of $K$ to get
\begin{align*}
&\int_{\T^d}  (\rho^\e - \rho) \lt(K\star (\pa_t \rho^\e - \pa_t \rho)\rt)dx\cr
&\quad =-\iint_{\T^d \times \T^d} (\rho^\e - \rho)(x) K(x-y) \lt( \nabla_y \cdot (\rho^\e u^\e)(y) - \nabla_y \cdot (\rho u)(y) \rt) dxdy\cr
&\quad = \iint_{\T^d \times \T^d} (\rho^\e - \rho)(x) \nabla_y \lt( K(x-y)\rt) \cdot \lt( (\rho^\e u^\e)(y) - (\rho u)(y) \rt) dxdy\cr
&\quad =-\iint_{\T^d \times \T^d} (\rho^\e - \rho)(x) \nabla K(x-y) \cdot \lt( (\rho^\e u^\e)(y) - (\rho u)(y) \rt) dxdy\cr
&\quad =\iint_{\T^d \times \T^d} (\rho^\e - \rho)(y) \nabla K(x-y) \cdot \lt( (\rho^\e u^\e)(x) - (\rho u)(x) \rt) dxdy\cr
&\quad =\int_{\T^d} \nabla K\star(\rho^\e - \rho)  \cdot \lt(\rho^\e u^\e- \rho u \rt) dx.
\end{align*}
Then, we obtain
\begin{align*}
\mathcal{K}_6 &+ \int_0^t \int_{\T^d} \nabla K\star(\rho^\e - \rho)  \cdot \lt(\rho^\e u^\e - \rho u \rt) dx ds\\
&= \int_0^t \int_{\T^d} \nabla K\star(\rho^\e - \rho)\cdot u (\rho^\e - \rho) \,dxds\\
&=- \int_0^t \int_{\T^d} \nabla K \star(\rho^\e - \rho)  \cdot u\  \Delta K\star(\rho^\e - \rho) \,dxds\\
&= -\frac{1}{2}\int_0^t \int_{\T^d} |\nabla K\star(\rho^\e - \rho)|^2 \nabla \cdot u \,dxds + \int_0^t \int_{\T^d} \nabla K\star(\rho^\e - \rho)\otimes \nabla K \star(\rho^\e -\rho) : \nabla u\,dxds\\
&\le C\int_0^t\int_{\T^d} |\nabla K\star(\rho^\e - \rho)|^2\,dxds,
\end{align*}
Now, we combine all the previous estimates to yield
\begin{align*}
&\int_{\T^d} \me(U^\e |U)\,dx +\frac12 \int_{\T^d} |\nabla K\star(\rho^\e - \rho)|^2\,dx \\
&\quad +\int_0^t \int_{\T^d} |\nabla (v^\e-v)|^2\,dxds + \int_0^t \int_{\T^d}  \rho^\e |(u^\e - v^\e) - (u-v)|^2\,dxds\\
&\qquad \le C\sqrt{\e} + \int_{\T^d} \me(U_0^\e |U_0)\,dx + \frac12 \int_{\T^d} |\nabla K\star(\rho_0^\e - \rho_0)|^2\,dx + C\int_0^t \int_{\T^d} \me(U^\e|U)\,dxds \\
&\qquad \quad + C\int_0^t\int_{\T^d} |\nabla K\star (\rho^\e - \rho)|^2\,dxds,
\end{align*}
and thus, we use {\bf (H2)} and Gr\"onwall's lemma to get
\begin{align*}
&\int_{\T^d} \me(U^\e |U)\,dx + \int_{\T^d} |\nabla K\star(\rho^\e - \rho)|^2\,dx \\
&\quad +\int_0^t \int_{\T^d} |\nabla (v^\e-v)|^2\,dxds + \int_0^t \int_{\T^d}  \rho^\e |(u^\e - v^\e) - (u-v)|^2\,dxds\\
&\qquad \le C\sqrt{\e},
\end{align*}
where $C>0$ is independent of $\e$.
\end{proof}

\begin{remark}\label{rmk_comm} In the proof of Proposition \ref{P2.2}, the most delicate term is $\mathcal{K}_5$ which appears due to the coupling, drag force, between particles and fluid. In the case with diffusion, i.e., the modulated energy is convex with respect to $\rho$, by using the same argument as in the proof of \cite[Proposition 5.2]{CCK16}, we can also estimate $\mathcal{K}_5$ as
\[
\mathcal{K}_5 \leq C\int_0^t\int_{\T^d} \me(U^\e |U)\,dxds + \frac12 \int_0^t \int_{\T^d}  \rho^\e |(u^\e - v^\e) - (u-v)|^2\,dxds,
\]
where $C>0$ is independent of $\e$. Here the relative entropy is used to control the difference between $\rho^\e$ and $\rho$. This estimate also provides the same result as in Proposition \ref{P2.2}. However, this strategy only works in the presence of diffusion. As mentioned in Introduction, we find that the Coulomb interaction can be used to control $\rho^\e - \rho$, and this works regardless of the presence of diffusion.
\end{remark}

\begin{remark}\label{rmk_rel} Proposition \ref{P2.2} does not require the boundedness and periodicity of the domain. The estimates in Proposition \ref{P2.2} hold in the whole space as long as there exist the weak solutions $(f^\e, v^\e)$ to the system \eqref{A-2} and the strong solutions $(\rho,u,v)$ to the system \eqref{A-3} in the whole space with desired regularities, at least locally in time.
\end{remark}

\subsection{Proof of Theorem \ref{T2.2}} 
In this subsection, we provide the details of the proof of Theorem \ref{T2.2} showing the EPNS system  \eqref{A-3} can be well approximated by the VPNS system \eqref{A-2} for $\e>0$ small enough.

\subsubsection{Isothermal pressure case}
Let us first show the convergences in the isothermal pressure case. For this, we only show the following convergence:
\[
f^\e \to M_{\rho,u} := \frac{\rho}{(2\pi\sigma)^{d/2}}e^{-\frac{|u-\xi|^2}{2\sigma}} \quad \mbox{a.e.} \quad \mbox{and} \quad L^\infty(0,T;L^1(\T^d\times\R^d)),
\]
since the other convergences can be directly obtained by using the modulated energy estimated in Proposition \ref{P2.2}, see \cite[Corollary 2.1]{CCJpre} or \cite[Corollary 1.1]{CJpre}. For simplicity, we set $\sigma=1$ and consider 
\begin{align*}
\mh(f^\e | M_{\rho,u}) &= f^\e \log f^\e - M_{\rho,u} \log M_{\rho,u} + (f^\e - M_{\rho,u}) (1 + \log M_{\rho,u}) \cr
&= f^\e \lt(\log f^\e - \log M_{\rho,u}\rt) + (f^\e - M_{\rho,u}).
\end{align*}
Since 
\[
\log M_{\rho,u} = \log \rho  - \frac{|u-\xi|^2}{2} - \frac d2 \log (2\pi),
\]
we find
\[
\iint_{\T^d \times \R^d} \mh(f^\e | M_{\rho,u})\,dxd\xi = \iint_{\T^d \times \R^d} f^\e \log f^\e\,dxd\xi - \int_{\T^d} \rho^\e \log \rho\,dx + \frac12\iint_{\T^d \times \R^d} |u-\xi|^2f^\e \,dxd\xi + \frac d2 \log(2\pi).
\]
Note that 
\[
\|f^\e - M_{\rho,u}\|_{L^1}^2 \leq 4\iint_{\T^d \times \R^d} \mh(f^\e | M_{\rho,u})\,dxd\xi.
\]
In the proposition below, we estimate the above integral. Although its proof is similar to \cite{CCJpre} or \cite{CJpre}, for the completeness of our work we provide the details in Appendix \ref{app_conv}.
\begin{proposition}\label{prop_h1} Let $(f^\e,v^\e)$ be a global weak entropy solution to the system \eqref{A-2} and $(\rho,u,v)$ be a strong solution to the system \eqref{A-3} on the time interval $[0,T]$. Then we have
\begin{align*}
& \iint_{\T^d \times \R^d} \mh(f^\e | M_{\rho,u})\,dxd\xi +\frac12\int_{\T^d} |\nabla K\star (\rho^\e - \rho)|^2\,dx + \frac1{2\e}\int_0^t\iint_{\T^d \times \R^d} \frac{1}{f^\e}| \nabla_\xi f^\e - (u^\e - \xi )f^\e|^2\,dxd\xi  ds\cr
&\quad \leq  \iint_{\T^d \times \R^d} \mh(f_0^\e | M_{\rho_0,u_0})\,dxd\xi +\frac12\int_{\T^d} |\nabla  K \star(\rho_0^\e - \rho_0)|^2\,dx\cr
&\qquad + C\int_0^t\lt(\min\lt\{1,\int_{\T^d} \rho^\e |u - u^\e|^2\,dx\rt\}\rt)^{1/2} ds +  \frac{1}{\e^{1/4}}\int_0^t\int_{\T^d} |u^\e - u|^2 \rho^\e\,dxds \cr
&\qquad + C \int_0^t\int_{\T^d} |\nabla K \star (\rho^\e - \rho)|^2\,dxds+ C\e\int_0^t\iint_{\T^d \times \R^d} |\xi|^2 f^\e\,dxd\xi ds\cr
&\qquad + C\e \int_0^t\iint_{\T^d \times \R^d} |v^\e - \xi|^2 f^\e\,dxd\xi ds  + C\e^{1/4}\int_0^t\int_{\T^d} |v^\e - u^\e|^2 \rho^\e\,dxds,
\end{align*}
where $C>0$ is independent of $\e>0$.
\end{proposition}

As a direct consequence of Proposition \ref{prop_h1}, we obtain from the modulated energy estimates in Theorem \ref{T2.2} and the entropy estimate in Lemma \ref{L4.1} that
\[
 \iint_{\T^d \times \R^d} \mh(f^\e | M_{\rho,u})\,dxd\xi \leq  \iint_{\T^d \times \R^d} \mh(f_0^\e | M_{\rho_0,u_0})\,dxd\xi + C\e^{1/4}.
\]

\subsubsection{Pressureless case}
For the pressureless case, the convergence of $\rho^\e$ towards $\rho$ is not clear from Proposition \ref{P2.2}. In order to get the desired convergence of $\rho^\e$ towards $\rho$, we can use the previous result \cite[Lemma 4.1]{CCJpre}, see also \cite[Lemma 2.2]{CC20},  \cite[Proposition 3.1]{Cpre}, and \cite[Lemma 5.2]{FK19}, which asserts that the bounded Lipschitz distance $d_{BL}$ between local densities can be bounded from above by the modulated kinetic energy. Let $\mu, \nu \in \mathcal{M}(\T^d)$ be two Radon measures, then the bounded Lipschitz distance, which is denoted by $d_{BL}: \mathcal{M}(\T^d) \times \mathcal{M}(\T^d) \to \R_+$, between $\mu$ and $\nu$ is defined by
\[
d_{BL}(\mu,\nu) := \sup_{\phi \in \Omega} \lt|\int_{\T^d} \phi(x)( \mu(dx) - \nu(dx))\rt|,
\]
where the admissible set $\Omega$ of test functions are given by
\[
\Omega:= \lt\{\phi: \T^d \to \R: \|\phi\|_{L^\infty} \leq 1, \ Lip(\phi) := \sup_{x \neq y} \frac{|\phi(x) - \phi(y)|}{|x-y|} \leq 1 \rt\}.
\]
Then by \cite[Lemma 4.1]{CCJpre} we have the following estimate which gives the convergence $\rho^\e \rightharpoonup \rho$  weakly in $L^\infty(0,T;\mathcal{M}(\T^d))$.
\begin{lemma}\label{lem_dbl} Let $(f^\e,v^\e)$ be a global weak entropy solution to the system \eqref{A-2} and $(\rho,u,v)$ be a strong solution to the system \eqref{A-3} on the time interval $[0,T]$. Then we have
\[
d_{BL}(\rho(t), \rho^\e(t)) \leq Cd_{BL}(\rho_0, \rho_0^\e) + C\lt(\int_0^t  \int_{\T^d} \rho^\e|u^\e- u|^2\,dx ds\rt)^{1/2}
\]
for $0 \leq t \leq T$, where $C > 0$ is independent of $\e>0$.
\end{lemma}
This together with Proposition \ref{P2.2} asserts
\[
d_{BL}(\rho(t), \rho^\e(t)) \leq Cd_{BL}(\rho_0, \rho_0^\e) + C\e^{1/4}.
\]
For the other convergence estimates, we refer to \cite[Corollary 2.3]{CCJpre}. This completes the proof of Theorem \ref{T2.2}. 

\begin{remark}\label{rmk_whole} Lemma \ref{lem_dbl} holds in the whole space. This together with Remark \ref{rmk_rel} implies that the estimates of hydrodynamic limit in Theorem \ref{T2.2} also hold when we add the assumption that $d_{BL}(\rho_0, \rho_0^\e) = \mathcal{O}(\e^{1/4})$ to Theorem \ref{T2.2}.
On the other hand, in the spatial periodic domain $\T^d$, we can bound the bounded Lipschitz between $\rho^\e$ and $\rho$ by $H^{-1}(\T^d)$-norm between them. Indeed, we find 
\[
\int_{\T^d} (\rho^\e-\rho)\phi \,dx \le \|\rho^\e - \rho\|_{H^{-1}}\|\phi\|_{H^1} \le \|\rho^\e - \rho\|_{H^{-1}}\|\phi\|_{W^{1,\infty}} \leq \|\rho^\e - \rho\|_{H^{-1}}
\]
for $\phi \in W^{1,\infty}(\T^d)$ with $\|\phi\|_{W^{1,\infty}} \le 1$. This implies that we do not need to add additional assumption on the initial densities $\rho^\e_0$ and $\rho_0$ to get the convergence estimates in Theorem \ref{T2.2}. 
\end{remark}
%
%
%
%
\section{Global-in-time strong solvability for the isothermal/pressureless EPNS system}\label{sec:strong_ext}
 
In this section, we study the global-in-time existence and uniqueness of strong solutions to the system \eqref{A-3}. 

\subsection{Local solvability}
We discuss the local-in-time strong solvability for the isothermal and pressureless cases separately.

\subsubsection{Isothermal EPNS system}
Consider the following isothermal Euler--Poisson system coupled with Navier--Stokes system:
\begin{align}
\begin{aligned}\label{E-1}
&\pa_t \rho + \nabla \cdot (\rho u) = 0, \quad (x,t) \in \T^d \times \R_+,\cr
&\pa_t (\rho u) + \nabla \cdot (\rho u \otimes u) + \nabla \rho =  -\rho(u-v + \nabla K \star (\rho-1)) ,\cr
&\pa_t v + (v \cdot \nabla) v + \nabla p -\Delta v = \rho(u-v),\cr
&\nabla \cdot v =0.
\end{aligned}
\end{align}
Here we set $\sigma=1$ without loss of generality. Then, we reformulate the system \eqref{E-1} by letting $g := \log \rho$ as follows:
\begin{align}\label{E-2}
\begin{aligned}
&\pa_t g + \nabla g \cdot  u + \nabla \cdot u = 0, \quad (x,t) \in \T^d \times \R_+,\cr
&\pa_t u +  (u \cdot \nabla ) u + \nabla g = - (u-v+ \nabla K\star (e^g-1)),\\
&\pa_t v + (v \cdot \nabla) v + \nabla p -\Delta v = e^g(u-v),\cr
&\nabla \cdot v =0
\end{aligned}
\end{align}
subject to initial data:
\bq\label{ini_F-1}
(g(x,0), u(x,0),v(x,0)) =: (g_0(x), u_0(x), v_0(x)), \quad x \in \T^d.
\eq

Now, we state the result on the local well-posedness of the system \eqref{E-2}.

\begin{theorem}\label{L5.1} Let $d\geq 2$ and $s > d/2+1$. Suppose that the initial data  satisfies
\[
(g_0, u_0,v_0) \in H^s(\T^d) \times H^s(\T^d) \times H^s(\T^d) \quad \mbox{and}\quad e^{g_0} > 0.
\] 
Then for any positive constants $\epsilon_0 < M_0$, there exists a positive constant $T^*$ such that if $$\|g_0\|_{H^s}^2 + \|u_0\|_{H^s}^2 +\|v_0\|_{H^s}^2 < \epsilon_0,$$ then the system \eqref{E-2}-\eqref{ini_F-1} admits a unique solution 
\[
(g,u,v) \in \mc([0,T^*]; H^s(\T^d)) \times \mc([0,T^*]; H^s(\T^d)) \times \mc([0,T^*];H^s(\T^d))
\] 
satisfying
\[
\sup_{0 \leq t \leq T^*} \lt(\|g(\cdot,t)\|_{H^s}^2 + \|u(\cdot,t)\|_{H^s}^2 +\|v(\cdot,t)\|_{H^s}^2 \rt) \leq M_0.
\]
\end{theorem}
\begin{proof}
Since the proof is rather lengthy and technical, we leave it in Appendix \ref{app_local}.
\end{proof}

\subsubsection{Pressureless EPNS system}\label{sec:npE}

For the pressureless case $\sigma=0$, we set $h := \rho - 1$ and reformulate the system \eqref{E-1} as 
\begin{align}\label{E-4}
\begin{aligned}
&\partial_t h + \nabla\cdot ((1+ h) u) = 0, \quad (x,t) \in \T^d \times \R_+,\\
&\partial_t u + (u \cdot \nabla) u = - (u-v+ \nabla K \star h),\\
&\pa_t v + (v \cdot \nabla) v +\nabla p -\Delta v = (1+h)(u-v),\\
&\nabla \cdot v = 0.
\end{aligned}
\end{align}
For the local-in-time strong solvability of the above system, we use a similar strategy to that for Theorem \ref{L5.1}, see Appendix \ref{app_local}; we construct a sequence of approximate solutions to the reformulated system:
\begin{align*}
&\partial_t h^{n+1} + \nabla\cdot ((1+ h^{n+1}) u^n) = 0, \quad (x,t) \in \T^d \times \R_+,\\
&\partial_t u^{n+1} + (u^n \cdot \nabla) u^{n+1} = - (u^{n+1}-v^n+ \nabla K \star h^n),\\
&\pa_t v^{n+1} + (v^n \cdot \nabla) v^{n+1} +\nabla p -\Delta v^{n+1} = (1+h^n)(u^n-v^{n+1}),\\
&\nabla \cdot v^{n+1} = 0,
\end{align*}
In this case, we use the following estimate for the Coulomb interaction term:
\begin{align*}
\int_{\T^d} \nabla (\nabla K \star h^n) : \nabla u^{n+1} \, dx & = \sum_{i,j=1}^d  \int_{\T^d} \partial_{x_j} (\partial_{x_i}(K\star h^n)) \partial_{x_j} u_i^{n+1}\,dx\\
&= \sum_{i,j=1}^d \int_{\T^d} \partial_{x_j}\partial_{x_j} (K\star h^n)) \partial_{x_i} u_i^{n+1} \,dx\\
&=  \int_{\T^d} \Delta K \star h^n \nabla \cdot u^{n+1} \,dx\cr
& = -\int_{\T^d} h^n \nabla \cdot u^{n+1}\,dx.
\end{align*}
Thus, the above observation enables us to have $H^{s+1}$-estimates for $u$ and thus for $v$, i.e., for any $M>N$, if
\[
\|h_0\|_{H^s}^2 + \|u_0\|_{H^{s+1}}^2 + \|v_0\|_{H^{s+1}}^2 < N,
\]
then there exists $T^*>0$ such that
\[
\sup_{0 \le t \le T^*}\left( \|h^n(\cdot,t)\|_{H^s}^2 + \|u^n(\cdot,t)\|_{H^{s+1}} + \|v^n(\cdot,t)\|_{H^{s+1}}^2 \right) < M \quad \forall \, n \in \bbn.
\]
Thus, we can employ the similar argument to Theorem \ref{L5.1} to obtain the local-in-time well-posedness of strong solutions to the system \eqref{E-1} with $\sigma=0$.

\begin{theorem}\label{L5.2} 
Let $d\geq 2$ and $s > d/2+1$. Suppose that the initial data  satisfies
\[
(h_0, u_0,v_0) \in H^s(\T^d) \times H^{s+1}(\T^d) \times H^{s+1}(\T^d) \quad \mbox{and}\quad 1+h_0 > 0.
\] 
Then for any positive constants $\epsilon_0 < M_0$, there exists a positive constant $T^*$ such that if 
\[
\|h_0\|_{H^s}^2 + \|u_0\|_{H^{s+1}}^2 +\|v_0\|_{H^{s+1}}^2 < \epsilon_0,
\]
 then the system \eqref{E-4} admits a unique solution 
 \[
 (h,u,v) \in \mc([0,T^*]; H^s(\T^d)) \times \mc([0,T^*]; H^{s+1}(\T^d)) \times \mc([0,T^*];H^{s+1}(\T^d)) 
 \]
 satisfying
\[
\sup_{0 \leq t \leq T^*} \lt(\|h(\cdot,t)\|_{H^s}^2 + \|u(\cdot,t)\|_{H^{s+1}}^2 +\|v(\cdot,t)\|_{H^{s+1}}^2 \rt) \leq M_0.
\]
\end{theorem}
%
%
%
%
%
%

\subsection{Proof of Theorem \ref{T2.3}}\label{app.B}
 
In this subsection, we take further steps to obtain the global-in-time existence of strong solutions to the system \eqref{E-2}. First, we define
\[
\mathfrak{X}(T;k) := \sup_{0\le t \le T} \left( \|g(\cdot, t)\|_{H^k}^2 + \|u(\cdot, t)\|_{H^k}^2 + \|v(\cdot,t)\|_{H^k}^2\right) \quad \mbox{and} \quad \mathfrak{X}_0(k) = \|g_0\|_{H^k}^2 + \|u_0\|_{H^k}^2 + \|v_0\|_{H^k}^2.
\]
We then temporarily move back to the original system \eqref{E-1} and provide the energy estimate.

\begin{proposition}\label{P5.1}
Let $T>0$, and suppose that $(\rho, u,v)$ is a strong solution to \eqref{E-1} on the time interval $[0,T]$ corresponding to the initial data $(\rho_0, u_0, v_0)$. Then, we have
\[
\begin{split}
\frac12\frac{d}{dt}&\left(\int_{\T^d} \rho |u|^2\,dx + \int_{\T^d}|\nabla K\star(\rho-1)|^2\,dx +\int_{\T^d}|v|^2\,dx + 2\sigma \int_{\T^d}\rho \log\rho\,dx\right) \\
&\quad + \int_{\T^d}|\nabla v|^2\,dx + \int_{\T^d} \rho|u-v|^2\,dx=0.
\end{split}
\]
\end{proposition}
\begin{proof}
We first easily find
\[
\frac12\frac{d}{dt}\int_{\T^d}\rho |u|^2\,dx = -\sigma\int_{\T^d} \nabla \rho \cdot u\,dx -\int_{\T^d}\rho u \cdot \nabla K \star (\rho-1)\,dx -\int_{\T^d}\rho(u-v)\cdot u\,dx.
\]
On the other hand, the first two terms on the right hand side of the above can be estimated as
\[
\frac12\frac{d}{dt}\int_{\T^d} |\nabla K \star (\rho-1)|^2\,dx = \int_{\T^d} \rho u \cdot \nabla K \star (\rho-1)\,dx
\]
and
\[
\frac{d}{dt} \int_{\T^d} \rho \log \rho\,dx = \int_{\T^d}u \cdot \nabla \rho\,dx.
\]
For the incompressible Navier--Stokes equations, we get
\[
\frac12\frac{d}{dt}\int_{\T^d} v^2\,dx = -\int_{\T^d} |\nabla v|^2\,dx - \int_{\T^d}\rho(v-u)\cdot v\,dx.
\]
Finally, we combine all the above estimates to have the desired result.
\end{proof}

Note that
\[
\int_{\T^d}\rho \log\rho \,dx = \int_{\T^d}( \rho\log\rho + 1 -\rho)\,dx = \int_{\T^d} ((g-1)e^g + 1)\,dx,
\]
and we use the second-order Taylor approximation for $f(x) = (x-1) e^x$ at $x=0$ to get
\[
(1-\|g\|_{L^\infty})e^{-\|g\|_{L^\infty}}\int_{\T^d} |g|^2\,dx \le \int_{\T^d}\rho \log\rho \,dx \le (1+\|g\|_{L^\infty})e^{\|g\|_{L^\infty}}\int_{\T^d} |g|^2\,dx
\]
If we choose a sufficiently small $\e_1>0$ satisfying
\[
\mathfrak{X}(T;s)\le \e_1^2 \ll 1 \quad \mbox{so that}\quad \sup_{0 \le t \le T} \|g(\cdot, t)\|_{L^\infty} \le \log 2,
\]
then by Proposition \ref{P5.1} with $\sigma=1$ we have
\[
\begin{split}
\mathfrak{X}(T;0) &\le C\left(\int_{\T^d} \rho|u|^2\,dx + \int_{\T^d} |\nabla K\star(\rho-1)|^2\,dx + \int_{\T^d} |v|^2\,dx +\int_{\T^d}\rho\log\rho\,dx\right)\\
&\le C\left(\int_{\T^d} \rho_0|u_0|^2\,dx + \int_{\T^d} |\nabla K\star(\rho_0-1)|^2\,dx + \int_{\T^d} |v_0|^2\,dx +\int_{\T^d}\rho_0\log\rho_0\,dx\right)\\
&\le C\mathfrak{X}_0(0),
\end{split}
\]
where $C$ is independent of $T$ and we used
\[
\|\nabla K\star(\rho-1)\|_{L^2} \leq \|\nabla K\|_{L^1}\|\rho - 1\|_{L^2} = \|\nabla K\|_{L^1}\|e^g-e^0\|_{L^2} \le \|\nabla K\|_{L^1}e^{\|g\|_{L^\infty}}\|g\|_{L^2}.
\]
In the sequel, we provide higher-order derivative estimates, and for this, the following Moser-type inequalities will be significantly used.
\begin{lemma}\label{lem_moser} For any pair of functions $f,g \in (H^k \cap L^\infty)(\T^d)$, we obtain
\[
\|\nabla^k (fg)\|_{L^2} \le C\lt(\|f\|_{L^\infty} \|\nabla^k g\|_{L^2} + \|\nabla^k f\|_{L^2}\|g\|_{L^\infty}\rt).
\]
Furthermore, if $\nabla f \in L^\infty(\T^d)$, we have
\[
\|\nabla^k(fg) - f\nabla^k g\|_{L^2} \le C\lt(\|\nabla f\|_{L^\infty}\|\nabla^{k-1} g\|_{L^2} + \|g\|_{L^\infty}\|\nabla^k f\|_{L^2}\rt).
\]
Here $C>0$ only depends on $k$ and $d$.
\end{lemma}
We begin by estimating the higher-order derivates of $(g,u)$.
\begin{lemma}\label{LB.1}
Let $s>d/2+1$, $T>0$ be given and suppose that $\mathfrak{X}(T;s) \le \e_1^2 \ll 1$. Then we have
\begin{align*}
\frac{d}{dt}&\lt(\|\nabla^k (\nabla g)\|_{L^2}^2 + \|\nabla^k (\nabla u)\|_{L^2}^2\rt) + \frac32\|\nabla^k(\nabla u)\|_{L^2}^2\\
&\le C\e_1 \lt(\|\nabla^k (\nabla g)\|_{L^2}^2 + \|\nabla^k (\nabla u)\|_{L^2}^2\rt) + C\|g\|_{H^k}^2 + 4\|\nabla^k (\nabla v)\|_{L^2}^2
\end{align*}
for $0 \le k \le s-1$, where $C$ is a positive constant independent of $T$.
\end{lemma}
\begin{proof}
We separately estimate the case $k=0$ and $k>0$ as follows:\\

\noindent $\bullet$ (Step A: First-order estimate) For $\nabla g$, we have
\begin{align*}
\frac12\frac{d}{dt}\|\nabla g\|_{L^2}^2 &= -\int_{\T^d} \nabla(\nabla g \cdot u)\cdot \nabla g\,dx -\int_{\T^d} \nabla g \cdot \nabla(\nabla \cdot u)\,dx\\
&= \frac12 \int_{\T^d} (\nabla \cdot u)|\nabla g|^2\,dx -\int_{\T^d} (\nabla g \cdot \nabla)u \cdot \nabla g\,dx -\int_{\T^d} \nabla g \cdot \nabla (\nabla \cdot u)\,dx\\
&\le \frac32\|\nabla u\|_{L^\infty}\|\nabla g\|_{L^2}^2 -\int_{\T^d} \nabla g \cdot \nabla (\nabla \cdot u)\,dx\\
&\le C\e_1\|\nabla g\|_{L^2}^2 -\int_{\T^d} \nabla g \cdot \nabla (\nabla \cdot u)\,dx
\end{align*}
Then, for $\nabla u$, we obtain
\begin{align*}
\frac12\frac{d}{dt}\|\nabla u\|_{L^2}^2 &= -\int_{\T^d} \nabla(u \cdot \nabla u) :\nabla u\,dx -\int_{\T^d} \nabla^2 g : \nabla u\,dx \\
&\quad -\int_{\T^d} \nabla (\nabla K\star(e^g -1)):\nabla u\,dx -\int_{\T^d} \nabla(u-v):\nabla u\,dx\\
&\le C\|\nabla u\|_{L^\infty}\|\nabla u\|_{L^2}^2 -\int_{\T^d} \nabla^2 g :\nabla u\,dx + \int_{\T^d} (e^g-1)(\nabla\cdot u)\,dx -\frac78 \|\nabla u\|_{L^2}^2 + 2\|\nabla v\|_{L^2}^2\\
&\le C\e_1 \|\nabla u\|_{L^2}^2 -\int_{\T^d} \nabla^2 g :\nabla u\,dx + C\|g\|_{L^2}^2 -\frac34\|\nabla u\|_{L^2} + 2\|\nabla v\|_{L^2}^2,
\end{align*}
where $C$ depends on $k$ and $d$ and we used the smallness of $\e_1$ to get
\[
\|e^g-1\|_{L^2} \le e^{\|g\|_{L^\infty}} \|g\|_{L^2} \le e^{C\e_1} \|g\|_{L^2} \le C\|g\|_{L^2}.
\]
So we combine two estimates to yield the desired result when $k=0$.\\

\noindent $\bullet$ (Step B: Higher-order estimate) For $1 \le k \le s-1$, we estimate $\nabla^k (\nabla g)$ as
\begin{align*}
&\frac12\frac{d}{dt}\|\nabla^k (\nabla g)\|_{L^2}^2 \cr
&\quad = -\int_{\T^d} \nabla^k (u \cdot \nabla (\nabla g)) \cdot \nabla^k (\nabla g)\,dx -\int_{\T^d} \Big[ \nabla^k (\nabla (\nabla g \cdot u)) - u \cdot \nabla (\nabla^k(\nabla g)) \Big] \nabla^k(\nabla g)\,dx\\
&\qquad -\int_{\T^d} \nabla^k (\nabla(\nabla \cdot u)) \cdot \nabla^k (\nabla g)\,dx\\
&\quad \le \frac{\|\nabla \cdot u\|_{L^\infty}}{2}\|\nabla^k (\nabla g)\|_{L^2}^2 + C\|\nabla^k (\nabla g)\|_{L^2}\Big(\|\nabla u\|_{L^\infty} \|\nabla^k (\nabla g)\|_{L^2} + \|\nabla g\|_{L^\infty}\|\nabla^k (\nabla u)\|_{L^2}\Big)\\
&\qquad -\int_{\T^d} \nabla^k (\nabla(\nabla \cdot u)) \cdot \nabla^k (\nabla g)\,dx\\
&\quad \le C\e_1\left(\|\nabla^k (\nabla g)\|_{L^2}^2 + \|\nabla^k (\nabla u)\|_{L^2}^2\right) -\int_{\T^d} \nabla^k (\nabla(\nabla \cdot u)) \cdot \nabla^k (\nabla g)\,dx,
\end{align*}
where $C$ only depends on $k$ and $d$. For $\nabla^k(\nabla u)$, we get
\begin{align*}
&\frac12\frac{d}{dt}\|\nabla^k (\nabla u)\|_{L^2}^2 \cr
&\quad = -\int_{\T^d}u \cdot \nabla (\nabla^k (\nabla u)) : \nabla^k (\nabla u)\,dx -\int_{\T^d} \Big[ \nabla^k (\nabla (u \cdot \nabla u)) -  u \cdot \nabla (\nabla^k (\nabla u))\Big] : \nabla^k (\nabla u)\,dx\\
&\qquad -\int_{\T^d} \nabla^k ((\nabla u)^2) :\nabla^k (\nabla u)\,dx -\int_{\T^d} \nabla^k (\nabla^2 g) :\nabla^k (\nabla u)\,dx\\
&\qquad -\int_{\T^d} \nabla^k (\nabla^2 K\star(e^g-1)) :\nabla^k (\nabla u)\,dx -\int_{\T^d} \nabla^k (\nabla(u-v)):\nabla^k (\nabla u)\,dx\\
&\quad \le C\|\nabla u\|_{L^\infty}\|\nabla^k (\nabla u)\|_{L^2}^2 -\int_{\T^d} \nabla^k (\nabla^2 g) :\nabla^k (\nabla u)\,dx\\
&\qquad +\int_{\T^d} \nabla^k(e^g-1) \nabla^k (\nabla \cdot u)\,dx -\frac78 \|\nabla^k (\nabla u)\|_{L^2}^2 +2\|\nabla^k (\nabla v)\|_{L^2}^2\\
&\quad \le C\e_1\|\nabla^k(\nabla u)\|_{L^2}^2  -\int_{\T^d} \nabla^k (\nabla^2 g) :\nabla^k (\nabla u)\,dx +2\|\nabla^k (e^g-1)\|_{L^2}^2 -\frac34\|\nabla^k (\nabla u)\|_{L^2}^2 +2\|\nabla^k (\nabla v)\|_{L^2}^2.
\end{align*}
Here, for the estimate of $\|\nabla^k (e^g)\|_{L^2}$, we follow the same argument in \cite{CCJpre} as follows: let $a_k :=\|\nabla^k (e^g-1)\|_{L^2}$. First, one obtains
\[
a_0 \le e^{\|g\|_{L^\infty}}\|g\|_{L^2}\le Ce^{C\e_1}\|g\|_{L^2}.
\]
Then, we use Lemma \ref{lem_moser} to have
\[
\begin{aligned}
a_k&\le \|\nabla^{k-1}(e^g \nabla g)\|_{L^2}\\
&\le C(e^{\|g\|_{L^\infty}}\|\nabla^k g\|_{L^2} + a_{k-1}\|\nabla g\|_{L^\infty})\\
&\le C(e^{C\e_1}\|\nabla^k g\|_{L^2} + \e_1 a_{k-1})\\
&\le C(e^{C\e_1}\|\nabla^k g\|_{L^2} + C\e_1(e^{C\e_1}\|\nabla^{k-1} g\|_{L^2} + \e_1 a_{k-2}))\\
&\le C e^{C\e_1}\sum_{\ell=0}^k \e_1^{\ell} \|\nabla^{k-\ell} g\|_{L^2},
\end{aligned}
\]
where $C$ only depends on $k$ and $d$. Since $\e_1$ is sufficiently small, we obtain
\[
a_k \le C\|g\|_{H^k}.
\]
Combining all of the above estimates, we complete the proof.
\end{proof}

It is worth noticing that Lemma \ref{LB.1} does not provide the dissipation rate for $\nabla^k (\nabla g)$. In order to have it, inspired by \cite{Choi16}, we estimate the mixing term in the lemma below.
\begin{lemma}\label{LB.2}
Let $d\geq2$, $s>d/2+1$, and $T>0$ be given. Suppose that $\mathfrak{X}(T;s) \le \e_1^2 \ll 1$. Then we have
\begin{align*}
\frac{d}{dt}&\int_{\T^d} \nabla^k (\nabla g) \cdot \nabla^k u\,dx + \frac12\|\nabla^k(\nabla g)\|_{L^2}^2\\
&\le C\e_1 \lt(\|\nabla^k(\nabla g)\|_{L^2}^2 + \|\nabla^k(\nabla u)\|_{L^2}^2\rt) + C \lt(\| g\|_{H^k}^2 + \|\nabla^k u\|_{L^2}^2\rt) +\|\nabla^k(\nabla u)\|_{L^2}^2
\end{align*}
for $0 \le k \le s-1$, where $C$ is a positive constant independent of $T$.
\end{lemma}
\begin{proof}
Direct computations and applying Lemma \ref{lem_moser} yield
\begin{align*}
\frac{d}{dt}&\int_{\T^d} \nabla^k (\nabla g) \cdot \nabla^k u\,dx\\
&=-\frac{d}{dt}\int_{\T^d} (\nabla^k g) \ \nabla^k(\nabla \cdot u)\,dx\\
&= \int_{\T^d} \nabla^k (\nabla g \cdot u + \nabla \cdot u)\nabla^k (\nabla \cdot u)\,dx +\int_{\T^d} \nabla^k \lt(\nabla \cdot (u\cdot \nabla u) +\Delta g - (e^g-1) + \nabla \cdot u\rt)(\nabla^k g)\,dx\\
&= \int_{\T^d} (\nabla^k (\nabla g)\cdot u) \nabla^k (\nabla \cdot u)\,dx +\int_{\T^d}\left[ \nabla^k (\nabla g \cdot u) - \nabla^k (\nabla g) \cdot u\rt] \nabla^k (\nabla \cdot u)\,dx \\
&\quad + \|\nabla^k(\nabla \cdot u)\|_{L^2}^2 -\int_{\T^d} (u \cdot \nabla^k (\nabla u))\cdot \nabla^k(\nabla g)dx-\int_{\T^d}\lt[ \nabla^k (u \cdot \nabla u) - u \cdot \nabla^k (\nabla u)\rt] \cdot \nabla^k (\nabla g)\,dx\\
&\quad  -\|\nabla^k (\nabla g)\|_{L^2}^2 +\int_{\T^d} \nabla^k (e^g-1)\nabla^k g\,dx -\int_{\T^d} \nabla^k  u \cdot \nabla^k (\nabla g)\,dx\\
&\le 2\|u\|_{L^\infty}\|\nabla^k(\nabla g)\|_{L^2}\|\nabla^k (\nabla u)\|_{L^2}+ C\|\nabla^k(\nabla u)\|_{L^2}\lt(\|\nabla u\|_{L^\infty}\|\nabla^k g\|_{L^2} + \|\nabla g\|_{L^\infty}\|\nabla^k u\|_{L^2}\rt)\\
&\quad+\|\nabla^k(\nabla \cdot u)\|_{L^2}^2 + C\|\nabla u\|_{L^\infty}\|\nabla^k(\nabla g)\|_{L^2}\|\nabla^k u\|_{L^2} -\frac12\|\nabla^k(\nabla g)\|_{L^2}^2 +\frac12\|\nabla^k u\|_{L^2}^2+ C\|g\|_{H^k}^2  \\
&\le C\e_1 \lt(\|\nabla^k(\nabla g)\|_{L^2}^2 + \|\nabla^k(\nabla u)\|_{L^2}^2\rt) + C \lt(\|g\|_{H^k}^2 + \|\nabla^k u\|_{L^2}^2\rt)+\|\nabla^k(\nabla u)\|_{L^2}^2 -\frac12\|\nabla^k(\nabla g)\|_{L^2}^2.
\end{align*}
This completes the proof.
\end{proof}

Now, we provide the higher-order estimates for the incompressible Navier--Stokes equations.
\begin{lemma}\label{LB.3}
Let $d\geq2$, $s>d/2+1$, and $T>0$ be given. Suppose that $\mathfrak{X}(T;s) \le \e_1^2 \ll 1$ so that
\[
\sup_{0\le t \le T} \|g(\cdot, t)\|_{L^\infty} \le \log 2.
\]
 Then we have
\begin{align*}
\frac{d}{dt}&\|\nabla^k v\|_{L^2}^2 +\frac12\|\nabla^k v\|_{L^2}^2+\|\nabla^k (\nabla v)\|_{L^2}^2\\
&\le C\e_1 \|\nabla^k v\|_{L^2}^2 +C\lt(\|\nabla^k u\|_{L^2}^2+(\|\nabla^{k-1} v\|_{L^2} + \|\nabla^{k-1} u\|_{L^2}^2+\|g\|_{H^{k-1}}^2)(1-\delta_{k,0})\rt)
\end{align*}
for $0 \le k \le s-1$, where $C$ is a positive constant independent of $T$. Here $\delta_{k,0}$ denotes the Kronecker delta, i.e., the terms with $(k-1)$-th order do not appear when $k=0$.
\end{lemma}
\begin{proof}
It follows from \eqref{E-4} that 
\[\begin{aligned}
\frac12\frac{d}{dt}\|v\|_{L^2} + \|\nabla v\|_{L^2}^2&= -\int_{\T^d} (v \cdot \nabla v) \cdot v\,dx \int_{\T^d} e^g(v-u)\cdot v\,dx\\
&\le -\frac12\int_{\T^d}e^g |v|^2\,dx + \frac12\int_{\T^d}e^g |u|^2\,dx\\
&\le -\frac14\|v\|_{L^2}^2 + \|u\|_{L^2}^2.
\end{aligned}
\]
For $k \ge 1$, we estimate
\begin{align*}
\frac12&\frac{d}{dt}\|\nabla^k v\|_{L^2}^2 +\|\nabla^k (\nabla v)\|_{L^2}^2\\
&= -\int_{\T^d} \nabla^k (v \cdot \nabla v)\cdot \nabla^k v\,dx - \int_{\T^d}\nabla^k (e^g(v-u))\cdot\nabla^k v\,dx\\
&= -\int_{\T^d}\lt[ \nabla^k(v \cdot \nabla v) -v \cdot \nabla^k (\nabla v)\rt] \cdot \nabla^k v\,dx\\
&\quad -\int_{\T^d} e^g \nabla^k (v-u)\cdot \nabla^k v\,dx -\int_{\T^d} \lt[\nabla^k (e^g(v-u)) - e^g \nabla^k (v-u)\rt] \cdot \nabla^k v\,dx\\
&\le C\|\nabla v\|_{L^\infty}\|\nabla^k v\|_{L^2}^2 - e^{-\|g\|_{L^\infty}}\|\nabla^k v\|_{L^2}^2 + e^{\|g\|_{L^\infty}} \|\nabla^k u\|_{L^2}\|\nabla^k v\|_{L^2}\\
&\quad +\int_{\T^d} \left[ \nabla^{k-1} (e^g(v-u)) - e^g \nabla^{k-1}(v-u)\rt]\cdot\nabla^{k+1}v\,dx -\int_{\T^d} e^g \nabla g (\nabla^{k-1} (v-u))\cdot \nabla^{k+1}v\,dx\\
&\le C\e_1 \|\nabla^k v\|_{L^2}^2 -\frac14\|\nabla^k v\|_{L^2}^2 + 4\|\nabla^k u\|_{L^2}^2\\
&\quad + C\|\nabla^{k+1} v\|_{L^2}(e^{\|g\|_{L^\infty}} \|\nabla^{k-1}(v-u)\|_{L^2} + \|\nabla^{k-1}(e^g)\|_{L^2}\|v-u\|_{L^\infty})\\
&\quad +e^{\|g\|_{L^\infty}}\|\nabla g\|_{L^\infty}\|\nabla^{k-1}(v-u)\|_{L^2}\|\nabla^{k+1}v\|_{L^2}\\
&\le C\e_1 \|\nabla^k v\|_{L^2}^2 -\frac14\|\nabla^k v\|_{L^2}^2 + 4\|\nabla^k u\|_{L^2}^2 +\frac12\|\nabla^k(\nabla v)\|_{L^2}^2 \cr
&\quad +C\lt(\|\nabla^{k-1} v\|_{L^2} + \|\nabla^{k-1} u\|_{L^2}^2+\|g\|_{H^{k-1}}^2\rt),
\end{align*}
where $C$ is a positive constant independent of $T$. This concludes the desired result.
\end{proof}

Now, we can investigate the uniform energy estimates using Lemmas \ref{LB.1}--\ref{LB.3}.
\begin{corollary}\label{CB.1}
Let $s>d/2+1$, $T>0$ be given and suppose that $\mathfrak{X}(T;s)\le \e_1^2 \ll 1$ so that
\[
\sup_{0\le t \le T} \|g(\cdot, t)\|_{L^\infty} \le \log 2.
\]
 Then we have
\[
\mathfrak{X}(T;s) \le C_1 \mathfrak{X}_0(s), 
\]
where $C_1$ is independent of $T$.
\end{corollary} 
\begin{proof}
The proof is based on the inductive argument. Since it suffices to show the induction step, we suppose that there exists a positive constant $C>0$ that is independent of $T$ such that
\[
\mathfrak{X}(T;m) \le C\mathfrak{X}_0(m) \quad \mbox{for} \quad 0 \le m \le k.
\]
First, Lemmas \ref{LB.1} and \ref{LB.2} deduce
\[
\begin{aligned}
\frac{d}{dt}&\left(\|\nabla^k(\nabla g)\|_{L^2}^2 + \|\nabla^k(\nabla u)\|_{L^2}^2 +\int_{\T^d}\nabla^k (\nabla g) \cdot \nabla^k  u\,dx\right)\\
&\le \left(-\frac12 + C\e_1\right) (\|\nabla^k (\nabla g)\|_{L^2}^2 + \|\nabla^k(\nabla u)\|_{L^2}^2) +C\mathfrak{X}(T;k) + 4\|\nabla^k (\nabla v)\|_{L^2}^2,
\end{aligned}
\]
where $C$ is independent of $T$. Moreover, the estimates in Lemma \ref{LB.1} imply
\[
\begin{aligned}
\frac12\frac{d}{dt}\|\nabla^k u\|_{L^2}^2 &\le C\e_1\|\nabla^k u\|_{L^2}^2  -\int_{\T^d} \nabla^k (\nabla g) \cdot \nabla^k u\,dx\\
&\quad +2\|\nabla K \star \nabla^k (e^g-1)\|_{L^2}^2 -\frac34\|\nabla^k u\|_{L^2}^2 +2\|\nabla^k  v\|_{L^2}^2\\
&\le C\e_1 \|\nabla^k u\|_{L^2}^2 -\frac{1}{12}\|\nabla^k u\|_{L^2}^2 + \frac38\|\nabla^k (\nabla g)\|_{L^2}^2 + C\mathfrak{X}(T;k).
\end{aligned}
\]
Here $C$ is independent of $T$. Thus, we obtain
\begin{align*}
&\frac{d}{dt}\left(\|\nabla^k(\nabla g)\|_{L^2}^2 + \|\nabla^k(\nabla u)\|_{L^2}^2 +\int_{\T^d}\nabla^k (\nabla g) \cdot \nabla^k  u\,dx + \frac58 \|\nabla^k u\|_{L^2}^2\right)\\
&\quad \le \left(-\frac{1}{32} + C\e_1\right) (\|\nabla^k (\nabla g)\|_{L^2}^2 + \|\nabla^k(\nabla u)\|_{L^2}^2 + \|\nabla^k u\|_{L^2}^2) +C\mathfrak{X}(T;k) + 4\|\nabla^k (\nabla v)\|_{L^2}^2.
\end{align*}
Next, we combine the previous relation with Lemma \ref{LB.3} to get
\[
\begin{aligned}
&\frac{d}{dt}\left(\|\nabla^k(\nabla g)\|_{L^2}^2 + \|\nabla^k(\nabla u)\|_{L^2}^2 +\int_{\T^d}\nabla^k (\nabla g) \cdot \nabla^k  u\,dx + \frac58 \|\nabla^k u\|_{L^2}^2 + 4\|\nabla^k v\|_{L^2}^2\right)\\
&\quad \le \left(-\frac{1}{32} + C\e_1\right) (\|\nabla^k(\nabla g)\|_{L^2}^2 + \|\nabla^k(\nabla u)\|_{L^2}^2 + \|\nabla^k u\|_{L^2}^2 + \|\nabla^k v\|_{L^2}^2) +C\mathfrak{X}(T;k)\\
&\quad \le \left(-\frac{1}{32} + C\e_1\right) (\|\nabla^k(\nabla g)\|_{L^2}^2 + \|\nabla^k(\nabla u)\|_{L^2}^2 + \|\nabla^k u\|_{L^2}^2 + \|\nabla^k v\|_{L^2}^2)+C\mathfrak{X}_0(k),
\end{aligned}
\]
where $C$ is independent of $T$. Since $\e_1$ is sufficiently small, we have $-1/32+C\e_1<0$. Since Young's inequality gives
\[
\begin{aligned}
&\frac14 \lt(\|\nabla^k(\nabla g)\|_{L^2}^2 + \|\nabla^k(\nabla u)\|_{L^2}^2 + \|\nabla^k u\|_{L^2}^2 + \|\nabla^k v\|_{L^2}^2\rt)\cr
&\quad \leq \|\nabla^k(\nabla g)\|_{L^2}^2 + \|\nabla^k(\nabla u)\|_{L^2}^2 +\int_{\T^d}\nabla^k (\nabla g) \cdot \nabla^k  u\,dx + \frac58 \|\nabla^k u\|_{L^2}^2 + 4\|\nabla^k v\|_{L^2}^2\\
&\quad \leq 4\lt( \|\nabla^k(\nabla g)\|_{L^2}^2 + \|\nabla^k(\nabla u)\|_{L^2}^2 + \|\nabla^k u\|_{L^2}^2 + \|\nabla^k v\|_{L^2}^2\rt),
\end{aligned}
\]
the previous relation combined with Gr\"onwall's lemma yields
\bq\label{AppB-1}
\sup_{0\le t \le T} (\|\nabla^k (\nabla g)\|_{L^2}^2 + \|\nabla^k(\nabla u)\|_{L^2}^2) \le C\mathfrak{X}_0(k+1),
\eq
where $C$ is independent of $T$. Finally, we obtain from Lemma \ref{LB.3} that
\begin{align*}
\frac{d}{dt}&\|\nabla^{k+1} v\|_{L^2}^2 + \|\nabla^{k+1}(\nabla v)\|_{L^2}^2\\
&\le \left(-\frac12 +C\e_1\right)\|\nabla^{k+1}v\|_{L^2}^2+C\lt(\|\nabla^{k+1} u\|_{L^2}^2+\|\nabla^k v\|_{L^2} + \|\nabla^k u\|_{L^2}^2+\|g\|_{H^k}^2\rt)\\
&\le  \left(-\frac12 +C\e_1\right)\|\nabla^{k+1}v\|_{L^2}^2 + C\mathfrak{X}_0(k+1).
\end{align*}
Here we used \eqref{AppB-1} and $C$ is independent of $T$. Thus, we again use Gr\"onwall's lemma to get
\bq\label{AppB-2}
\sup_{0\le t \le T}\|\nabla^{k+1} v\|_{L^2}^2 \le C\mathfrak{X}_0(k+1).
\eq
Hence, we gather the estimates \eqref{AppB-1} and \eqref{AppB-2} to complete the proof.
\end{proof}

Based on the {\it a priori} estimates in Corollary \ref{CB.1}, we can extend our local-in-time existence result in Theorem \ref{L5.2} to the global-in-time estimate.

\begin{proof}[Proof of Theorem \ref{T2.3}]
First, we choose $\e_1 \ll1$ sufficiently small so that the conditions in Corollary \ref{CB.1} hold. Set 
\[
\tilde\e_0^2 := \frac{\e_1^2}{2(1+C_1)},
\] 
where $C_1 > 0$ is appeared in Corollary \ref{CB.1}, and assume that the initial data $(g_0,u_0,v_0)$ satisfy
\[
\|g_0\|_{H^s}^2 + \|u_0\|_{H^s}^2 + \|v_0\|_{H^s}^2 < \tilde\e_0^2.
\]
Now, we define the lifespan of the solutions $(g,u,v)$ to the system \eqref{E-2}-\eqref{ini_F-1} as
\[
\tilde T := \sup\{ t \ge 0\ | \ \sup_{0 \le \tau \le t} (\|g(\cdot, \tau)\|_{H^s}^2 + \|u(\cdot, \tau)\|_{H^s}^2 + \|v(\cdot ,\tau)\|_{H^s}^2) < \e_1^2\}.
\]
Since $\tilde\e_0 < \e_1$, Theorem \ref{L5.2} implies $\tilde T >0$. Suppose $\tilde T< + \infty$. Then, by definition we have
\bq\label{contra}
\sup_{0 \le t \le \tilde T} \lt(\|g(\cdot, t)\|_{H^s}^2 + \|u(\cdot, t)\|_{H^s}^2 + \|v(\cdot ,t)\|_{H^s}^2\rt) = \e_1^2.
\eq
However, Corollary \ref{CB.1} implies
\[
\sup_{0 \le t \le \tilde T} (\|g(\cdot, t)\|_{H^s}^2 + \|u(\cdot, t)\|_{H^s}^2 + \|v(\cdot ,t)\|_{H^s}^2) \le C_1 \tilde{\e_0}^2 \le \frac{\e_1^2}{2} <\e_1^2,
\]
which contradicts our assumption \eqref{contra}. Therefore, the lifespan $\tilde{T} = \infty$, i.e., the strong solution exists globally in time.
\end{proof}

\subsection{Proof of Theorem \ref{T2.3_2}}
In this part, we investigate the global-in-time existence of strong solutions to the system  \eqref{E-4}. Similarly to the isothermal pressure case, we define
\[
\mathfrak{Y}(T;k) := \sup_{0\le t \le T} \left( \|h\|_{H^{k-1}}^2 + \|u\|_{H^k}^2 + \|v\|_{H^k}^2\right) \quad \mbox{and} \quad \mathfrak{Y}_0(k) :=  \|h_0\|_{H^{k-1}}^2 + \|u_0\|_{H^k}^2 + \|v_0\|_{H^k}^2.
\]
\begin{remark}\label{R5.1}
Note that
\[
\|h\|_{H^{-1}} \leq \|\nabla K \star( \rho-1)\|_{L^2} \leq C\|h\|_{H^{-1}}.
\]
Thus, once we choose a sufficiently small $\e_1>0$ such that
\[
\mathfrak{Y}(T;s+1) \le \e_1^2 \ll 1 \quad \mbox{so that} \quad \sup_{0 \le t \le T_1} \|h(\cdot, t)\|_{L^\infty} \le \frac{1}{2},
\]
then by Proposition \ref{P5.1} with $\sigma=0$ we have 
\[
\begin{split}
\mathfrak{Y}(T, 0) &\le 2 \left(\int_{\T^d} \rho |u|^2\,dx + \int_{\T^d}|\nabla K\star(\rho-1)|^2\,dx +\int_{\T^d}|v|^2\,dx \right) \\
&\le 2 \left(\int_{\T^d} \rho_0 |u_0|^2\,dx + \int_{\T^d}|\nabla K\star(\rho_0-1)|^2\,dx +\int_{\T^d}|v_0|^2\,dx \right)\cr
&\le C\mathfrak{Y}_0(0).
\end{split}
\]
\end{remark}
We next present some relevant estimates similar to the isothermal pressure case.
\begin{lemma}\label{L5.3}
Let $s>d/2+1$, $T>0$ be given and suppose that $\mathfrak{Y}(T;s+1)\le \e_1^2 \ll 1$. Then we have
\[
\begin{aligned}
\frac{d}{dt}&\left(\|\nabla^k h\|_{L^2}^2 + \|\nabla^k (\nabla u)\|_{L^2}^2\right) +\frac32\|\nabla^k (\nabla u)\|_{L^2}^2 \\
&\le C\e_1 \left(\|\nabla^k h\|_{L^2}^2 + \|\nabla^k (\nabla u)\|_{L^2}^2 + \|\nabla^k u\|_{L^2}^2 \right)+ 2\|\nabla^k(\nabla v)\|_{L^2}^2
\end{aligned}
\]
for $0 \le k \le s$, where $C=C(d,k)$ is independent of $T$. 
\end{lemma}
\begin{proof}
Since zeroth-order estimates are analogous, we only consider higher-order estimates. For $1\le k\le s$, it follows from \eqref{E-4} that
\begin{align*}
\frac12\frac{d}{dt}\|\nabla^k h\|_{L^2}^2 &= -\int_{\T^d} \nabla^k (\nabla h \cdot u) \nabla^k h\,dx -\int_{\T^d} \nabla^k ((1+h)\nabla \cdot u)\nabla^k h\,dx\\
&= \frac12\int_{\T^d}(\nabla \cdot u)|\nabla^k h|^2\,dx - \int_{\T^d}[\nabla^k (\nabla h \cdot u) - \nabla^k (\nabla h)\cdot u]\nabla^k h\,dx\\
&\quad - \int_{\T^d} h \nabla^k (\nabla \cdot u) \nabla^k h\,dx - \int_{\T^d} \nabla^k h \, \nabla^k (\nabla \cdot u)\,dx\\
&\quad - \int_{\T^d} [\nabla^k (h \nabla \cdot u) - h \nabla^k (\nabla \cdot u)]\nabla^k h\,dx\\
&\le \frac{\|\nabla \cdot u\|_{L^\infty}}{2}\|\nabla^k h\|_{L^2}^2 + C\|\nabla^k h\|_{L^2} (\|\nabla u\|_{L^\infty}\|\nabla^k h\|_{L^2} + \|\nabla h\|_{L^\infty} \|\nabla^k u\|_{L^2})\\
&\quad +\|h\|_{L^\infty}\|\nabla^k(\nabla \cdot u)\|_{L^2}\|\nabla^k h\|_{L^2} - \int_{\T^d} \nabla^k h\,\nabla^k (\nabla \cdot u)\,dx\\
&\quad + C\|\nabla^k h\|_{L^2} (\|\nabla h\|_{L^\infty}\|\nabla^{k-1}(\nabla \cdot u)\|_{L^2} + \|\nabla^k h\|_{L^2}\|\nabla \cdot u\|_{L^\infty})\\
&\le C\e_1 (\|\nabla^k h\|_{L^2}^2 + \|\nabla^k (\nabla u)\|_{L^2}^2) + C\e_1\|\nabla^k u\|_{L^2}^2 -\int_{\T^d} \nabla^k h\,\nabla^k (\nabla \cdot u)\,dx,
\end{align*}
where $C=C(d,k)$ is a positive constant independent of $\e_0$, $\e_1$ and $T_1$ and we used Lemma \ref{lem_moser}.

Next, we estimate $\nabla^k (\nabla u)$ as
\begin{align*}
\frac12\frac{d}{dt}\|\nabla^k (\nabla u)\|_{L^2}^2 &= -\int_{\T^d} \nabla^k (\nabla(u \cdot\nabla  u)):\nabla^k (\nabla u)\,dx - \int_{\T^d} \nabla^k (\nabla(\nabla K \star h)):\nabla^k (\nabla u)\,dx\\
&\quad - \int_{\T^d} \nabla^k (\nabla(u-v)):\nabla^k (\nabla u)\,dx\\
&= -\int_{\T^d} u \cdot \nabla(\nabla^k(\nabla u)) :\nabla^k (\nabla u)\,dx -\int_{\T^d} [\nabla^k ((u \cdot \nabla^2 u)) - u \cdot \nabla^k (\nabla^2 u)]:\nabla^k (\nabla u)\,dx\\
&\quad -\int_{\T^d} \nabla^k (\nabla u)^2:\nabla^k (\nabla u)\,dx + \int_{\T^d} \nabla^k h \nabla^k (\nabla \cdot u)\,dx \\
&\quad- \|\nabla^k (\nabla u)\|_{L^2}^2 + \int_{\T^d}\nabla^k (\nabla v):\nabla^k(\nabla u)\,dx\\
& \le C\|\nabla u\|_{L^\infty}\|\nabla^k (\nabla u)\|_{L^2}^2 + \int_{\T^d} \nabla^k h \nabla^k (\nabla \cdot u)\,dx  -\frac34\|\nabla^k(\nabla u)\|_{L^2}^2 + \|\nabla^k(\nabla v)\|_{L^2}^2\\
& \le C\e_1\|\nabla^k (\nabla u)\|_{L^2}^2 + \int_{\T^d} \nabla^k h \nabla^k (\nabla \cdot u)\,dx -\frac34\|\nabla^k(\nabla u)\|_{L^2}^2 + \|\nabla^k(\nabla v)\|_{L^2}^2.
\end{align*}
Here we used
\begin{align*}
-\int_{\T^d} \nabla^k (\nabla(\nabla K \star h)):\nabla^k (\nabla u)\,dx &= -\int_{\T^d} \nabla^k (\Delta K \star h) : \nabla^k (\nabla \cdot u)\,dx= \int_{\T^d} \nabla^k h \nabla^k (\nabla \cdot u)\,dx.
\end{align*}
Thus, we combine the previous two estimates to get the desired result.
\end{proof}
Similarly as before, in the lemma below, we present the estimate which gives the dissipation rate for $\nabla^k h$.
\begin{lemma}\label{L5.4}
Let $s>d/2+1$, $T>0$ be given and suppose that $\mathfrak{Y}(T;s+1)\le \e_1^2 \ll 1$. Then we have
\[
\begin{aligned}
-\frac{d}{dt}&\int_{\T^d}\nabla^k h \nabla^k(\nabla \cdot u)\,dx+\frac14 \|\nabla^k h\|_{L^2}^2\\
&\le C\e_1 \left(\|\nabla^k h\|_{L^2}^2 + \|\nabla^k (\nabla u)\|_{L^2}^2 + \|\nabla^{k-1} h\|_{L^2}^2 + \|\nabla^k u\|_{L^2}^2 \right)+ \frac43\|\nabla^k(\nabla u)\|_{L^2}^2
\end{aligned}
\]
for $0 \le k\le s$, where $C=C(d,k)$ is independent of $T$. 
\end{lemma}
\begin{proof}
Direct computation yields
\begin{align*}
-\frac{d}{dt}\int_{\T^d}\nabla^k h \nabla^k (\nabla \cdot u)\,dx&= \int_{\T^d} \nabla^k (\nabla \cdot ((1+h)u)) \nabla^k (\nabla \cdot u)\,dx\\
&\quad + \int_{\T^d}\nabla^k (\nabla \cdot (u \cdot \nabla u)  -h +\nabla \cdot u)\nabla^k h\,dx\\
& =: \mathcal{I}_1 + \mathcal{I}_2.
\end{align*}
For $\mathcal{I}_1$, we use Lemma \ref{lem_moser} to get
\begin{align*}
\mathcal{I}_1 &= -\int_{\T^d}\nabla^k ((1+h)u)\cdot \nabla (\nabla^k (\nabla\cdot u))\,dx\\
&= -\int_{\T^d}(\nabla^k h) u \cdot \nabla (\nabla^k (\nabla \cdot u))\,dx + \int_{\T^d}\nabla \cdot [\nabla^k((1+h)u) - u \nabla^k h ]\nabla^k (\nabla \cdot u)\,dx\\
&=-\int_{\T^d}(\nabla^k h) u \cdot \nabla (\nabla^k (\nabla \cdot u))\,dx + \int_{\T^d}[\nabla^k(\nabla h \cdot u) - \nabla^k (\nabla h) \cdot u]\nabla^k (\nabla \cdot u)\,dx \\
&\quad + \int_{\T^d}[\nabla^k (h \nabla \cdot u) - \nabla^k h \nabla \cdot u]\nabla^k (\nabla \cdot u)\,dx + \int_{\T^d} |\nabla^k (\nabla \cdot u)|^2\,dx\\
& \le -\int_{\T^d}(\nabla^k h) u \cdot \nabla (\nabla^k (\nabla \cdot u))\,dx +C\|\nabla^k (\nabla \cdot u)\|_{L^2}(\|\nabla u\|_{L^\infty}\|\nabla^k h\|_{L^2} + \|\nabla h\|_{L^\infty}\|\nabla^k u\|_{L^2})\\
&\quad + C\|\nabla^k (\nabla \cdot u)\|_{L^2}(\|\nabla^2 u\|_{L^\infty}\|\nabla^{k-1} h\|_{L^2} + \|h\|_{L^\infty}\|\nabla^k(\nabla \cdot u)\|_{L^2}) + \|\nabla^k(\nabla\cdot u)\|_{L^2}^2\\
&\le  -\int_{\T^d}(\nabla^k h) u \cdot \nabla (\nabla^k (\nabla \cdot u))\,dx\\
&\quad +C\e_1 \left(\|\nabla^k h\|_{L^2}^2 + \|\nabla^k (\nabla u)\|_{L^2}^2 + \|\nabla^{k-1} h\|_{L^2}^2 + \|\nabla^k u\|_{L^2}^2 \right)+ \|\nabla^k(\nabla u)\|_{L^2}^2,
\end{align*}
where $C$ only depends on $d$ and $k$. Here we used
\[
\|\nabla u\|_{W^{1,\infty}} \leq C\|u\|_{H^{s+1}} \quad \mbox{and} \quad \|h\|_{W^{1,\infty}} \leq C\|h\|_{H^s}
\]
due to $s > d/2 + 1$.

Next, we estimate $\mathcal{I}_2$ as 
\begin{align*}
\mathcal{I}_2 &= \int_{\T^d} u\cdot\nabla (\nabla^k(\nabla \cdot u)) \nabla^k h\,dx + \int_{\T^d} \lt(\nabla^k (u \cdot \nabla (\nabla \cdot u)) - u \cdot \nabla^k (\nabla (\nabla \cdot u))\rt)\nabla^k h\,dx\\
&\quad +\int_{\T^d} \nabla^k   \lt(\nabla\cdot(u \cdot \nabla u)  - u \cdot \nabla(\nabla \cdot u)\rt)\nabla^k h \,dx -\|\nabla^k h\|_{L^2}^2 + \int_{\T^d}\nabla^k h \nabla^k (\nabla \cdot u)\,dx\\
&\le  \int_{\T^d} u\cdot\nabla (\nabla^k(\nabla \cdot u)) \nabla^k h\,dx + C\|\nabla^k h\|_{L^2}(\|\nabla u\|_{L^\infty}\|\nabla^k (\nabla u)\|_{L^2} + \|\nabla^2  u\|_{L^\infty} \|\nabla^k u\|_{L^2})\\
&\quad -\frac14\|\nabla^k h\|_{L^2}^2 + \frac13\|\nabla^k (\nabla u)\|_{L^2}^2,
\end{align*}
where we used Lemma \ref{lem_moser} and $C$ only depends on $d$ and $k$. Finally, we combine the estimates for $\mathcal{I}_1$ and $\mathcal{I}_2$ to get the desired result.
\end{proof}

\begin{remark}\label{R5.2}
Since 
\[
\left|\int_{\T^d} \nabla^k h \nabla^k(\nabla \cdot u)\,dx\right| \le \frac12\|\nabla^k h\|_{L^2}^2 +\frac12\|\nabla^k (\nabla u)\|_{L^2}^2,
\]
the following relation is direct:
\[
\begin{aligned}
\frac12\lt( \|\nabla^k h\|_{L^2}^2 +\|\nabla^k (\nabla u)\|_{L^2}^2\rt) &\le \|\nabla^k h\|_{L^2}^2 +\|\nabla^k (\nabla u)\|_{L^2}^2 -\int_{\T^d}\nabla^k h \nabla^k(\nabla \cdot u)\,dx \\
&\le \frac32\lt( \|\nabla^k h\|_{L^2}^2 +\|\nabla^k (\nabla u)\|_{L^2}^2\rt).
\end{aligned}
\]
\end{remark}
Now, we consider the estimates for the Navier-Stokes part in \eqref{E-4}.
\begin{lemma}\label{L5.5}
Let $s>d/2+1$, $T>0$ be given and suppose that $\mathfrak{Y}(T;s+1)\le \e_1^2 \ll 1$. Then we have
\[
\begin{aligned}
\frac{d}{dt}&\|\nabla^k v\|_{L^2}^2 + \frac32\|\nabla^k v\|_{L^2}^2 + \|\nabla^k (\nabla v)\|_{L^2}^2\\
&\le C\e_1 \lt(\|\nabla^k v\|_{L^2}^2 + \lt(\|\nabla^{k-1} v\|_{L^2}^2 + \|\nabla^{k-1} u\|_{L^2}^2 + \|\nabla^{k-1} h\|_{L^2}^2\rt)(1-\delta_{k,0})\rt) + 2\|\nabla^k u\|_{L^2}^2
\end{aligned}
\]
for $0 \le k\le s+1$, where $C=C(d,k)$ is independent of $T$. 
\end{lemma}
\begin{proof} For $k=0$, we estimate 
\[\begin{aligned}
\frac12\frac{d}{dt}\|v\|_{L^2} + \|\nabla v\|_{L^2}^2&= -\int_{\T^d} (v \cdot \nabla v) \cdot v\,dx -\int_{\T^d} (1+h)(v-u)\cdot v\,dx\\
&\le -\frac34 \|v\|_{L^2} + \|u\|_{L^2} + \|h\|_{L^\infty}\|u\|_{L^2}\|v\|_{L^2}\\
&\le C\e_1 \lt(  \|v\|_{L^2} + \|u\|_{L^2}\rt) -\frac34 \|v\|_{L^2} + \|u\|_{L^2},
\end{aligned}\]
where $C$ is independent of $T$. For $k \geq 1$, we deduce from \eqref{E-4} that
\begin{align*}
\frac12\frac{d}{dt}&\|\nabla^k v\|_{L^2}^2 + \|\nabla^k(\nabla v)\|_{L^2}^2\\
&=-\int_{\T^d}\nabla^k (v \cdot \nabla v) \cdot \nabla^k v\,dx -\int_{\T^d}\nabla^k((1+h)(v-u))\nabla^k v\,dx\\
&= -\int_{\T^d}[\nabla^k (v \cdot \nabla v) - v\cdot \nabla (\nabla^k v)]\cdot \nabla^k v\,dx\\
&\quad + \int_{\T^d}\nabla^{k-1}(h(v-u))\nabla^{k+1} v\,dx -\int_{\T^d}\nabla^k(v-u)\cdot \nabla^k v\,dx\\
&\le C\|\nabla v\|_{L^\infty}\|\nabla^k v\|_{L^2}^2 + \|\nabla^{k-1}(h(v-u))\|_{L^2}\|\nabla^k (\nabla v)\|_{L^2} -\frac34\|\nabla^k v\|_{L^2}^2 + \|\nabla^k u\|_{L^2}^2\\
&\le C\e_1 \|\nabla^k v\|_{L^2}^2 +C\|\nabla^k (\nabla v)\|_{L^2}\lt(\|h\|_{L^\infty}\|\nabla^{k-1}(v-u)\|_{L^2} + \|\nabla^{k-1} h \|_{L^2}\|u-v\|_{L^\infty}\rt)\\
&\quad -\frac34\|\nabla^k v\|_{L^2}^2  + \|\nabla^k u\|_{L^2}^2\\
&\le C\e_1 \lt(\|\nabla^k v\|_{L^2}^2 + \|\nabla^{k-1} v\|_{L^2}^2 + \|\nabla^{k-1} u\|_{L^2}^2 + \|\nabla^{k-1} h\|_{L^2}^2\rt)\\
&\quad + \frac12\|\nabla^k(\nabla v)\|_{L^2}^2 -\frac34\|\nabla^k v\|_{L^2}^2  + \|\nabla^k u\|_{L^2}^2.
\end{align*}
Here we used Lemma \ref{lem_moser} and $C$ only depends on $d$ and $k$. This implies our desired estimate.
\end{proof}
Now, we are ready to present the uniform energy estimates based on Lemmas \ref{L5.3}-\ref{L5.5}.
\begin{corollary}\label{C5.1}
Let $s>d/2+1$, $T>0$ be given and suppose that $\mathfrak{Y}(T;s+1)\le \e_1^2 \ll 1$. Then we have
\[
\mathfrak{Y}(T;s+1) \le C_2 \mathfrak{Y}_0(s+1), 
\]
where $C_2$ is independent of $T$.
\end{corollary} 
\begin{proof}
The proof is similar to that of Corollary \ref{CB.1}.
 By Remark \ref{R5.1}, it suffices to show the induction step. Suppose that there exists a positive constant $C>0$ that is independent of $T$ such that
\[
\mathfrak{Y}(T;m) \le C\mathfrak{Y}_0(m) \quad \mbox{for} \quad 0 \le m \le k.
\]
First, we combine Lemmas \ref{L5.3} and \ref{L5.4} to obtain
\[
\begin{aligned}
\frac{d}{dt}&\left(\|\nabla^k h\|_{L^2}^2 + \|\nabla^k(\nabla u)\|_{L^2}^2 -\int_{\T^d}\nabla^k h \nabla^k (\nabla \cdot u)\,dx\right)\\
&\le \left(-\frac16 + C\e_1\right) (\|\nabla^k h\|_{L^2}^2 + \|\nabla^k(\nabla u)\|_{L^2}^2) +C\mathfrak{Y}(T;k) + 2\|\nabla^k (\nabla v)\|_{L^2}^2,
\end{aligned}
\]
where $C$ is independent of $T$. Next, we combine the previous relation with Lemma \ref{L5.5} to get
\[
\begin{aligned}
\frac{d}{dt}&\left(\|\nabla^k h\|_{L^2}^2 + \|\nabla^k(\nabla u)\|_{L^2}^2 -\int_{\T^d}\nabla^k h \nabla^k (\nabla \cdot u)\,dx + 2\|\nabla^k v\|_{L^2}^2\right)\\
&\le \left(-\frac16 + C\e_1\right) (\|\nabla^k h\|_{L^2}^2 + \|\nabla^k(\nabla u)\|_{L^2}^2 + \|\nabla^k v\|_{L^2}^2) +C\mathfrak{Y}(T;k)\\
&\le \left(-\frac16 + C\e_1\right) (\|\nabla^k h\|_{L^2}^2 + \|\nabla^k(\nabla u)\|_{L^2}^2 + \|\nabla^k v\|_{L^2}^2) +C\mathfrak{Y}_0(k).
\end{aligned}
\]
Here we used the induction hypothesis and $C$ is independent of $T$. Since $\e_1$ is sufficiently small, one knows $-1/6+C\e_1<0$. Thus, we use the estimates in Remark \ref{R5.2} and Gr\"onwall's lemma to have
\bq\label{E-7}
\sup_{0\le t \le T_1} (\|\nabla^k h\|_{L^2}^2 + \|\nabla^k(\nabla u)\|_{L^2}^2) \le C\mathfrak{Y}_0(k+1),
\eq
where $C$ is independent of $T$. Next, Lemma \ref{L5.5} implies
\begin{align*}
\frac{d}{dt}&\|\nabla^{k+1} v\|_{L^2}^2 + \|\nabla^{k+1}(\nabla v)\|_{L^2}^2\\
&\le \left(-\frac32 +C\e_1\right)\|\nabla^{k+1}v\|_{L^2}^2 + C(\|\nabla^k v\|_{L^2}^2 +\|\nabla^k u\|_{L^2}^2 + \|\nabla^k h\|_{L^2}^2 + \|\nabla^{k+1} u\|_{L^2}^2)\\
&\le  \left(-\frac32 +C\e_1\right)\|\nabla^{k+1}v\|_{L^2}^2 + C\mathfrak{Y}_0(k+1),
\end{align*}
where we used the induction hypothesis and \eqref{E-7}. Here $C$ is independent of $T$. Thus, we again use Gr\"onwall's lemma to have
\bq\label{E-8}
\sup_{0\le t \le T_1}\|\nabla^{k+1} v\|_{L^2}^2 \le C\mathfrak{Y}_0(k+1).
\eq
Thus we combine \eqref{E-7} and \eqref{E-8} to prove the induction step and hence, we conclude the proof.
\end{proof}

\begin{proof}[Proof of Theorem \ref{T2.3_2}] By using Corollary \ref{C5.1} and almost the same argument as in the proof of Theorem \ref{T2.3}, we can extend the local-in-time strong solutions to the global-in-time ones.
\end{proof}

%
%
%
%
%
%

\section*{Acknowledgments}
The work of Y.-P. Choi is supported by NRF grant (No. 2017R1C1B2012918), POSCO Science Fellowship of POSCO TJ Park Foundation, and Yonsei University Research Fund of 2019-22-021. The work of J. Jung is supported by NRF grant (No. 2019R1A6A1A10073437).

\appendix

%
%
%
%
%
%




\section{Global-in-time existence of weak entropy solutions to VPNS system}\label{sec:weak}
 
In this appendix, we discuss the global-in-time existence of weak solutions to the system \eqref{main_eq}. To be more precise, our goal is to prove:

\begin{theorem}\label{T2.1}
For $d=2,3$, suppose that the initial data $(f_0, v_0)$ satisfies the following conditions:
\[
f_0 \in (L_+^1\cap L^\infty)(\T^d\times\R^d) \quad \mbox{and} \quad  v_0 \in \mathsf{H}.
\]
Furthermore, we assume that the initial free energy is bounded:
\[
\iint_{\T^d \times \R^d} \left(\frac{|\xi|^2}{2} + \sigma |\log f_0| \right) f_0 \,dxd\xi + \frac{1}{2}\int_{\T^d} |\nabla U_0|^2 \,dx + \frac{1}{2}\int_{\T^d} |v_0|^2\,dx < \infty.
\]
Then for every $T>0$, there exists at least one weak solution $(f,v)$ to the system \eqref{main_eq} in the sense of Definition \ref{D2.1} on the time interval $[0,T]$. 
\end{theorem}

%
Here, without loss of generality, we set $\tau=1$.\\

In order to show the global-in-time existence of weak solutions to the system \eqref{main_eq}, we first regularize the system by employing a cut-off function for the fluid and local particle velocities in the kinetic equation and the drag force in the fluid equation. We also mollify the convection velocity of the fluid equation and remove the singularity in the local particle velocity $u$. Next, we show the existence of weak solutions for the regularized system by using a fixed point argument. We then use some weak and strong compactness arguments together with the velocity averaging lemma to pass to the limit when the regularization parameters tend to zero or infinity. Finally, we prove the limiting functions obtained from the previous step satisfy our VPNS system in the sense of Definition \ref{D2.1}. 

\subsection{Regularized system}We regularize the system \eqref{main_eq} and investigate the existence of solutions and associated entropy inequalities. For the regularization parameters $\lambda$ and $\e$, we consider
\begin{align}
\begin{aligned}\label{C-1}
&\partial_t f^{\lambda,\e} + \xi \cdot \nabla f^{\lambda,\e} + \nabla_\xi \cdot (((\chi_\lambda(v^{\lambda,\e}) - \xi) - \nabla K^\e \star (\rho^{\lambda,\e}-1))f^{\lambda,\e}) \\
&\hspace{5cm}= \nabla_\xi \cdot (\sigma \nabla_\xi f^{\lambda,\e} - (\chi_\lambda(u_{\e}^{\lambda,\e})-\xi)f^{\lambda,\e}),\\
&\partial_t v^{\lambda,\e} + ((\theta_\e \star v^{\lambda,\e})\cdot \nabla) v^{\lambda,\e} + \nabla p^{\lambda,\e} -\Delta v^{\lambda,\e} = \rho^\eta (u^{\lambda,\e} -v^{\lambda,\e})\mathds{1}_{\{|v^{\lambda,\e}| \le \lambda\}},\\
&\nabla \cdot v^{\lambda,\e} =0
\end{aligned}
\end{align}
subject to initial data:
\[
f^{\lambda,\e}(x,\xi,0)=f^\lambda(x,\xi,0) := f_0(x,\xi)\mathds{1}_{\{|\xi|\le \lambda\}}(\xi) \quad \mbox{and} \quad v^{\lambda,\e}(x,0) = v^\e(x,0) := (\theta_\e \star v_0)(x),
\]
where $\chi_\lambda(v)$ is the truncation function given by
\[
\chi_\lambda(v) = v \mathds{1}_{\{ |v|\le \lambda\}}
\]
and $u_\e^{\lambda,\e}$ is defined by
\[
u_\e^{\lambda,\e} := \frac{\rho^{\lambda,\e} u^{\lambda,\e}}{\rho^{\lambda,\e} + \e}.
\]
Here the fluid velocity field $v$ is also regularized by using $\theta_\e = (1/\e^d)\theta(x/\e)$, where $\theta$ is a standard mollifier satisfying
\[
\theta \in \mathcal{C}_0^\infty(\T^d), \quad \theta \ge 0, \quad \int_{\T^d} \theta(x)\,dx = 1, \quad \mbox{and} \quad \mbox{supp } \theta \subset B_1(0).
\]
The regularized interaction potential $K^\e$ is given as
\bq\label{reg_K}
K^\e(x) = \left\{\begin{array}{lcl}\displaystyle  -\frac{c_0}{2} \log (\e + |x|^2) + G_0(x) & \mbox{ if } & d=2, \\
c_1(\e + |x|^2)^{-1/2} + G_1(x) & \mbox{ if } & d=3.\end{array}\right.
\eq
We use $\eta = (\lambda,\e)$ for notational simplicity whenever there is no confusion.
Then, we partially linearize the system \eqref{C-1} as follows:
\begin{align}
\begin{aligned}\label{C-2}
&\partial_t f^\eta + \xi \cdot \nabla f^\eta + \nabla_\xi \cdot (((\chi_\lambda(\tilde{v}) - \xi) - \nabla K^\e \star (\rho^\eta-1))f^\eta) = \nabla_\xi \cdot (\sigma \nabla_\xi f^\eta - (\chi_\lambda(\tilde{u})-\xi)f^\eta),\\
&\partial_t v^\eta + ((\theta_\e \star v^\eta)\cdot \nabla) v^\eta + \nabla p^\eta -\Delta v^\eta = \rho^\eta (u^\eta -\tilde{v})\mathds{1}_{\{|\tilde{v}|\le \lambda\}},\\
&\nabla \cdot v^\eta =0,
\end{aligned}
\end{align}
where $(\tilde{u}, \tilde{v})$ belong to $\mathcal{S} := L^2(\T^d \times (0,T))\times L^2(\T^d \times (0,T))$. As mentioned before, our strategy is to apply the fixed point theory argument to the system \eqref{C-2} to obtain the existence of weak solutions to system \eqref{C-1} and associated entropy inequality. For this, we need to estimate $L^p$-norm and velocity-moments of solutions $f^\eta$ to the kinetic equation in \eqref{C-2}. On the other hand, since $\nabla K^\e$ is bounded and Lipschitz continuous for fixed $\e>0$, the global-in-time existence of weak solutions to the kinetic equation in \eqref{C-2} can be found in \cite{De86, KMT13}. Furthermore, we get
\begin{align*}
\frac{d}{dt}\iint_{\T^d \times \R^d} (f^\eta)^p\,dxd\xi &= (p-1) \iint_{\T^d \times \R^d} (f^\eta)^p \nabla_\xi \cdot \lt( 2\xi  + \nabla K^\e \star (\rho^\eta-1) - \chi_\lambda (\tilde{u})  -\chi_\lambda(\tilde{v})\rt) dxd\xi \cr
&\quad - \sigma p(p-1) \iint_{\T^d \times \R^d} (f^\eta)^{p-2} |\nabla_\xi f^\eta|^2\,dxd\xi \cr
&=2d(p-1)\iint_{\T^d \times \R^d} (f^\eta)^p \,dxd\xi  - \frac{4\sigma (p-1)}{p}\iint_{\T^d \times \R^d} |\nabla_\xi (f^\eta)^{p/2}|^2\,dxd\xi
\end{align*}
for $p \in [1,\infty)$. We combine this with Gr\"onwall's lemma to obtain
\[
\|f^\eta(\cdot,\cdot,t)\|_{L^p}^p + \frac{4\sigma(p-1)}{p}\int_0^t e^{2d(p-1) (t-s)}\|\nabla_\xi (f^\eta)^{p/2}(\cdot,\cdot,s)\|_{L^2}^2\,ds \leq \|f^\eta_0\|_{L^p}^pe^{2d(p-1)t}.
\]
In particular, we have
\bq\label{bdd_f}
\|f^\eta(\cdot,\cdot,t)\|_{L^1} \leq \|f^\eta_0\|_{L^1} \le \|f_0\|_{L^1} = 1 \quad \mbox{and} \quad \|f^\eta(\cdot,\cdot,t)\|_{L^\infty} \le \|f_0^\eta\|_{L^\infty}e^{2dt}
\eq
for $t \in [0,T]$. 

We next present the estimates for higher-order velocity moments and entropy inequality of solutions to the system \eqref{C-2}.
\begin{lemma}\label{L3.1}
For a weak solution $f^\eta$ to the kinetic equation in \eqref{C-2}, its velocity moments satisfy the following bound:
\[
\sup_{t \in (0,T)}\iint_{\T^d \times \R^d} |\xi|^k f^\eta(x,\xi) \,dx d\xi  \le C(d,\eta, k, \sigma, T) \quad \forall \,k \ge 0. 
\]
\end{lemma}
\begin{proof}
First, we define a $k$-th moment of $f$ in velocity by
\[
m_k(f) := \iint_{\T^d \times \R^d} |\xi|^k f\,dxd\xi .
\] 
Then, for $k\ge 2$, we estimate
\begin{align*}
\frac{d}{dt} m_k(f^\eta)&= -k \iint_{\T^d \times \R^d} (\xi -\chi_\lambda(\tilde{v}) + \nabla K^\e\star(\rho^\eta-1))\cdot \xi f^\eta |\xi|^{k-2}\,dxd\xi \\
&\quad  - k\iint_{\T^d \times \R^d} (\xi-\chi_\lambda(\tilde{u}))\cdot \xi f^\eta |\xi|^{k-2}\,dxd\xi - \sigma k \iint_{\T^d \times \R^d} \nabla_\xi f^\eta \cdot \xi |\xi|^{k-2}\,dxd\xi \\
&= -2k m_k(f^\eta) -k \iint_{\T^d \times \R^d}  \nabla K^\e\star(\rho^\eta-1) \cdot \xi f^\eta |\xi|^{k-2}\,dxd\xi \\
&\quad + k\iint_{\T^d \times \R^d}(\chi_\lambda(\tilde{u}) + \chi_\lambda(\tilde{v})) \cdot \xi f^\eta |\xi|^{k-2} \,dxd\xi  + \sigma k(k-2+d)m_{k-2}(f^\eta)\\
&\le -k m_k(f^\eta)+ \frac{k}{2} \iint_{\T^d \times \R^d} |\nabla K^\e\star(\rho^\eta-1)|^2 f^\eta |\xi|^{k-2}\,dxd\xi \\
&\quad + \frac{k}{2}\iint_{\T^d \times \R^d}(\chi_\lambda (\tilde{u})+\chi_\lambda(\tilde{v}))^2  f^\eta |\xi|^{k-2} \,dxd\xi  + \sigma k(k-2+d)m_{k-2}(f^\eta)\\
&\le \frac{k}{2} \iint_{\T^d \times \R^d} |\nabla K^\e\star(\rho^\eta-1)|^2 f^\eta |\xi|^{k-2}\,dxd\xi \\
&\quad + \frac{k}{2}\iint_{\T^d \times \R^d}(\chi_\lambda (\tilde{u})+\chi_\lambda(\tilde{v}))^2  f^\eta |\xi|^{k-2} \,dxd\xi  + \sigma k(k-2+d)m_{k-2}(f^\eta)\\
&\le C m_{k-2}(f^\eta),
\end{align*}
where $C = C(d,\eta, k, \sigma,T)$ is a positive constant and we used Young's inequality and
\[
\|\nabla K^\e \star(\rho^\eta-1)\|_{L^\infty} \le \|\nabla K^\e \|_{L^\infty}\|\rho^\eta -1\|_{L^1} \le C(\e).
\]
Since $m_0(f^\eta)$ is just $\|f^\eta(\cdot,\cdot,t)\|_{L^1} = \|f_0^\eta\|_{L^1}$ for $t\geq0$, we combine Gr\"onwall's lemma and induction argument to yield
\[
\sup_{t \in (0,T)}\iint_{\T^d \times \R^d} |\xi|^k f^\eta \,dx d\xi  \le C(d,\eta, k, \sigma, T) \quad \forall\, k=0,2,4, \dots. 
\]
Moreover, for $k \in \R_+ \setminus\{0,2,4, \dots\}$, we get
\[
\iint_{\T^d \times \R^d} |\xi|^k f^\eta \,dxd\xi \le \left(\iint_{\T^d \times \R^d} |\xi|^{2\lfloor k \rfloor} f^\eta\,dxd\xi\right)^{1/2}\left(\iint_{\T^d \times \R^d} |\xi|^{2(k-\lfloor k \rfloor)} f^\eta\,dxd\xi \right)^{1/2},
\]
where $\lfloor k \rfloor$ denotes the greatest integer less than or equal to $k$. Furthermore, we have
\[
\begin{split}
\iint_{\T^d \times \R^d} |\xi|^{2(k-\lfloor k \rfloor)} f^\eta \,dxd\xi  &\le \left(\iint_{\T^d \times \R^d} |\xi|^2 f^\eta \,dxd\xi  \right)^{\frac{2(k-\lfloor k \rfloor)}{2}} \left( \iint_{\T^d \times \R^d} f^\eta \, dxd\xi  \right)^{\frac{2-2(k-\lfloor k \rfloor)}{2}}\\
& \le C(d,\eta, k, \sigma, T).
\end{split}
\]
This asserts our desired result.
\end{proof}

In the proposition below, we provide the estimate of entropy inequality. 
\begin{proposition}\label{P3.1}
For a weak solution $f^\eta$ to the kinetic equation in \eqref{C-2}, we have the following relation:
\begin{align*}
& \iint_{\T^d \times \R^d} \left(\frac{|\xi|^2}{2} +  \sigma \log f^\eta \right) f^\eta \,dxd\xi  + \int_{\T^d} (\rho^\eta-1)K^\e\star(\rho^\eta-1)  \,dx \\
&\quad + \int_0^t\iint_{\T^d \times \R^d}\frac{1}{f^\eta} \left|\sigma \nabla_\xi f^\eta - (\chi_\lambda(\tilde{u})-\xi)f^\eta\right|^2 \,dxd\xi ds \\
&\qquad \leq   \iint_{\T^d \times \R^d} \left(\frac{|\xi|^2}{2} +  \sigma \log f_0^\lambda \right) f^\lambda \,dxd\xi  + \int_{\T^d} (\rho_0^\lambda-1)K^\e\star(\rho_0^\lambda-1)  \,dx \\
&\quad \qquad + \int_0^t\iint_{\T^d \times \R^d} (\chi_\lambda(\tilde{u})-\xi)\cdot \chi_\lambda(\tilde{u}) f^\eta \,dxd\xi ds  + \int_0^t\iint_{\T^d \times \R^d} (\chi_\lambda (\tilde{v}) -\xi)\cdot \xi f^\eta \,dxd\xi ds + \sigma d \|f_0^\lambda\|_{L^1}t 
\end{align*}
for almost all $t \in [0,T]$.
\end{proposition}
\begin{proof}
First, it follows from Lemma \ref{L3.1} that
\begin{align*}
&\frac12\frac{d}{dt}\left(\iint_{\T^d \times \R^d} |\xi|^2 f^\eta \,dxd\xi \right) \cr
&\quad = -\iint_{\T^d \times \R^d}(\xi - \chi_\lambda(\tilde{v})+ \nabla K^\e \star (\rho^\eta-1))\cdot \xi f^\eta \,dxd\xi  \\
&\qquad - \iint_{\T^d \times \R^d} (\xi-\chi_\lambda(\tilde{u}))\cdot \xi f^\eta \,dxd\xi   - \sigma\iint_{\T^d \times \R^d} \xi \cdot \nabla_\xi f^\eta \,dxd\xi \\
&\quad =  -\iint_{\T^d \times \R^d}(\xi-\chi_\lambda(\tilde{v}) + \nabla K^\e \star (\rho^\eta-1))\cdot \xi f^\eta \,dxd\xi - \iint_{\T^d \times \R^d} |\xi-\chi_\lambda(\tilde{u})|^2 f^\eta \,dxd\xi \\
&\qquad   - \iint_{\T^d \times \R^d} (\xi-\chi_\lambda(\tilde{u}))\cdot \chi_\lambda(\tilde{u}) f^\eta \,dxd\xi  - \sigma\iint_{\T^d \times \R^d} (\xi-\chi_\lambda(\tilde{u})) \cdot \nabla_\xi f^\eta \,dxd\xi .
\end{align*}
On the other hand, we get
\begin{align*}
\frac{d}{dt}\left(\int_{\T^d} (\rho^\eta-1)K^\e\star(\rho^\eta-1) \,dx\right) 
& = -\iint_{\T^d\times\T^d} K^\e(x-y)\nabla \cdot (\rho^\eta u^\eta)(x) (\rho^\eta(y)-1)\,dxdy\\
& = \iint_{\T^d\times\T^d} \nabla K^\e(x-y) (\rho^\eta(y)-1) \cdot (\rho^\eta u^\eta)(x) \,dxdy\\
& = \int_{\T^d} (\nabla K^\e \star (\rho^\eta-1)) \cdot  (\rho^\eta u^\eta) \,dx\\
& = \iint_{\T^d \times \R^d}(\nabla K^\e \star (\rho^\eta-1)) \cdot \xi f^\eta \,dxd\xi .
\end{align*}
Moreover, we estimate the entropy as
\begin{align*}
&\frac{d}{dt}\left(\iint_{\T^d \times \R^d} \sigma f^\eta \log f^\eta \,dxd\xi \right)\cr
&\quad = \iint_{\T^d \times \R^d} \sigma(\partial_t f^\eta) \log f^\eta\,dxd\xi \\
&\quad = -\sigma \iint_{\T^d \times \R^d} ( \xi - \chi_\lambda(\tilde{v}) + \nabla K^\e \star( \rho^\eta-1))\cdot \nabla_\xi f^\eta \,dxd\xi \\
&\qquad - \sigma \iint_{\T^d \times \R^d} (\xi-\chi_\lambda(\tilde{u}))\cdot \nabla_\xi f^\eta\,dxd\xi  -\sigma^2 \iint_{\T^d \times \R^d} \frac{|\nabla_\xi f^\eta|^2}{f^\eta}\,dxd\xi \\
&\quad = \sigma d \|f_0^\eta\|_{L^1} - \sigma \iint_{\T^d \times \R^d} (\xi-\chi_\zeta(\tilde{u}))\cdot \nabla_\xi f^\eta \,dxd\xi-\sigma^2 \iint_{\T^d \times \R^d} \frac{|\nabla_\xi f^\eta|^2}{f^\eta}\,dxd\xi.
\end{align*}
We then combine all the previous estimates to get the desired result.
\end{proof}
For the fluid part, the following estimate is obvious.
\begin{proposition}\label{P3.2}
For a weak solution $v^\eta$ to the Navier--Stokes system in \eqref{C-2}, we have
\[
\frac12\int_{\T^d} |v^\eta|^2\,dx  + \int_0^t\int_{\T^d} |\nabla v^\eta|^2\,dxds \leq \frac12 \int_{\T^d} |v_0^\eta|^2\,dx +  \int_0^t\int_{\T^d} \rho^\eta(u^\eta-\tilde{v} )\mathds{1}_{\{|\tilde{v}|\le \lambda\}}\cdot v^\eta \,dxds
\]
for almost all $t \in [0,T]$.
\end{proposition}

\subsection{Existence of weak solutions to the regularized system} Now, we proceed to the existence of weak solutions to system \eqref{C-1}. For this, we define a map $\mathcal{T}: \mathcal{S} \to \mathcal{S}$ as 
\[
(\tilde{u}, \tilde{v}) \mapsto \mathcal{T}(\tilde{u}, \tilde{v}) =: (u_\e^\eta, v^\eta).
\]

We first recall from \cite[Lemma 2.4]{KMT13} or \cite[Lemma 3.2]{MV07} that the following lemma which provides some $L^p$ bound estimates for $\rho$ and $\rho u$.

\begin{lemma}\label{lem_mom} Assume that $f$ satisfies
\bq\label{bdd_conds}
\|f\|_{L^\infty(\T^d\times\R^d\times(0,T))} \leq M \quad \mbox{and} \quad \sup_{0 \leq t \leq T}\iint_{\T^d \times \R^d} |\xi|^k f(x,\xi,t) \,dxd\xi \leq M \quad \forall\, k  \in [0,k^*]
\eq
for some $k^* > 1$. Then there exists a constant $C = C(M) > 0$ such that
\[
\|\rho(\cdot,t)\|_{L^{p_1}} \leq C \quad \forall\, p_1 \in [1, (k^*+d)/d) \quad \mbox{and} \quad \|(\rho u)(\cdot,t)\|_{L^{p_2}} \leq C \quad \forall\, p_2 \in [1, (k^*+d)/(d+1))
\]
for all $t \in [0,T]$.
\end{lemma}

In the lemma below, we show that $\mathcal{T}$ is well-defined.

\begin{lemma}\label{L3.2}
There exists a constant $C = C(d, \eta,\sigma, T)$ such that
\[
\|\mathcal{T}(\tilde{u}, \tilde{v})\|_{\mathcal{S}} \le C \quad \forall\, (\tilde{u}, \tilde{v}) \in \mathcal{S}.
\]
\end{lemma}
\begin{proof}
For the kinetic part, we use \eqref{bdd_f} and Lemmas \ref{L3.1} and \ref{lem_mom} to obtain
\[
\rho^\eta \in L^p(\T^d \times (0,T)), \quad \rho^\eta u^\eta \in L^p(\T^d \times (0,T)) \quad \forall\, p \in [1,\infty),
\]
and thus
\[
\|u_\e^\eta\|_{L^2} \le \frac{1}{\e} \|\rho^\eta u^\eta\|_{L^2} \le C.
\]
For the fluid part, we use Cauchy--Schwarz inequality and Young's inequality to get
\[
\begin{split}
\left|\int_{\T^d} \rho^\eta(u^\eta-\tilde{v} )\mathds{1}_{\{|\tilde{v}|\le \lambda\}}\cdot v^\eta \,dx\right| &\le \|\rho^\eta u^\eta\|_{L^2} \|v^\eta\|_{L^2} + \lambda\|\rho^\eta\|_{L^2}\|v^\eta\|_{L^2}\\
&\le  \|v^\eta\|_{L^2}^2 + \frac{1}{2}\|\rho^\eta u^\eta\|_{L^2}^2 + \frac{\lambda^2}{2}\|\rho^\eta\|_{L^2}^2\le \|v^\eta\|_{L^2}^2  +C.
\end{split}
\]
We then combine the above inequality with Proposition \ref{P3.2} and Gr\"onwall's lemma to conclude the proof.
\end{proof}

Next, we show that $\mathcal{T}$ is compact. Here, we consider the velocity averaging lemma from \cite[Lemma 3.2]{CCK16} (see also \cite[Theorem 2]{PS98} and \cite[Lemma 2.7]{KMT13}) and state a modified version to be used in the following proof.
\begin{lemma}\label{L3.3}
Let $T>0$, $r >1$, $q \in [1,(d+r)/(d+1))$, and $\{G^m\}$ be bounded in $L_{loc}^q(\T^d \times \R^d \times (0,T))$. Assume that 
\[
\mbox{$f^m$ is bounded in $L^\infty(\T^d \times \R^d \times (0,T))$}
\]
and
\bq\label{r-mom}
\mbox{$|\xi|^r f^m$ is bounded in $L^\infty(0,T; L^1(\T^d \times \R^d))$. }
\eq
If $f^m$ and $G^m$ satisfy
\[
\partial_t f^m + \xi \cdot \nabla f^m = \nabla_\xi^\alpha G^m, \quad f^m|_{t=0} = f_0 \in L^\infty(\T^d\times\R^d) 
\]
for some multi-index, then for any $\varphi(\xi)$, such that $|\varphi(\xi)| \leq c|\xi|$ as $|\xi| \to \infty$, the sequence 
\[ \left\{ \int_{\R^d} f^m \varphi(\xi) \,d\xi \right\} \]
is relatively compact in $L^q(\T^d \times (0,T))$.
\end{lemma}
 
Now, we are ready to prove the compactness of $\mathcal{T}$.
\begin{lemma}\label{C3.1}
For a uniformly bounded sequence $(\tilde{u}^m, \tilde{v}^m)$ in $\mathcal{S}$, the sequence $\mathcal{T}(\tilde{u}^m, \tilde{v}^m) =  ((u_\e^\eta)^m, (v^\eta)^m)$ converges strongly in $\mathcal{S}$, up to a subsequence.
\end{lemma}
\begin{proof}
Since the convergence of $\{(v^\eta)^m\}$ follows from the same argument as \cite[Lemma 4.2]{CCK16}, it suffices to show the convergence of $\{(u_\e^\eta)^m\}$. For the convergence of $\{(u_\e^\eta)^m\}$, we set
\[ 
f^m := (f^\eta)^m , \quad G^m :=\lt(\sigma \nabla_\xi (f^\eta)^m + (2\xi + \nabla K^\e \star ((\rho^\eta)^m-1) - \chi_\lambda(\tilde{v}^m)- \chi_\lambda(\tilde{u}^m))(f^\eta)^m \rt),
\]
then it is easy to see $G^m \in L^2_{loc}(\T^d \times \R^d \times (0,T))$. Choose $r$ appeared in \eqref{r-mom} so that $\frac{d+r}{d+1}>2$. Then, we set $\varphi(\xi) =1$ and $\varphi(\xi) = \xi$ in Lemma \ref{L3.3}, respectively, and obtain the following strong convergence up to a subsequence:
\begin{align*}
&(\rho^\eta)^m \to \rho^\eta \quad \quad \mbox{in} \quad L^2(\T^d \times (0,T)) \quad \mbox{and a.e.},\\
&(\rho^\eta)^m (u^\eta)^m \to \rho^\eta u^\eta \quad \mbox{in} \quad L^2(\T^d \times (0,T)).
\end{align*}
Consequently, it gives the convergence of $\{(u_\e^\eta)^m\}$ up to a subsequence.
\end{proof}

In conclusion, from Lemma \ref{L3.2} and Lemma \ref{C3.1}, the operator $\mathcal{T}$ is well-defined, continuous, and compact. Thus, we can use Schauder's fixed point theorem to attain the existence of a fixed point of $\mathcal{T}$, which asserts the existence of weak solutions to system \eqref{C-1}. Then, we employ the fixed point argument with Propositions \ref{P3.1} and \ref{P3.2} to yield the following estimate.

\begin{corollary}\label{C3.2}
Let $T>0$ and \emph{($f^\eta, v^\eta)$} be a weak solution to the system \eqref{C-1} on the time interval $[0,T]$. Then, it satisfies the following entropy inequality:
\[
\begin{aligned}
&\iint_{\T^d \times \R^d} \left(\frac{|\xi|^2}{2} + \sigma \log f^\eta \right) f^\eta \,dxd\xi +\int_{\T^d}  (\rho^\eta-1)K^\e\star(\rho^\eta-1) \,dx + \frac12 \int_{\T^d} |v^\eta|^2 \,dx \\
&\quad + \int_0^t\iint_{\T^d \times \R^d} \left(\frac{1}{f^\eta} \left|\sigma \nabla_\xi f^\eta - (\chi_\lambda(u_\e^\eta)-\xi)f^\eta\right|^2 + |\chi_\lambda(v^\eta)-\xi|^2 f^\eta\right)dxd\xi ds + \int_0^t\int_{\T^d} |\nabla v^\eta|^2  \,dxds\\
&\qquad \le \iint_{\T^d \times \R^d} \left(\frac{|\xi|^2}{2} + \sigma \log f_0^\lambda \right) f_0^\lambda \,dxd\xi +  \int_{\T^d}  (\rho_0^\lambda-1)K^\e\star(\rho_0^\lambda-1) \,dx + \frac12 \int_{\T^d} |v_0^\e|^2 \,dx +\sigma d \|f_0^\lambda\|_{L^1}t 
\end{aligned}
\]
for almost all $t \in [0,T]$.
\end{corollary} 
\begin{proof}
From Propositions \ref{P3.1} and \ref{P3.2}, we have
\[
\begin{aligned}
&\iint_{\T^d \times \R^d} \left(\frac{|\xi|^2}{2} + \sigma \log f^\eta \right) f^\eta \,dxd\xi +  \int_{\T^d}  (\rho^\eta-1)K^\e\star(\rho^\eta-1) \,dx + \frac12 \int_{\T^d} |v^\eta|^2 \,dx \\
&\quad + \int_0^t\iint_{\T^d \times \R^d} \left(\frac{1}{f^\eta} \left|\sigma \nabla_\xi f^\eta - (\chi_\lambda(u_\e^\eta)-\xi)f^\eta\right|^2 + |\chi_\lambda(v^\eta)-\xi|^2 f^\eta\right) dxd\xi ds + \int_0^t\int_{\T^d} |\nabla v^\eta|^2  \,dx ds\\
&\qquad \le \iint_{\T^d \times \R^d} \left(\frac{|\xi|^2}{2} + \sigma \log f_0^\lambda \right) f_0^\lambda \,dxd\xi +  \int_{\T^d}  (\rho_0^\lambda-1)K^\e\star(\rho_0^\lambda-1) \,dx + \frac12 \int_{\T^d} |v_0^\e|^2 \,dx \cr
&\qquad \quad+ \sigma d \|f_0^\lambda\|_{L^1} t + \int_0^t\iint_{\T^d \times \R^d} (\chi_\lambda(u_\e^\eta)-\xi)\cdot \chi_\lambda(u_\e^\eta) f^\eta \,dxd\xi ds.
\end{aligned}
\]
Here we obtain
\[
\begin{split}
\iint_{\T^d \times \R^d} (\chi_\lambda(u_\e^\eta)-\xi)\cdot \chi_\lambda(u_\e^\eta) f^\eta \,dxd\xi&= \int_{\T^d} \rho^\eta \chi_\lambda(u_\e^\eta)\cdot (\chi_\lambda(u_\e^\eta)-u^\eta)\,dx\\
&= \int_{\T^d} \rho^\eta \frac{\rho^\eta u^\eta}{\rho^\eta + \e} \cdot \left(\frac{\rho^\eta u^\eta}{\rho^\eta +\e}-u^\eta\right)\mathds{1}_{\{ |u_\e^\eta|\le\lambda\}}\,dx \\
&= -\e \int_{\T^d}| \chi_\lambda (u_\e^\eta)|^2\,dx \le 0,
\end{split}
\]
which implies our desired result.
\end{proof}

Before we proceed to the proof of the existence of a weak solution to system \eqref{main_eq}, we need two technical lemmas concerning the interaction potential energy.
\begin{lemma}\label{L3.5}
For $d=2,3$, suppose that $\varrho \in L^\infty(0,T;L^1(\T^d))$. Then we have
\[
\int_{\T^d} (K^\e\star\varrho) \varrho\,dx \ge -C\|\varrho\|_{L^\infty(0,T;L^1(\T^d))}\lt(1+\|\varrho\|_{L^\infty(0,T;L^1(\T^d))}\rt),
\]
where $C$ is a positive constant satisfying $C = \mathcal{O}(1)$ as $\e \to 0$.
\end{lemma}
\begin{proof} When $d=3$, it is obvious since
\[
\iint_{\T^3 \times \T^3}K^\e(x-y)\varrho(x)\varrho(y)\,dxdy \ge \iint_{\T^3 \times \T^3} G_0(x-y)\varrho(x)\varrho(y)\,dxdy.
\]
Thus, it suffices to show the case $d=2$. First, we use the following inequality:
\begin{align*}
\e + |x-y|^2 &\le (1+\e)(1+|x-y|^2)\le (1+\e)(1+2|x|^2 + 2|y|^2)\le 2(1+\e)(1+|x|^2)(1+|y|^2), 
\end{align*}
which subsequently gives
\[
\log(\e + |x-y|^2)\le \log \lt(2(1+\e)\rt) + \log(1+|x|^2) + \log(1+|y|^2).
\]
We combine the previous inequality with $\log(1+x) \le x$ on $x\ge 0$  to get
\begin{align*}
&\iint_{\T^2 \times \T^2}K^\e(x-y)\varrho(x)\varrho(y)\,dxdy\\
&\quad =-\frac{c_1}{2} \iint_{\T^2\times\T^2} (\log(\e+|x-y|^2)+G_1(x-y))\varrho(x)\varrho(y)\,dxdy\\
&\quad \ge -\frac{c_1}{2}\iint_{\T^2\times\T^2} \left[\log \lt(2(1+\e)\rt)  + \log(1+|x|^2) + \log(1+|y|^2)+G_1(x-y)\right]\varrho(x)\varrho(y)\,dxdy\\
&\quad \ge -\frac{c_1}{2}(\log 2(1+\e)+\|G_1\|_{L^\infty})\|\varrho\|_{L^\infty(0,T;L^1)}^2  -\frac{c_1}{2}\iint_{\T^2\times\T^2}\left(\log(1+|x|^2) + \log(1+|y|^2)\right)\varrho(x)\varrho(y)\,dxdy\\
&\quad \ge -\frac{c_1}{2}(\log \lt(2(1+\e)\rt) +\|G_1\|_{L^\infty})\|\varrho\|_{L^\infty(0,T;L^1)}^2- c_1\|\varrho\|_{L^\infty(0,T;L^1)} \int_{\T^2}\varrho \log(1+|x|^2)\,dx\\
&\quad \ge -\frac{c_1}{2}(\log \lt(2(1+\e)\rt) +\|G_1\|_{L^\infty})\|\varrho\|_{L^\infty(0,T;L^1)}^2- c_1\|\varrho\|_{L^\infty(0,T;L^1)} \int_{\T^2}|x|^2\varrho \,dx.
\end{align*}
This completes the proof.

\end{proof}

\begin{lemma}\label{L3.7}
For $d=2,3$, suppose that a sequence $\{\varrho^n\}$  is uniformly bounded in $L^\infty(0,T;L^p(\T^d))$ for each $p \in [1,(d+2)/d)$ and $\{\varrho^n\}$ converges to $\varrho$ almost everywhere. Then  we have
\[
\lim_{n \to \infty} \int_{\T^d} (K^\e\star\varrho^n)\varrho^n \,dx = \int_{\T^d} (K^\e\star\varrho)\varrho\,dx.
\]
Moreover, if $\e \to 0$ as $n \to \infty$, then we have
\[
\lim_{n \to \infty} \int_{\T^d} (K^\e\star\varrho^n)\varrho^n \,dx = \int_{\T^d} (K\star\varrho)\varrho\,dx.
\]
\end{lemma}
\begin{proof} Since the proof for $d=3$ is analogous, we only consider the case $d=2$. First, we recall a partial result from \cite[Theorem 5.2]{CCJpre}.
\begin{lemma}\label{L3.6}
A sequence $\{h_n\}$ in $L^1(\T^d)$ converges to $h\in L^1(\T^d)$ in $L^1(\T^d)$ if the following three conditions hold:
\begin{enumerate}
\item[(i)] $h_n$ converges to $h$ almost everywhere.
\item[(ii)] for every $\e>0$, there exists $\delta = \delta(\e)>0$ such that whenever $m(E)<\delta$,
\[
\sup_{n \in \bbn} \int_E |h_n| \,dx <\e.
\]
\end{enumerate}
\end{lemma}
Without loss of generality, we may assume 
\[
\sup_{n\in \bbn}\|\varrho^n\|_{L^\infty(0,T;L^1(\T^2))} \le 1.
\] 
First, we show the pointwise convergence of $K^\e \star \varrho^n$ to $K^\e \star \varrho$. Since $\varrho^n$ converges  to $\varrho$ pointwise, $K^\e(x-\cdot)\varrho^n(\cdot)$ also pointwisely converges to $K^\e(x-\cdot)\varrho(\cdot)$ for each $x$. When $d=2$, we use the fact
\[
|\log x| \le \max\lt\{x, x^{-1}\rt\} \quad \forall \, x>0,
\] 
and choose  a measurable set $E$  and $p \in (1,2)$ to obtain
\[\begin{split}
&\int_E K^\e(x-y)\varrho^n(y)\,dy \\
&\quad\le 2c_1\int_E \max\lt\{(\e + |x-y|^2)^{1/4}, (\e + |x-y|^2)^{-1/4}\rt\}  \varrho^n(y)\,dy + \int_E G_1(x-y)\varrho^n(y)\,dy\\
&\quad\le C\int_{E\cap \{y\in \T^2\,:\,|x-y|\ge 1-\e\}} \sqrt{\e + |x-y|^2} \varrho^n(y)\,dy + C\int_{E\cap \{y\in \T^2\,:\,|x-y|< 1-\e\}}(\e + |x-y|^2)^{-1/4} \varrho^n(y)\,dy \\
&\qquad + \|G_1\|_{L^\infty} \|\varrho^n\|_{L^p} m(E)^{1/p'}\\
&\quad\le  C\lt(|x|+\sqrt{\e}\rt)m(E)^{1/(2p')} + \| |\cdot|^{-1/2}\mathds{1}_{\{|\cdot|\le 1\}}\|_{L^{7/2}}\|\varrho^n\|_{L^{7/4}}m(E)^{1/7} + \|G_1\|_{L^\infty} \|\varrho^n\|_{L^p} m(E)^{1/p'},
\end{split}\]
where $p'$ is the H\"older conjugate of $p$. This guarantees the condition (ii) in Lemma \ref{L3.6}. Thus, Lemma \ref{L3.6} gives the pointwise convergence of $K^\e \star \varrho^n$ to $K^\e \star \varrho$, and subsequently this asserts $(K^\e \star \varrho^n)\varrho^n$ converges to $(K^\e \star \varrho)\varrho$ almost everywhere.  To prove the desired $L^1$-convergence, we notice from the uniform boundedness of $\varrho^n$ in $L^\infty(0,T;L^p(\T^2))$ for $p \in [1,2)$ that
\[
|K^\e \star \varrho^n | \le C(1+|x|),
\]
where $C$ is a constant independent of $n$ and $C = \mathcal{O}(1)$ as $\e \to 0$. Indeed, we have
\[\begin{split}
&\int_{\T^2} K^\e(x-y)\varrho^n(y)\,dy \\
&\quad\le 2c_1\int_{\T^2} \max\lt\{(\e + |x-y|^2)^{1/4}, (\e + |x-y|^2)^{-1/4}\rt\} \varrho^n(y)\,dy + \int_{\T^2} G_1(x-y)\varrho^n(y)\,dy\\
&\quad\le C\int_{\{y\in \T^2\,:\,|x-y|\ge 1-\e\}} \sqrt{\e + |x-y|^2} \varrho^n(y)\,dy + C\int_{ \{y\in \T^2\,:\,|x-y|< 1-\e\}}(\e + |x-y|^2)^{-1/4} \varrho^n(y)\,dy  + \|G_1\|_{L^\infty}\\
&\quad\le  C(1+\sqrt\e + |x|) + \int_{ \{y\in \T^2\,:\,|x-y|< 1\}}|x-y|^{-1/2} \varrho^n(y)\,dy\\
&\quad\le C(1+|x|) + \|| \cdot|^{-1/2} \mathds{1}_{\{|\cdot | <1\}} \|_{L^{7/2}} \|\varrho^n\|_{L^{7/5}} \cr
&\quad \le C(1+|x|).
\end{split}\]
Thus, for a measurable set $E \subset \T^2$ we use the above estimate to show
\[\begin{split}
\int_E (K^\e \star\varrho^n)\varrho^n \,dx &\le C\int_E (1+|x|)\varrho^n\,dx \le  C \|\varrho^n\|_{L^p} m(E)^{1/p'},
\end{split}\]
where $p \in (1,2)$. This together with Lemma \ref{L3.6} yields
\[
\lim_{n \to \infty}\int_E (K^\e \star\varrho^n)\varrho^n \,dx = \int_E (K^\e \star\varrho)\varrho \,dx.
\]
Since the estimates can be made independently of $\e$, the similar analysis also gives
\[
\lim_{n \to \infty}\int_E (K^\e \star\varrho^n)\varrho^n \,dx = \int_E (K\star\varrho)\varrho \,dx,
\]
if $\e \to 0$ as $n \to 0$. This concludes the proof of Lemma \ref{L3.7}.

\end{proof}

\subsection{Existence of a weak solution to VPNS} We are ready to present the existence of a weak solution to the system \eqref{main_eq}. 

\subsubsection{Convergence $\lambda \to \infty$}
We will show the convergences of regularized solutions $(f^{\lambda,\e}, v^{\lambda,\e})$ as $\lambda$ tends to infinity. \\

\noindent $\bullet$ (Step A: Uniform bound estimates) First, we attain an upper bound estimate which is uniform in $\lambda$ using the entropy inequality from Corollary \ref{C3.2}. Here, we recall that
\[
\iint_{\T^d \times \R^d} f^{\lambda,\e} \log_{-}f^{\lambda,\e} \,dx d\xi  \le \frac{1}{4\sigma}\iint_{\T^d \times \R^d} f^{\lambda,\e} \left( 1 + |\xi|^2 \right) dx d\xi  + C(\sigma),
\]
where $\log_{-}g(x) := \max\{0, -\log g(x)\}$. We apply the above inequality  and Lemma \ref{L3.5} to Corollary \ref{C3.2} and obtain
\begin{align*}
&\iint_{\T^d \times \R^d} \left(\frac{|\xi|^2}{4} +\sigma |\log f^{\lambda,\e}| \right) f^{\lambda,\e} \,dxd\xi + \frac{1}{2}\int_{\T^d} |v^{\lambda,\e}|^2\,dx + \int_0^t \int_{\T^d} |\nabla v^{\lambda,\e}|^2\,dxds\\
& \quad+ \int_0^t\iint_{\T^d \times \R^d}\left(\frac{1}{f^{\lambda,\e}} \left|\sigma \nabla_\xi f^{\lambda,\e} - (\chi_\lambda(u_\e^{\lambda,\e})-\xi)f^{\lambda,\e}\right|^2 + |\chi_\lambda(v^{\lambda,\e})-\xi|^2f^{\lambda,\e} \right) dxd\xi ds \\
&\qquad \le \iint_{\T^d \times \R^d}\left(\frac{|\xi|^2}{2} + \sigma |\log f_0^\lambda| \right) f_0^\lambda \,dxd\xi +  \iint_{\T^d \times \T^d}  \lt|(\rho_0^\lambda-1)K^\e\star(\rho_0^\lambda-1) \rt|\,dx + \frac12\int_{\T^d} |v_0^\e|^2\,dx \\
&\quad \qquad +\sigma d t\|f_0^\lambda\|_{L^1} + C,
\end{align*}
where $C=C(T,\sigma)$ is a positive constant independent of $\lambda$ and $\e$, and it vanishes when $\sigma = 0$. Here, note that
\[
\int_{\T^d}  \lt|(\rho_0^\lambda-1)K^\e\star(\rho_0^\lambda-1) \rt|\,dx \le \|K^\e\|_{L^q} \|\rho_0^\lambda -1\|_{L^p}^2,
\]
where $p$ and $q$ satisfies $2/p + 1/q = 2$ and we used Young's convolution inequality:
\[
\int_{\T^d} (f\star g)h \,dx \le \|f\|_{L^p}\|g\|_{L^q}\|h\|_{L^r}, \quad \frac{1}{p} + \frac{1}{q} + \frac{1}{r} = 2.
\]Since the initial data $f_0^\lambda$ satisfy \eqref{bdd_conds} with $k^* = 2$ uniformly in $\eta$, we can get
\[
\|\rho_0^\lambda -1\|_{L^p} \le C \quad \forall \,p \in [1,(d+2)/d) 
\]
uniformly in $\lambda$, and this requires $\|K^\e\|_{L^q} <\infty$ uniformly in $\e$ for some $q>(d+2)/4$. Since $d=2$ or $3$, we can choose such $q$ and hence, we combine this argument with $\|f_0^\lambda\|_{L^1} \le \|f_0\|_{L^1}$ and Lemma \ref{L3.5} and apply Gr\"onwall's lemma to yield, for $t \in [0,T]$,
\bq\label{est_ee}
\iint_{\T^d \times \R^d} |\xi|^2 f^{\lambda,\e} \,dxd\xi  + \int_{\T^d} |v^{\lambda,\e}|^2\,dx + \int_0^t \int_{\T^d} |\nabla v^{\lambda,\e}|^2\,dxds \le C.
\eq
Here $C=C(T,\sigma)$ is a positive constant independent of $\lambda$ and $\e$. Once we recall that 
\[
\|f^{\lambda,\e}(\cdot,\cdot,t)\|_{L^p}^p + \frac{4\sigma(p-1)}{p}\int_0^t e^{2d(p-1)(t-s)}\|\nabla_\xi (f^{\lambda,\e})^{p/2}(\cdot,\cdot,s)\|_{L^2}^2\,ds  \le \|f^\lambda_0\|_{L^p}^pe^{2d(p-1)t}
\]
for every $p \in [1,\infty]$, we obtain the following uniform estimates:
\[
\|f^{\lambda,\e}\|_{L^\infty(0,T;L^p(\T^d\times\R^d))} + \|\rho^{\lambda,\e}\|_{L^\infty(0,T;L^{q_1}(\T^d))} + \|\rho^{\lambda,\e} u^{\lambda,\e}\|_{L^\infty(0,T;L^{q_2}(\T^d))} \le C(T),
\]
where $p \in [1,\infty]$, $q_1 \in [1,(d+2)/d)$, $q_2 \in [1,(d+2)/(d+1))$ and $C=C(T)$ is a positive constant independent of $\lambda$ and $\e$. Consequently, we can obtain the following weak convergence as $\lambda\to\infty$ up to a subsequence:
\[
\begin{array}{lcll}
f^{\lambda,\e} \rightharpoonup f^\e  &\mbox{ in }&  L^\infty(0,T;L^p(\T^d\times\R^d)), & p\in[1,\infty],\\
\displaystyle \rho^{\lambda,\e} \rightharpoonup \rho^\e & \mbox{ in } & L^\infty(0,T;L^p(\T^d)), & p\in[1,(d+2)/d),\\
\displaystyle \rho^{\lambda,\e} u^{\lambda,\e} \rightharpoonup \rho^\e u^\e & \mbox{ in } & L^\infty(0,T;L^p(\T^d)), & p\in [1,(d+2)/(d+1)),\\
\displaystyle v^{\lambda,\e} \rightharpoonup v^\e & \mbox{ in } & L^\infty(0,T;L^2(\T^d))\cap L^2(0,T;H^1(\T^d)).
\end{array}
\]
For the proof, we need to obtain the strong convergence of the averaged quantities and fluid velocity. Let $p \in (1,(d+2)/(d+1))$ and write $G^{\lambda,\e}$ as
$$
G^{\lambda,\e} := \sigma \nabla_\xi f^{\lambda,\e} + (2\xi+ \nabla K^\e \star (\rho^{\lambda,\e}-1) -\chi_\lambda(u_\e^{\lambda,\e})-\chi_\lambda(v^{\lambda,\e}))f^{\lambda,\e},
$$
then we can easily check that $G^{\lambda,\e} \in L_{loc}^p(\T^d \times \R^d \times (0,T) )$. Thus, we set $G^{\lambda,\e}$ as above, $r=2$ in  Lemma \ref{L3.3} to get, for $p \in (1, (d+2)/(d+1))$,
\[
\begin{array}{lcl}
\displaystyle \rho^{\lambda,\e} \to \rho^\e & \mbox{ in } & L^p(\T^d \times (0,T)) \ \mbox{ and a.e.},\\
\displaystyle \rho^{\lambda,\e} u^{\lambda,\e} \to \rho^\e u^\e & \mbox{ in } & L^p(\T^d \times (0,T))
\end{array}
\]
as $\lambda\to\infty$, up to a subsequence. For the fluid velocity, we first obtain from \eqref{est_ee} that
\bq\label{bdd_v}
\|v^{\lambda,\e}\|_{L^\infty(0,T;L^2)} + \|\nabla v^{\lambda,\e}\|_{L^2(0,T;L^2)} \leq C,
\eq
where $C>0$ is independent of $\lambda$ and $\e$. 

We next show that $\|\pa_t v^{\lambda,\e}\|_{L^q(0,T;\mathsf{V}')} \leq C$ for some $q \in (1,\infty)$. For $p \in (1,(d+2)/(d+1))$ and $\varphi\in\mathsf{V}$, we get
\begin{align*}
&\left| \int_0^T \int_{\T^d} \partial_t v^{\lambda,\e} \cdot \varphi \,dxds \right| \\
&\quad \le \|\theta_\e \star v^{\lambda,\e}\|_{L^3(\T^d \times (0,T))} \|\nabla v^{\lambda,\e}\|_{L^2(\T^d \times (0,T))} \|\varphi\|_{L^6(\T^d \times (0,T))} + \|\nabla v^{\lambda,\e}\|_{L^2(\T^d \times (0,T))}\|\nabla \varphi\|_{L^2(\T^d \times (0,T))}  \\
&\qquad + \|\rho^{\lambda,\e} u^{\lambda,\e}\|_{L^\infty(0,T;L^p(\T^d))}\|\varphi\|_{L^1(0,T;L^{p'}(\T^d))} + \|\rho^{\lambda,\e} v^{\lambda,\e}\|_{L^2(0,T;L^p(\T^d))}\|\varphi\|_{L^2(0,T;L^{p'}(\T^d))},
\end{align*}
where $p'$ is the H\"older conjugate of $p$. Here, we use Gagliardo--Nirenberg interpolation inequality to get
\[
\|v^{\lambda,\e}\|_{L^3(\T^d)} \le \left\{\begin{array}{lcl} C\|\nabla v^{\lambda,\e}\|_{L^2}^{1/3} \|v^{\lambda,\e}\|_{L^2}^{2/3} & \mbox{ if } & d=2,\\[2mm]
C\|\nabla v^{\lambda,\e}\|_{L^2}^{1/2}\|v^{\lambda,\e}\|_{L^2}^{1/2} & \mbox{ if } & d=3,\end{array}\right.
\]
where $C$ only depends on $d$. Thus, if $d=2$,
\[
\|\theta_\e \star v^{\lambda,\e}\|_{L^3(\T^2 \times (0,T))} \le \|v^{\lambda,\e}\|_{L^3(\T^2 \times (0,T))} \le C\|v^{\lambda,\e}\|_{L^4(0,T;L^2(\T^2))}^{2/3} \|\nabla v^{\lambda,\e}\|_{L^2(0,T;L^2(\T^2))}^{1/3}
\]
and if $d=3$,
\[
\|\theta_\e \star v^{\lambda,\e}\|_{L^3(\T^3\times (0,T))} \le \|v^{\lambda,\e}\|_{L^3(\T^3 \times (0,T))} \le C\|v^{\lambda,\e}\|_{L^6(0,T;L^2(\T^3))}^{1/2} \|\nabla v^{\lambda,\e}\|_{L^2(0,T;L^2(\T^3))}^{1/2},
\]
where $C$ only depends on $d$. On the other hand, the local moment can be estimated as
\begin{align*}
\|\rho^{\lambda,\e} v^{\lambda,\e}\|_{L^2(0,T;L^p(\T^d))} &\le \|\rho^{\lambda,\e}\|_{L^2(0,T;L^{\frac{(d+2)p}{d+2-p}}(\T^d))}\|v^{\lambda,\e}\|_{L^2(0,T;L^{d+2}(\T^d))}\\
&\le \|\rho^{\lambda,\e}\|_{L^2(0,T;L^{\frac{(d+2)p}{d+2-p}}(\T^d))}\|v^{\lambda,\e}\|_{L^2(0,T;H^1(\T^d))},
\end{align*}
which is finite since $p\in (1, (d+2)/(d+1))$ implies $(d+2)p/(d+2-p) < (d+2)/d$. Thus, taking $p = (d+3)/(d+2)$ gives us $\|\varphi\|_{L^{p'}} \leq C \|\varphi\|_{H^1}$, and thus we obtain
\[
\partial_t v^{\lambda,\e} \in L^{\frac{d+3}{d+2}}(0,T;\mathsf{V}'), \quad \mbox{uniformly in } \lambda \mbox{ and } \e.
\] 
Combining this with \eqref{bdd_v}, we can use Aubin--Lions lemma to get
\[
v^{\lambda,\e} \to v^\e \qquad \mbox{in} \quad L^2(\T^d \times (0,T)) \ \mbox{ and a.e.}, 
\]
up to a subsequence.\\

\noindent $\bullet$ (Step B: Existence of weak solutions and entropy inequality) Now, we show that the limit $(f^\e, v^\e)$ satisfies the following system in distributional sense:
\begin{equation}\label{C-4}
\begin{split}
&\pa_t f^\e + v\cdot\nabla f^\e -  \nabla_\xi \cdot \lt((\xi -v^\e + \nabla K^\e \star (\rho^\e-1))f^\e \rt)=  \nabla_\xi \cdot ( \sigma \nabla_\xi f^\e - (  u_\e^\e-\xi)f^\e),\cr
&\pa_t v^\e + ((\theta_\e \star v^\e)\cdot \nabla) v^\e + \nabla p^\e -\Delta v^\e = -\rho^\e(v^\e-u^\e),\\
&\nabla \cdot v^\e = 0,
\end{split}
\end{equation}
and the entropy inequality:
\begin{align}\label{C-5}
\begin{aligned}
&\iint_{\T^d \times \R^d} \left(\frac{|\xi|^2}{2} + \sigma \log f^\e \right) f^\e \,dxd\xi +  \int_{\T^d} (\rho^\e-1)K^\e\star(\rho^\e-1) \,dx + \frac12 \int_{\T^d} |v^\e|^2 \,dx\\
&\quad + \int_0^t\iint_{\T^d \times \R^d} \left(\frac{1}{f^\e} \left|\sigma \nabla_\xi f^\e - (u_\e^\e-\xi)f^\e\right|^2 + |v^\e-\xi|^2 f^\e\right) dxd\xi ds + \int_0^t \int_{\T^d} |\nabla v^\e|^2  \,dxds\\
&\qquad  \le \iint_{\T^d \times \R^d} \left(\frac{|\xi|^2}{2} + \sigma \log f_0 \right) f_0 \,dxd\xi + \int_{\T^d} (\rho_0 -1)K^\e\star(\rho_0 -1) \,dx  + \frac12 \int_{\T^d} |v_0^\e|^2 \,dx + \sigma d t.
\end{aligned}
\end{align}
For the weak solution $(f^\e,v^\e)$ to \eqref{C-4}, it suffices to show the following convergence in distributional sense since the others are linear:
\begin{equation}\label{C-6}
\begin{cases}
\mbox{(i)}~~(\nabla K^\e \star \rho^{\lambda,\e})f^{\lambda,\e} \to (\nabla K^\e \star \rho^\e)f^\e,\\[2mm]
\mbox{(ii)}~~ \chi_\lambda(u_\e^{\lambda,\e}) f^{\lambda,\e} \to u_\e^\e f^\e.\\
\end{cases}
\end{equation}
Although the proof is similar to that in \cite{CCJpre}, we present the detail for the completeness of our work. For the proof, we arbitrarily choose $\Phi\in \mathcal{C}_c^\infty(\T^d \times\R^d\times [0,T]) $ with $\Phi(\cdot, \cdot ,T)=0$.\\

\noindent $\diamond$ (Step B-1: Convergence of \eqref{C-6} (i))  We write 
\[
\rho^{\lambda,\e}_\Phi(x,t) := \int_{\R^d} f^{\lambda,\e}(x,\xi,t)\Phi(x,\xi,t)\,d\xi.
\]
Then, we estimate
\begin{align*}
\int_0^T& \iint_{\T^d \times \R^d} \left[(\nabla K^\e \star \rho^{\lambda,\e})f^{\lambda,\e} - (\nabla K^\e \star \rho^\e)f^\e\right]\Phi\,dxd\xi dt\\
&=\int_0^T \iint_{\T^d \times \R^d} \nabla K^\e \star (\rho^{\lambda,\e} -\rho^\e)\rho_\Phi^{\lambda,\e}\,dxdt + \int_0^T \iint_{\T^d \times \R^d} (\nabla K^\e \star\rho^\e )(f^{\lambda,\e}-f^\e) \Phi\,dxd\xi dt\\
&=: \mathcal{I}^{\lambda,\e}_1 + \mathcal{I}^{\lambda,\e}_2.
\end{align*}
For $\mathcal{I}^{\lambda,\e}_1$, we use the uniform bound of  $f^{\lambda,\e}$ in $L^\infty(\T^d \times\R^d\times (0,T))$ and the compact support of $\Phi$ to yield
\[
\rho_\Phi^{\lambda,\e} \in L^\infty(0,T; L^p(\T^d)) \quad \mbox{for any } \ p \in [1,\infty] 
\]
uniformly in $\lambda$ and $\e$. This gives
\begin{align*}
\mathcal{I}^{\lambda,\e}_1 &\le \|\nabla K^\e\|_{L^1(\T^d)} \|\rho^{\lambda,\e} - \rho^\e\|_{L^p(\T^d \times (0,T))} \|\rho_\Phi^{\lambda,\e}\|_{L^{p'}(\T^d \times (0,T))},
\end{align*}
where $p \in (1,(d+2)/(d+1))$, $p'$ is the H\"older conjugate of $p$, and we used Young's convolution inequality. Since  $|\nabla K^\e(x)| \le C(1+ |x|^{-d+1})$, we  use the strong convergence of $\rho^{\lambda,\e}$ to get $\mathcal{I}^{\lambda,\e}_1 \to 0$ as $\lambda \to \infty$. For $\mathcal{I}^{\lambda,\e}_2$, we notice 
\[
(\nabla K^\e \star \rho^\e) \Phi \in L^1(0,T;L^p(\T^d)) \ \mbox{ for some } \ p \in (1,\infty). 
\]
Thus, the weak convergence of $f^{\lambda,\e}$ asserts that $\mathcal{I}^{\lambda,\e}_2 \to 0$ as $\lambda \to \infty$.\\

\noindent $\diamond$ (Step B-2: Convergence of \eqref{C-6} (ii)) We estimate
\begin{align*}
\int_0^T& \iint_{\T^d \times \R^d} (\chi_\lambda(u_\e^{\lambda,\e})f^{\lambda,\e} - u_\e^\e f^\e)\Phi \,dxd\xi dt\\
&=\int_0^T \int_{\T^d} (u_\e^{\lambda,\e} - u_\e^\e)\rho_\Phi^{\lambda,\e} \,dxds + \int_0^T \iint_{\T^d \times \R^d} u_\e^\e (f^{\lambda,\e}-f^\e) \Phi \,dxd\xi dt \cr
&\quad + \int_0^T \iint_{\T^d \times \R^d} u_\e^{\lambda,\e} f^{\lambda,\e}\mathds{1}_{\{|u_\e^{\lambda,\e}|>\lambda \}} \Phi \,dxd\xi dt\\
&=: \mathcal{J}^{\lambda,\e}_1 + \mathcal{J}^{\lambda,\e}_2 + \mathcal{J}^{\lambda,\e}_3.
\end{align*}
For $\mathcal{J}^{\lambda,\e}_1$, we find
\[
\mathcal{J}^{\lambda,\e}_1 = \int_0^T \int_{\T^d} \left(\left(\frac{1}{\rho^{\lambda,\e} + \e} - \frac{1}{\rho^\e + \e}\right)(\rho^\e u^\e) + \frac{1}{\rho^{\lambda,\e} +\e} (\rho^{\lambda,\e} u^{\lambda,\e} - \rho^\e u^\e) \right)\rho_\Phi^{\lambda,\e}\,dxdt.
\]
Since $\rho^{\lambda,\e} \to \rho^\e$ a.e. and $\rho_\Phi^{\lambda,\e} \in L^\infty(\T^d \times (0,T))$ uniformly in $\lambda$ and $\e$, we get 
\[
\left(\frac{1}{\rho^{\lambda,\e} + \e} - \frac{1}{\rho^\e + \e}\right)(\rho^\e u^\e)\rho_\Psi^{\lambda,\e} \to 0 \quad \mbox{a.e.}
\]
as $\lambda \to \infty$, and
\[
\left|\left(\frac{1}{\rho^{\lambda,\e} + \e} - \frac{1}{\rho^\e + \e}\right)(\rho^\e u^\e)\rho_\Phi^{\lambda,\e}\right| \le \frac{2\|\rho_\Phi^{\lambda,\e}\|_{L^\infty}}{\e}|\rho^\e u^\e| \le C(\e) |\rho^\e u^\e|,
\]
where $C = C(\e)$ is a positive constant independent of $\lambda$. Then, we use the dominated convergence theorem to get
\[
\int_0^T \int_{\T^d} \left(\frac{1}{\rho^{\lambda,\e} + \e} - \frac{1}{\rho^\e + \e}\right)(\rho^\e u^\e)\rho_\Phi^{\lambda,\e} \,dxdt \to 0
\] 
as $\lambda \to \infty$. Moreover, we obtain
\[
\int_0^T \int_{\T^d}  \frac{1}{\rho^{\lambda,\e} +\e} (\rho^{\lambda,\e} u^{\lambda,\e} - \rho^\e u^\e) \rho_\Phi^{\lambda,\e}\,dxdt\le \frac{1}{\e}\|\rho^{\lambda,\e} u^{\lambda,\e} -\rho^\e u^\e\|_{L^p} \|\rho_\Phi^{\lambda,\e}\|_{L^{p'}},
\]
where $p \in (1,(d+2)/(d+1))$, and this implies $\mathcal{J}^{\lambda,\e}_1 \to 0$ as $\lambda \to \infty$. For $\mathcal{J}^{\lambda,\e}_2$, it is clear that $u_\e^\e \Phi \in L^1(0,T;L^p(\T^d\times\R^d))$ for some $p\in(1,\infty)$. Thus, the weak convergence of $f^{\lambda,\e}$ gives $\mathcal{J}^{\lambda,\e}_2 \to 0$  as $\lambda \to \infty$. Finally for $\mathcal{J}^{\lambda,\e}_3$, we use 
\[
|u^{\lambda,\e}| \le \left(\frac{\int_{\R^d} |\xi|^2 f^{\lambda,\e} \,d\xi}{\rho^{\lambda,\e}}\right)^{1/2}
\] 
to get
\[
\mathcal{J}^{\lambda,\e}_3 \le \frac{1}{\lambda} \int_0^T \iint_{\T^d \times \R^d}|u_\e^{\lambda,\e}|^2 f^{\lambda,\e} \Phi\,dxd\xi dt\le \frac{\|\Phi\|_{L^\infty}}{\lambda} \int_0^T\iint_{\T^d \times \R^d} |\xi|^2 f^{\lambda,\e} \,dxd\xi dt \to 0
\]
as $\lambda \to \infty$. This concludes that $(f^\e,v^\e)$ is a weak solution to the system \eqref{C-4}.\\

\noindent $\diamond$ (Step B-3: Entropy inequality) For the entropy inequality \eqref{C-5}, we first take the liminf on the left hand side of Corollary \ref{C3.2}, convexity of the entropy and use $\|f_0^\lambda \|_{L^1} \le \|f_0\|_{L^1}=1$ and Lemma \ref{L3.7} to get
\[
\begin{split}
&\iint_{\T^d \times \R^d} \left(\frac{|\xi|^2}{2} +  \sigma \log f^\e \right) f^\e \,dxd\xi +  \int_{\T^d}  (\rho^\e-1)K^\e\star(\rho^\e-1) \,dx + \frac{1}{2}\int_{\T^d} |v^\e|^2\,dx\\
&\quad + \int_0^t\iint_{\T^d \times \R^d}\left(\frac{1}{f^\e} \left|\sigma \nabla_\xi f^\e - (u_\e^\e - \xi)\right|^2  + |v^\e-\xi|^2 f^\e\right) dxd\xi ds + \int_0^t \int_{\T^d}|\nabla v^\e |^2 \,dxds\\
&\qquad \le \liminf_{\lambda \to \infty}\left(\iint_{\T^d \times \R^d} \left(\frac{|\xi|^2}{2} +  \sigma \log f_0^\lambda \right) f_0^\lambda \,dxd\xi +  \int_{\T^d}  (\rho_0^\lambda-1)K^\e\star(\rho_0^\lambda-1) \,dx + \frac{1}{2}\int_{\T^d} |v_0^\e|^2\,dx\right) +\sigma dt.
\end{split}
\] 
We employ the reverse Fatou's lemma and almost everywhere convergences $\rho_0^\lambda \to \rho_0$ and $f_0^\lambda \to f_0$ as $\lambda \to \infty$ to have
\begin{align*}
&\liminf_{\lambda \to \infty}\left(\iint_{\T^d \times \R^d} \left(\frac{|\xi|^2}{2} +  \sigma \log f_0^\lambda \right) f_0^\lambda \,dxd\xi + \int_{\T^d} (\rho_0^\lambda-1)K^\e\star(\rho_0^\lambda-1) \,dx +\frac{1}{2}\int_{\T^d} |v_0^\e|^2\,dx \right)\\
&\quad \le\limsup_{\lambda \to \infty}\left(\iint_{\T^d \times \R^d} \left(\frac{|\xi|^2}{2} +  \sigma \log f_0^\lambda \right) f_0^\lambda \,dxd\xi +  \int_{\T^d}(\rho_0^\lambda-1)K^\e\star(\rho_0^\lambda-1) \,dx+ \frac{1}{2}\int_{\T^d} |v_0^\e|^2\,dx\right)\\
&\quad \le \left(\iint_{\T^d \times \R^d} \left(\frac{|\xi|^2}{2} +  \sigma \log f_0 \right) f_0 \,dxd\xi +  \int_{\T^d} (\rho_0-1)K^\e\star(\rho_0-1) \,dx + \frac{1}{2}\int_{\T^d} |v_0^\e|^2\,dx\right),
\end{align*}
which implies the desired estimate.

\subsubsection{Convergence $\e \to 0$} Now, it remains to show the convergence as $\e \to 0$. First, the weak convergence $f^\eta \rightharpoonup f^\e$ implies the following uniform upper bound estimate:
\begin{equation}\label{C-7}
\|f^\e\|_{L^\infty(0,T;L^p(\T^d\times\R^d))}^p + \frac{4\sigma(p-1)}{p}\int_0^T \|\nabla_\xi (f^\e)^{p/2}(\cdot,\cdot,s)\|_{L^2}^2\,ds \leq \|f_0\|_{L^p}^pe^{2d(p-1)t}
\end{equation}
for $p\in[1,\infty)$ and
\bq\label{C-72}
\|f^\e\|_{L^\infty(0,T;L^\infty(\T^d\times\R^d))}  \leq C\|f_0\|_{L^\infty}
\eq
for some $C>0$ independent of $\e>0$.  Hence, the inequalities \eqref{C-7} and \eqref{C-72} together with the entropy inequality \eqref{C-5} yield the following weak convergence up to a subsequence as $\e\to0$:
\[
\begin{array}{lcll}
f^\e \rightharpoonup f  &\mbox{ in }&  L^\infty(0,T;L^p(\T^d\times\R^d)), & p\in[1,\infty],\\
\displaystyle \rho^\e \rightharpoonup \rho & \mbox{ in } & L^\infty(0,T;L^p(\T^d)_, & p\in[1,(d+2)/d),\\
\displaystyle \rho^\e u^\e \rightharpoonup \rho u & \mbox{ in } & L^\infty(0,T;L^p(\T^d)), & p\in [1,(d+2)/(d+1)),\\
\displaystyle v^\e \rightharpoonup v & \mbox{ in } & L^\infty(0,T;L^2(\T^d))\cap L^2(0,T;H^1(\T^d)).
\end{array}
\]
Once again, we apply Lemma \ref{L3.3} to obtain the strong convergence up to a subsequence for $p \in (1,(d+2)/(d+1))$:
\[
\begin{array}{lcl}
\displaystyle \rho^\e \to \rho & \mbox{ in } & L^p(\T^d \times (0,T)) \ \mbox{ and a.e.},\\
\displaystyle \rho^\e u^\e \to \rho u & \mbox{ in } & L^p(\T^d \times (0,T))
\end{array}
\]
as $\e \to 0$. In order to show that $(f,v)$ is a weak solution of \eqref{main_eq}, it suffices to show the following convergences in distribution sense, since the other convergences can be deduced from the previous argument \cite{CCK16}:
\begin{equation}\label{C-8}
\begin{cases}
\mbox{(i)}~~(\nabla K^\e \star \rho^\e)f^\e \to (\nabla K \star \rho)f,\\[2mm]
\mbox{(ii)}~~ u_\e^\e f^\e  \to u f.
\end{cases}
\end{equation}
However, the convergence estimates for \eqref{C-8} can be found in \cite{CCJpre}, thus it remains to show that the weak solution obtained up to now satisfies the entropy inequality:
\[
\begin{split}
&\iint_{\T^d \times \R^d} \left(\frac{|\xi|^2}{2} + \sigma \log f \right) f \,dxd\xi + \int_{\T^d} (\rho-1)K\star(\rho-1) \,dx + \frac{1}{2}\int_{\T^d} |v|^2\,dx\\
&\quad + \int_0^t\iint_{\T^d \times \R^d}\left(\frac{1}{f} \left|\sigma \nabla_\xi f - (u-\xi)f\right|^2  + |v-\xi|^2 f\right) dxd\xi ds + \int_0^t\int_{\T^d}|\nabla v|^2  \,dxds\\
&\qquad  \le \iint_{\T^d \times \R^d} \left(\frac{|\xi|^2}{2} + \sigma \log f_0 \right) f_0 \,dxd\xi + \int_{\T^d} (\rho_0-1)K\star(\rho_0-1) \,dx + \frac{1}{2}\int_{\T^d} |v_0|^2\,dx +\sigma dt \|f_0\|_{L^1}.
\end{split}
\] 
With the entropy inequality \eqref{C-5}, the strong convergences of $\rho^\e$ and $\rho^\e u^\e$ thanks to the velocity averaging lemma, and Lemma \ref{L3.3}, we employ a similar argument as in the previous step to conclude the desired result.

\section{Convergence of $f^\e$ towards $M_{\rho,u}$ in $L^\infty(0,T;L^1(\T^d \times \R^d))$}\label{app_conv}
In this part, we provide the details of proof for Proposition \ref{prop_h1}. Let us first recall the relative entropy:
\[
\mh(f^\e | M_{\rho,u}) = f^\e \lt(\log f^\e - \log M_{\rho,u}\rt) + (f^\e - M_{\rho,u}).
\]
In the rest, we separately estimate the terms on the right hand side of the above equality.\\ 

A direct computation gives
\begin{align}\label{est_conv1}
\begin{aligned}
&\frac{d}{dt}\iint_{\T^d \times \R^d} f^\e \log f^\e\,dxd\xi \cr
&\quad = \iint_{\T^d \times \R^d} \nabla_\xi f^\e \cdot (v^\e - \xi - \nabla K\star(\rho^\e-1))\,dxd\xi - \frac1\e \iint_{\T^d \times \R^d} \frac{\nabla_\xi f^\e}{f^\e} \cdot \lt( \nabla_\xi f^\e - (u^\e - \xi )f^\e \rt) dxd\xi\cr
&\quad =  \iint_{\T^d \times \R^d} \nabla_\xi f^\e \cdot (v^\e - \xi)\,dxd\xi - \frac1\e \iint_{\T^d \times \R^d} \frac{1}{f^\e}| \nabla_\xi f^\e - (u^\e - \xi )f^\e|^2\,dxd\xi\cr
&\qquad - \frac1\e \iint_{\T^d \times \R^d}(u - \xi) \cdot \lt( \nabla_\xi f^\e - (u^\e - \xi )f^\e \rt) dxd\xi,
\end{aligned}
\end{align}
where we used 
\[
 \iint_{\T^d \times \R^d}(u^\e - u) \cdot \lt( \nabla_\xi f^\e - (u^\e - \xi )f^\e \rt) dxd\xi = 0
\]
for the estimate of the last term on the right hand side of the above. We next use the continuity equations of $\rho^\e$ and $\rho$ to obtain
\bq\label{est_conv2}
- \frac{d}{dt}\int_{\T^d} \rho^\e \log \rho = -\int_{\T^d} \rho^\e (u^\e-  u) \cdot \frac{\nabla \rho}{\rho}\,dx - \int_{\T^d} u \cdot \nabla \rho^\e\,dx.
\eq
We next estimate
\begin{align*}
\frac12\frac{d}{dt}\iint_{\T^d \times \R^d} |u-\xi|^2f^\e \,dxd\xi &= \iint_{\T^d \times \R^d} (u - \xi)f^\e \cdot \pa_t u\,dxd\xi + \frac12\iint_{\T^d \times \R^d} |u - \xi|^2 \pa_t f^\e \,dxd\xi =: \ml_1 + \ml_2,
\end{align*}
where we use the momentum equation in \eqref{A-3} to get
\begin{align*}
\ml_1 &=  \int_{\T^d} (u - u^\e)\rho^\e \cdot \pa_t u\,dx\cr
&=- \int_{\T^d} (u - u^\e)\rho^\e \cdot (u \cdot \nabla u + (u-v))\,dx - \int_{\T^d} \frac{\nabla \rho}{\rho} \cdot (u - u^\e) \rho^\e\,dx - \int_{\T^d} \nabla K\star(\rho-1) \cdot (u - u^\e) \rho^\e\,dx\cr
&\leq C\lt(\int_{\T^d} \rho^\e |u - u^\e|^2\,dx\rt)^{1/2} - \int_{\T^d} \frac{\nabla \rho}{\rho} \cdot (u - u^\e) \rho^\e\,dx - \int_{\T^d} \nabla K\star(\rho-1) \cdot (u - u^\e) \rho^\e\,dx.
\end{align*}
Here $C>0$ is independent of $\e > 0$. For the estimate of $\ml_2$, we use the kinetic equation in \eqref{A-2} to find
\begin{align*}
\ml_2 &= \iint_{\T^d \times \R^d} \xi f^\e \otimes (u-\xi) : \nabla u\,dxd\xi - \iint_{\T^d \times \R^d} (v^\e - \xi) \cdot (u - \xi) f^\e\,dxd\xi\cr
&\quad + \int_{\T^d} \nabla K\star(\rho^\e-1) \cdot (u - u^\e) \rho^\e\,dx + \frac1\e \iint_{\T^d \times \R^d} (u - \xi) \cdot (\nabla_\xi f^\e - (u^\e - \xi)f^\e)\,dxd\xi.
\end{align*}
On the other hand, the first term on the right hand side can be estimated as
\begin{align*}
&\iint_{\T^d  \times \R^d} \xi f^\e \otimes (u - \xi) :\nabla u\,dxd\xi\cr
&\quad = -\iint_{\T^d  \times \R^d}(u- \xi) f^\e \otimes (u - \xi) :\nabla u\,dxd\xi +\iint_{\T^d \times\R^d} u\otimes (u-\xi)f^\e : \nabla u\,dxd\xi\\
 &\quad =- \iint_{\T^d  \times \R^d} \lt( (u - u^\e)\otimes (u - u^\e) +   (u^\e - \xi) \otimes (u^\e - \xi) \rt) f^\e  : \nabla u\,dxd\xi  +\iint_{\T^d \times\R^d} u\otimes (u-\xi)f^\e : \nabla u\,dxd\xi \cr
&\quad \leq \|\nabla u\|_{L^\infty} \int_{\T^d } \rho^\e |u - u^\e|^2\,dx + \|u\|_{L^\infty}\|\nabla u\|_{L^\infty}\int_{\T^d } \rho^\e |u-u^\e|\,dx \\
&\qquad- \iint_{\T^d \times \R^d} \lt((u^\e - \xi)\sqrt{f^\e} - 2\nabla_\xi \sqrt{f^\e} \rt) \otimes (u^\e - \xi) \sqrt{f^\e} : \nabla u\,dxd\xi\cr
&\qquad - 2\iint_{\T^d \times \R^d} \nabla_\xi \sqrt{f^\e} \otimes (u^\e - \xi) \sqrt{f^\e} : \nabla u\,dxd\xi\cr
&\quad \leq C\int_{\T^d } \rho^\e |u- u^\e|^2\,dx  + C\int_{\T^d } \rho^\e|u-u^\e|\,dx- \iint_{\T^d \times \R^d} \nabla_\xi f^\e \otimes (u^\e - \xi) : \nabla u\,dxd\xi\cr
&\qquad + C\lt(\iint_{{\T^d } \times \R^d} |u^\e - \xi|^2 f^\e\,dxd\xi \rt)^{1/2}\lt(\iint_{\T^d \times \R^d} \frac{1}{f^\e} |\nabla_\xi f^\e - (u^\e - \xi)f^\e|^2\,dxd\xi \rt)^{1/2}\cr
&\quad \le C\int_{\T^d } \rho^\e |u- u^\e|^2\,dx + C\lt(\int_{\T^d} \rho^\e |u-u^\e|^2\,dx\rt)^{1/2}+\int_{\T^d } \nabla \rho^\e \cdot u\,dx + C\e\iint_{\T^d \times \R^d} |\xi|^2 f^\e\,dxd\xi\cr
&\qquad + \frac{1}{4\e}\iint_{\T^d \times \R^d} \frac{1}{f^\e} |\nabla_\xi f^\e - (u^\e - \xi)f^\e|^2\,dxd\xi,
\end{align*}
where we used
\[
\iint_{\T^d \times \R^d} |u^\e - \xi|^2 f^\e\,dxd\xi \leq \iint_{\T^d \times \R^d} |\xi|^2 f^\e\,dxd\xi - \int_{\T^d} \rho^\e|u^\e|^2\,dx \leq \iint_{\T^d \times \R^d} |\xi|^2 f^\e\,dxd\xi.
\]
Combining these estimates implies
\begin{align*}
\frac12&\frac{d}{dt}\iint_{\T^d \times \R^d} |u-\xi|^2f^\e \,dxd\xi \cr
& \leq C\lt(\min\lt\{1,\int_{\T^d} \rho^\e |u - u^\e|^2\,dx\rt\}\rt)^{1/2} - \int_{\T^d} \frac{\nabla \rho}{\rho} \cdot (u - u^\e) \rho^\e\,dx + \int_{\T^d} \nabla K\star(\rho^\e - \rho) \cdot (u - u^\e) \rho^\e\,dx\cr
&\quad - \iint_{\T^d \times \R^d} (v^\e - \xi) \cdot (u - \xi) f^\e\,dxd\xi + \frac1\e \iint_{\T^d \times \R^d} (u - \xi) \cdot (\nabla_\xi f^\e - (u^\e - \xi)f^\e)\,dxd\xi   \cr
&\quad + C\e\iint_{\T^d \times \R^d} |\xi|^2 f^\e\,dxd\xi + \frac{1}{4\e}\iint_{\T^d \times \R^d} \frac{1}{f^\e} |\nabla_\xi f^\e - (u^\e - \xi)f^\e|^2\,dxd\xi.
\end{align*}
We now combine this with \eqref{est_conv1} and \eqref{est_conv2} to have
\begin{align}\label{est_conv3}
\begin{aligned}
&\frac{d}{dt} \iint_{\T^d \times \R^d} \mh(f^\e | M_{\rho,u})\,dxd\xi + \frac1\e\iint_{\T^d \times \R^d} \frac{1}{f^\e}| \nabla_\xi f^\e - (u^\e - \xi )f^\e|^2\,dxd\xi \cr
&\quad \leq C\lt(\min\lt\{1,\int_{\T^d} \rho^\e |u - u^\e|^2\,dx\rt\}\rt)^{1/2} + \int_{\T^d} \nabla K\star(\rho^\e - \rho) \cdot (u - u^\e) \rho^\e\,dx\cr
&\qquad + C\e\iint_{\T^d \times \R^d} |\xi|^2 f^\e\,dxd\xi + \frac{1}{4\e}\iint_{\T^d \times \R^d} \frac{1}{f^\e} |\nabla_\xi f^\e - (u^\e - \xi)f^\e|^2\,dxd\xi\cr
&\qquad + \iint_{\T^d \times \R^d} \nabla_\xi f^\e \cdot (v^\e - \xi)\,dxd\xi  - \iint_{\T^d \times \R^d} (v^\e - \xi) \cdot (u - \xi) f^\e\,dxd\xi.
\end{aligned}
\end{align}
We further estimate the last two terms on the right hand side as
\begin{align*}
&\iint_{\T^d \times \R^d} \nabla_\xi f^\e \cdot (v^\e - \xi)\,dxd\xi - \iint_{\T^d \times \R^d} (v^\e - \xi) \cdot (u-\xi)f^\e\,dxd\xi\cr
&\quad  = \iint_{\T^d \times \R^d} \lt(\nabla_\xi f^\e - (u^\e - \xi)f^\e\rt) \cdot (v^\e - \xi)\,dxd\xi + \iint_{\T^d \times \R^d} (v^\e - \xi) \cdot (u^\e - u)f^\e\,dxd\xi\cr
&\quad \leq \lt(\iint_{\T^d \times \R^d} \frac{1}{f^\e} |\nabla_\xi f^\e - (u^\e - \xi)f^\e|^2 \,dxd\xi\rt)^{1/2}\lt(\iint_{\T^d \times \R^d} |v^\e - \xi|^2 f^\e\,dxd\xi \rt)^{1/2} \cr
&\qquad + \lt(\int_{\T^d} |v^\e - u^\e|^2 \rho^\e\,dx \rt)^{1/2}\lt(\int_{\T^d} |u^\e - u|^2 \rho^\e\,dx \rt)^{1/2} \cr
&\quad \leq C\e \iint_{\T^d \times \R^d} |v^\e - \xi|^2 f^\e\,dxd\xi  + \frac1{4\e}\iint_{\T^d \times \R^d} \frac{1}{f^\e} |\nabla_\xi f^\e - (u^\e - \xi)f^\e|^2 \,dxd\xi\cr
&\qquad + C\e^{1/4}\iint_{\T^d \times \R^d} |v^\e - \xi|^2 f^\e\,dxd\xi  + \frac{1}{\e^{1/4}}\int_{\T^d} |u^\e - u|^2 \rho^\e\,dx.
\end{align*}
Here we used
\[
\int_{\T^d} |v^\e - u^\e|^2 \rho^\e\,dx \leq \iint_{\T^d \times \R^d} |v^\e - \xi|^2 f^\e\,dxd\xi.
\]
For the second term on the right hand side of \eqref{est_conv3}, we recall the estimates in the proof of Proposition \ref{P2.2}: 
\begin{align*}
\frac12\frac{d}{dt}\int_{\T^d} |\nabla K\star(\rho^\e - \rho)|^2\,dx &= \int_{\T^d} \nabla K\star(\rho^\e - \rho) \cdot (\rho^\e u^\e - \rho u)\,dx\cr
&= \int_{\T^d} \nabla K\star(\rho^\e - \rho) \cdot (u^\e - u)\rho^\e\,dx + \int_{\T^d} \nabla K\star(\rho^\e - \rho) \cdot u (\rho^\e - \rho)\,dx
\end{align*}
and 
\begin{align*}
&\int_{\T^d} \nabla K\star(\rho^\e - \rho) \cdot u (\rho^\e - \rho)\,dx \cr
&\quad =-\frac{1}{2} \int_{\T^d} |\nabla K\star(\rho^\e - \rho)|^2 \nabla \cdot u \,dx +  \int_{\T^d} \nabla K\star(\rho^\e - \rho)\otimes \nabla K \star(\rho^\e -\rho) : \nabla u\,dx\\
&\quad \le C\int_{\T^d} |\nabla K\star(\rho^\e - \rho)|^2\,dx,
\end{align*}
where $C>0$ is independent of $\e > 0$. Hence we have
\begin{align*}
&\frac{d}{dt}\lt(\iint_{\T^d \times \R^d} \mh(f^\e | M_{\rho,u})\,dxd\xi +\frac12\int_{\T^d} |\nabla K\star(\rho^\e - \rho)|^2\,dx\rt)+ \frac1{2\e}\iint_{\T^d \times \R^d} \frac{1}{f^\e}| \nabla_\xi f^\e - (u^\e - \xi )f^\e|^2\,dxd\xi \cr
&\quad \leq C\lt(\min\lt\{1,\int_{\T^d} \rho^\e |u - u^\e|^2\,dx\rt\}\rt)^{1/2} +  \frac{1}{\e^{1/4}}\int_{\T^d} |u^\e - u|^2 \rho^\e\,dx  + C \int_{\T^d} |\nabla K\star(\rho^\e - \rho)|^2\,dx\cr
&\qquad + C\e\iint_{\T^d \times \R^d} |\xi|^2 f^\e\,dxd\xi + C\e \iint_{\T^d \times \R^d} |v^\e - \xi|^2 f^\e\,dxd\xi  + C\e^{1/4}\iint_{\T^d \times \R^d} |v^\e - \xi|^2 f^\e\,dxd\xi.
\end{align*}
This completes the proof.
%
%
%
%
%
%

\section{Local solvability for the EPNS system}\label{app_local}
 
In this appendix, we present the details of the proof for Theorem \ref{L5.1}.

\subsection{Solvability for the linearized system} First, we linearize the system \eqref{E-2} and investigate its local-in-time estimates. To be specific, for a given triplet
\[
(\tilde{g}, \tilde u, \tilde v) \in \mathcal{C}([0,T];H^s(\T^d)) \times \mathcal{C}([0,T];H^s(\T^d)) \times \mc([0,T];H^s(\T^d))\quad \mbox{with} \quad \nabla \cdot \tilde v = 0,
\]
we consider the following linearized system:
\begin{align}\label{AppA-1}
\begin{aligned}
&\pa_t g + \tilde{u} \cdot \nabla g   + \nabla \cdot u = 0, \quad (x,t) \in \T^d \times \R_+,\cr
&\pa_t u +  (\tilde{u} \cdot \nabla ) u + \nabla g =  - (u-\tilde{v}+ \nabla K \star (e^{\tilde{g}}-1)),\\
&\pa_t v + (\tilde{v} \cdot \nabla)v +\nabla p -\Delta v = e^{\tilde g}(\tilde{u}-v),\\
&\nabla \cdot v=0,
\end{aligned}
\end{align}
subject to the initial data $(g_0, u_0, v_0) \in H^s(\T^d) \times H^s(\T^d) \times H^s(\T^d)$.

\begin{lemma}\label{LA.1} Let $T>0$ and $s > d/2 +1$. For any positive constants $N<M$, if 
\bq\label{ini_gu}
\|g_0\|_{H^s}^2 + \|u_0\|_{H^s}^2  +\|v_0\|_{H^s}^2< N,
\eq
and
\[
\sup_{0 \le t \le T} \left( \|\tilde{g}(\cdot,t)\|_{H^s}^2 + \|\tilde{u}(\cdot,t)\|_{H^{s}}^2 + \|\tilde{v}(\cdot,t)\|_{H^s}^2\right) < M,
\]
then the system \eqref{AppA-1} admits a unique classical solution 
\[
(g,u,v) \in \mc([0,T];H^s(\T^d)) \times \mc([0,T];H^s(\T^d))\times\mc([0,T];H^s(\T^d)) 
\]
satisfying
\[
 \sup_{0 \le t \le T^*} \left( \|g(\cdot,t)\|_{H^s}^2 + \|u(\cdot,t)\|_{H^{s}}^2 + \|v(\cdot,t)\|_{H^s}^2\right) < M
\]
for some $T^* \leq T$.
\end{lemma}
\begin{proof} 
A standard theory of linear PDEs would assure the existence and uniqueness of solutions to the system \eqref{AppA-1}. Thus, it suffices to provide bound estimates for $g$, $u$, and $v$. A straightforward computation gives
\[
\frac{1}{2}\frac{d}{dt}\|g\|_{L^2}^2 \le \frac{\|\nabla \cdot \tilde{u}\|_{L^\infty}}{2}\|g\|_{L^2}^2 - \int_{\T^d} g \nabla \cdot u \,dx
\]
and
\begin{align*}
\frac12\frac{d}{dt}\|u\|_{L^2}^2 &\le \frac{\|\nabla \cdot \tilde{u}\|_{L^\infty}}{2}\|u\|_{L^2}^2  - \int_{\T^d} \nabla g \cdot u\,dx  + \|\nabla K \star (e^{\tilde{g}}-1)\|_{L^2}\|u\|_{L^2} +\|u\|_{L^2}\|\tilde{v}\|_{L^2}\\
&\le \frac{\|\nabla \cdot \tilde{u}\|_{L^\infty}}{2}\|u\|_{L^2}^2  - \int_{\T^d} \nabla g \cdot u\,dx  + \|\nabla K\|_{L^1} e^{\|\tilde{g}\|_{L^\infty}}\|\tilde g\|_{L^2}\|u\|_{L^2} +\|u\|_{L^2}\|\tilde{v}\|_{L^2}.
\end{align*}
For $v$, we have
\[
\frac12\frac{d}{dt}\|v\|_{L^2}^2 + \|\nabla v\|_{L^2}^2 \le e^{\|\tilde{g}\|_{L^\infty}} \|\tilde{u}\|_{L^2}\|v\|_{L^2}.
\]
Then we use Sobolev inequality and Young's inequality to get
\begin{equation}\label{AppA-3}
\frac{d}{dt}\left( \|g\|_{L^2}^2 +\|u\|_{L^2}^2 +\|v\|_{L^2}^2\right) \le Ce^{CM}(1+M)\left( \|g\|_{L^2}^2 +\|u\|_{L^2}^2 +\|v\|_{L^2}^2\right) +  Ce^{CM}(1+M),
\end{equation}
where $C > 0$ only depends on $s$, $d$, and $\|\nabla K\|_{L^1}$.

For $1 \le k \le s$, we first estimate $\nabla^k g$ as follows:
\begin{align*}
&\frac12\frac{d}{dt}\|\nabla^k g\|_{L^2}^2 \cr
&\quad = -\int_{\T^d} \nabla (\nabla^k g) \cdot \tilde{u} \nabla^k g \,dx-\int_{\T^d} \left( \nabla^k (\nabla g \cdot \tilde{u}) - \nabla (\nabla^k g)\cdot \tilde{u}\right) \nabla^k g \,dx -\int_{\T^d}(\nabla \cdot (\nabla^k u)) \nabla^k g \,dx\\
&\quad \le \frac{\|\nabla \cdot \tilde{u}\|_{L^\infty}}{2}\|\nabla^k g\|_{L^2}^2 +C\|\nabla^k g\|_{L^2}\Big( \|\nabla \tilde{u}\|_{L^\infty} \|\nabla(\nabla^{k-1} g)\|_{L^2} + \|\nabla g\|_{L^\infty}\|\nabla^k \tilde{u}\|_{L^2}\Big) -\int_{\T^d}(\nabla \cdot (\nabla^k u)) \nabla^k g \,dx\\
&\quad \le CM\|\nabla^k g\|_{L^2} + CM\|\nabla^k g\|_{L^2}\|g\|_{H^s} -\int_{\T^d}(\nabla \cdot (\nabla^k u)) \nabla^k g \,dx,
\end{align*}
where $C > 0$ only depends on $s$ and $d$. Similarly, $\nabla^k u$ can be estimated as
\begin{align*}
\frac12\frac{d}{dt}\|\nabla^k u\|_{L^2}^2 &= -\int_{\T^d}(\tilde{u} \cdot \nabla (\nabla^k u)) \cdot \nabla^k u \,dx - \int_{\T^d} \lt( \nabla^k (\tilde{u} \cdot \nabla u) - \tilde{u} \cdot \nabla (\nabla^k u)\rt) \cdot \nabla^k u \,dx\\
&\quad -\int_{\T^d}\nabla(\nabla^k g)\cdot \nabla^k u \,dx - \int_{\T^d}\nabla^k (\nabla K \star (e^{\tilde{g}}-1)) \cdot \nabla^k u \,dx -\int_{\T^d} \nabla^k(u-\tilde{v})\cdot\nabla^k u\,dx\\
&\le \frac{\|\nabla \cdot \tilde{u}\|_{L^\infty}}{2}\|\nabla^k u\|_{L^2} + C\|\nabla^k u\|_{L^2}\Big( \|\nabla \tilde{u}\|_{L^\infty}\|\nabla (\nabla^{k-1} u)\|_{L^2} + \|\nabla u\|_{L^\infty} \|\nabla^k \tilde{u}\|_{L^2}\Big)\\
&\quad  + \|\nabla K \star (\nabla^k e^{\tilde{g}})\|_{L^2}\|\nabla^k u\|_{L^2} -\int_{\T^d}\nabla(\nabla^k g)\cdot \nabla^k u \,dx + \|\nabla^ku\|_{L^2}\|\nabla^kv\|_{L^2}\\
&\le CM\|\nabla^k u\|_{L^2}^2 + CM\|\nabla^k u\|_{L^2}\|u\|_{H^s}+ Ce^{CM}(1+M)\|\nabla^k u\|_{L^2} \\
&\quad + C\|\nabla^k (e^{\tilde{g}}) \|_{L^2} \|\nabla^k u\|_{L^2}-\int_{\T^d}\nabla(\nabla^k g)\cdot \nabla^k u \,dx.
\end{align*}
Here $C > 0$ only depends on $d$, $s$, and $\|\nabla K\|_{L^1}$. For the estimate of the Poisson interaction term, we follow the same argument as Lemma \ref{LB.1}: we let $a_k := \|\nabla^k (e^{\tilde{g}}) \|_{L^2}$. Obviously, we have $a_0 \le e^{\|\tilde{g}\|_{L^\infty}} \le Ce^{CM}$. Then, the usage of Lemma \ref{lem_moser} and Sobolev inequality yields
\begin{align*}
a_k &=\|\nabla^{k-1} (e^{\tilde{g}} \nabla \tilde{g})\|_{L^2}\\
&\le \|e^{\tilde{g}} \nabla^k \tilde{g} \|_{L^2} + \|\nabla^{k-1}(e^{\tilde{g}}\nabla \tilde{g} ) - e^{\tilde{g}} \nabla^k \tilde{g}\|_{L^2}\\
&\le M e^{CM} + C(\|\nabla e^{\tilde{g}}\|_{L^\infty}\|\nabla^{k-1}\tilde{g}\|_{L^2} + \|\nabla^{k-1}(e^{\tilde{g}})\|_{L^2}\|\nabla \tilde{g}\|_{L^\infty})\\
&\le CMa_{k-1} + CMe^{CM}(1+M),
\end{align*}
where $C>0$ only depends on $d$ and $k$, and inductively, we get
\[
a_k \le CM^k a_0 + CM^{k-1}e^{CM}(1+M) \le Ce^{CM}.
\]
Here $C>0$ only depends on $d$ and $k$ and this gives
\[
\frac{d}{dt}\|\nabla^k u\|_{L^2}^2 \le CM\|\nabla^k u\|_{L^2}^2 + CM\|\nabla^k u\|_{L^2}\|u\|_{H^s}+ Ce^{CM}(1+M)\|\nabla^k u\|_{L^2} -\int_{\T^d}\nabla(\nabla^k g)\cdot \nabla^k u \,dx.
\]

 Finally, we estimate $\nabla ^k v$ as
\begin{align*}
\frac12\frac{d}{dt}\|\nabla^k v\|_{L^2}^2 &= -\int_{\T^d}(\tilde{v} \cdot \nabla (\nabla^k v))\cdot \nabla^k v\,dx -\int_{\T^d}\left(\nabla^k(\tilde{v}\cdot \nabla v) -\tilde{v}\cdot\nabla (\nabla^k v)\right) \cdot \nabla^k v\\
&\quad -\int_{\T^d} |\nabla (\nabla^k v)|^2\,dx -\int_{\T^d} e^{\tilde{g}}\nabla^k(v-\tilde u)\cdot\nabla^k v\,dx  -\int_{\T^d}\left(\nabla^k (e^{\tilde g} (v-\tilde{u})) -e^{\tilde{g}}\nabla^k(v-\tilde u)  \right)\cdot\nabla^k v\,dx\\
& \le C\|\nabla^k v\|_{L^2}(\|\nabla \tilde{v}\|_{L^\infty}\|\nabla (\nabla^{k-1} v)\|_{L^2} + \|\nabla v\|_{L^\infty}\|\nabla^k \tilde v\|_{L^2}) + e^{\|\tilde g\|_{L^\infty}}\|\nabla^kv\|_{L^2}\|\nabla^k \tilde u\|_{L^2}\\
&\quad + C\|\nabla^k v\|_{L^2}\left( \| e^{\tilde g} \nabla \tilde{g}\|_{L^\infty} \|\nabla^{k-1}(v-\tilde{u})\|_{L^2} + \|\nabla^k (e^{\tilde g})\|_{L^2}\|\tilde{u}-v\|_{L^\infty} \right)\\
&\le CM\|\nabla^k v\|_{L^2}^2 + Ce^{CM}(1+M)\|\nabla^k v\|_{L^2}(1 + \|v\|_{H^s}).
\end{align*}
Thus, we combine the estimates for $\nabla^k g$, $\nabla^k u$ and $\nabla^k v$ to obtain
\begin{align}
\begin{aligned}\label{AppA-4}
\frac{d}{dt}&\left( \|\nabla^k g\|_{L^2}^2 + \|\nabla^k u\|_{L^2}^2 + \|\nabla^k v\|_{L^2}^2 \right)\\
&\le CM\left( \|\nabla^k g\|_{L^2}^2 + \|\nabla^k u\|_{L^2}^2+ \|\nabla^k v\|_{L^2}^2\right) \\
&\quad +Ce^{CM}\Big( \|\nabla^k g\|_{L^2}\|g\|_{H^s} + \|\nabla^k u\|_{L^2}\|u\|_{H^s} + \|\nabla^k v\|_{L^2}\|v\|_{H^s}\Big)\\
&\quad + Ce^{CM}(1+M)(\|\nabla^k u\|_{L^2}+\|\nabla^k v\|_{L^2}).
\end{aligned}
\end{align}
Now we sum the relation \eqref{AppA-4} over $1 \le k \le s$ and combine this with zeroth-order estimate \eqref{AppA-3} to yield
\[
\frac{d}{dt}\left(\|g\|_{H^s}^2 + \|u\|_{H^s}^2+\|v\|_{H^s}^2\right)\le Ce^{CM}(1+M)\Big( \|g\|_{H^s}^2 + \|u\|_{H^s}^2 + \|v\|_{H^s}^2 \Big) +  Ce^{CM}(1+M).
\]
We set $\mathfrak{b}(M) := Ce^{CM}(1+M)$ and use Gr\"onwall's lemma to obtain
\begin{align*}
\|g\|_{H^s}^2 + \|u\|_{H^s}^2 + \|v\|_{H^s}^2 &\le (\|g_0\|_{H^s}^2 + \|u_0\|_{H^s}^2 + \|v_0\|_{H^s}^2)e^{\mathfrak{b}(M)t} + e^{\mathfrak{b}(M)t}\lt(1-e^{-\mathfrak{b}(M)t}\rt)\\
&\le Ne^{\mathfrak{b}(M)t} +e^{\mathfrak{b}(M)t}\lt(1-e^{-\mathfrak{b}(M)t}\rt)\cr
&= N + (N+1)\lt(e^{\mathfrak{b}(M)t} - 1\rt).
\end{align*}
Since $N<M$ and $e^{\mathfrak{b}(M)t} - 1$ can be arbitrary small if $t \ll 1$, we can find $T^*>0$ satisfying
\[
N + (N+1)\lt(e^{\mathfrak{b}(M)T^*} - 1\rt) < M.
\]
This asserts the desired result.
\end{proof}

\subsection{Construction of approximate solutions} Based on the estimates for the linearized system \eqref{AppA-1}, we construct a sequence of approximate solutions to \eqref{E-2} and present the uniform estimates for the sequence. Specifically, we consider a sequence $(g^n, u^n, v^n)$ which uniquely solves the following system:
\begin{align}\label{AppA-5}
\begin{aligned}
&\pa_t g^{n+1} + \nabla g^{n+1} \cdot  u^n + \nabla \cdot u^{n+1} = 0, \quad \quad (x,t) \in \T^d \times \R_+,\cr
&\pa_t u^{n+1} +  (u^n \cdot \nabla ) u^{n+1} + \nabla g^{n+1} =  - (u^{n+1}-v^n + \nabla K \star (e^{g^n}-1)),\cr
&\pa_t v^{n+1} + (v^n \cdot \nabla) v^{n+1} + \nabla p^{n+1} -\Delta v^{n+1} = e^{g^n}(u^n -v^{n+1}),\cr
&\nabla \cdot v^{n+1} = 0, 
\end{aligned}
\end{align}
with the initial step and initial data given by
\[
(g^0(x,t), u^0(x,t), v^0(x,t)) = (g_0(x), u_0(x), v_0(x)), \quad (x,t) \in \T^d \times \R_+
\]
and
\[
(g^n(x,0), u^n(x,0), v^n(x,0)) = (g_0(x), u_0(x),v_0(x)) \quad \forall \, n \in \bbn, \quad x \in \T^d,
\]
respectively. Here, Lemma \ref{LA.1} assures that the sequence $(g^n, u^n, v^n)$ is well-defined. Furthermore, Lemma \ref{LA.1}  gives the following uniform-in-$n$ bound estimates to the approximation sequence.
\begin{corollary}\label{CA.1} Let $s > d/2 + 1$. For any $M>N$, there exists $T^*>0$ such that if the initial data $(g_0, u_0,v_0)$ satisfy \eqref{ini_gu}, then for each $n \in \bbn$ 
\[
(g^n, u^n, v^n) \in \mc([0,T^*]; H^s(\T^d)) \times \mc([0,T^*]; H^s(\T^d))\times\mc([0,T^*];H^s(\T^d)),
\] 
and
\[
\sup_{0 \le t \le T^*}\left( \|g^n(\cdot,t)\|_{H^s}^2 + \|u^n(\cdot,t)\|_{H^s}^2 + \|v^n(\cdot, t)\|_{H^s}^2 \right) < M \quad \forall  \, n \in \bbn \cup \{0\}.
\]
\end{corollary}
\begin{proof} The proof employs the inductive argument. Since the initial step $(n=0)$ is trivial, it suffices to check the induction step. First, Lemma \ref{LA.1} asserts that 
\[
Ne^{\mathfrak{b}(M)T^*} + e^{\mathfrak{b}(M)T^*}(1-e^{-\mathfrak{b}(M)T^*}) < M
\]
for some $T^*>0$. Then, by the induction hypothesis, we get
\[
\sup_{0 \le t \le T^*}\left( \|g^n(\cdot,t)\|_{H^s}^2 + \|u^n(\cdot,t)\|_{H^s}^2 + \|v^n(\cdot, t)\|_{H^s}^2 \right) < M.
\]
This combined with the same estimates in Lemma \ref{LA.1} deduces
\begin{align*}
\|g^{n+1}\|_{H^s}^2 + \|u^{n+1}\|_{H^s}^2 + \|v^{n+1}\|_{H^s}^2 
&\le \lt(\|g_0\|_{H^s}^2 + \|u_0\|_{H^s}^2+\|v_0\|_{H^s}^2\rt)e^{\mathfrak{b}(M)t} + e^{\mathfrak{b}(M)t}\lt(1-e^{-\mathfrak{b}(M)t}\rt)\\
&\le Ne^{\mathfrak{b}(M)t} +e^{\mathfrak{b}(M)t}\lt(1-e^{-\mathfrak{b}(M)t}\rt) < M
\end{align*}
for $0 \leq t \le T^*$. This concludes the proof.
\end{proof}
Next, we show that the sequence $(g^n, u^n, v^n)$ is a Cauchy sequence in $\mc([0,T^*];L^2(\T^d)) \times \mc([0,T^*];L^2(\T^d))\times\mc([0,T^*];L^2(\T^d))$.

\begin{lemma}\label{LA.2} 
Let $(g^n, u^n, v^n)$ be a sequence of the approximate solutions with the initial data $(g_0, u_0, v_0)$ satisfying \eqref{ini_gu}. Then we have
\begin{align*}
&\|(g^{n+1} -g^n)(\cdot,t)\|_{L^2}^2 + \|(u^{n+1} - u^n)(\cdot,t)\|_{L^2}^2 + \|(v^{n+1} - v^n)(\cdot,t)\|_{L^2}^2 \cr
&\quad \leq  C\int_0^t \lt(\|(g^n -g^{n-1})(\cdot,s)\|_{L^2}^2 + \|(u^n - u^{n-1})(\cdot,s)\|_{L^2}^2 +  \|(v^n - v^{n-1})(\cdot,s)\|_{L^2}^2\rt) ds
\end{align*}
for $0\leq t \le T^*$ and $n \in \bbn$, where $C>0$ is independent of $n$.
\end{lemma}
\begin{proof}
First, we deduce from \eqref{AppA-5} that
\begin{align*}
\frac12\frac{d}{dt}\|g^{n+1} - g^n\|_{L^2}^2 
 &= -\int_{\T^d} u^n \cdot \nabla( g^{n+1} - g^n) (g^{n+1} - g^n) \,dx - \int_{\T^d} (u^n - u^{n-1})\cdot \nabla g^n (g^{n+1} - g^n) \,dx\\
&\quad -\int_{\T^d} \nabla \cdot (u^{n+1} - u^n) (g^{n+1} - g^n) \,dx\\
& \le \frac{\|\nabla \cdot u^n\|_{L^\infty}}{2}\|g^{n+1} - g^n\|_{L^2}^2 + \|\nabla g^n \|_{L^\infty} \|u^n - u^{n-1}\|_{L^2}\|g^{n+1} - g^n\|_{L^2}\\
&\quad - \int_{\T^d} \nabla \cdot (u^{n+1} - u^n) (g^{n+1} - g^n) \,dx\\
& \le C\left(\|g^{n+1}-g^n\|_{L^2}^2 + \|u^n - u^{n-1}\|_{L^2}^2\right) - \int_{\T^d} \nabla \cdot (u^{n+1} - u^n) (g^{n+1} - g^n) \,dx.
\end{align*}
Next, we estimate $(u^{n+1}-u^n)$ as
\begin{align*}
\frac12\frac{d}{dt}&\|u^{n+1} - u^n\|_{L^2}^2\\
&= -\int_{\T^d} u^n \cdot \nabla(u^{n+1} - u^n) \cdot (u^{n+1} - u^n)\,dx - \int_{\T^d}(u^n - u^{n-1}) \cdot \nabla u^n \cdot (u^{n+1} - u^n)\,dx\\
&\quad -\int_{\T^d}\nabla(g^{n+1} - g^n) \cdot (u^{n+1} - u^n)\,dx -\|u^{n+1}-u^n\|_{L^2}^2 \\
&\quad -\int_{\T^d} \nabla K \star(e^{g^n} - e^{g^{n-1}})\cdot (u^{n+1} - u^n)\,dx + \int_{\T^d} (v^n - v^{n-1})\cdot (u^{n+1} - u^n)\,dx \\
&\le \frac{\|\nabla \cdot u^n\|_{L^\infty}}{2}\|u^{n+1} - u^n\|_{L^2}^2 + \|\nabla u^n\|_{L^\infty}\|u^{n+1} - u^n\|_{L^2}\|u^n - u^{n-1}\|_{L^2}\\
&\quad -\int_{\T^d}\nabla(g^{n+1} - g^n) \cdot (u^{n+1} - u^n)\,dx + \|\nabla K\|_{L^1}\|e^{g^n} - e^{g^{n-1}}\|_{L^2}\|u^{n+1} - u^n\|_{L^2}\\
&\quad +\|v^n - v^{n-1}\|_{L^2}\|u^{n+1}-u^n\|_{L^2} \\
&\le C\left(\|u^{n+1}-u^n\|_{L^2}^2 + \|u^n- u^{n-1}\|_{L^2}^2 +\|g^n - g^{n-1}\|_{L^2}^2 + \|v^n- v^{n-1}\|_{L^2}^2 \rt)\\
&\quad -\int_{\T^d}\nabla(g^{n+1} - g^n) \cdot (u^{n+1} - u^n)\,dx,
\end{align*}
where we used the mean value theorem to obtain
\[
\left\|e^{g^n} - e^{g^{n-1}}\right\|_{L^2} \le \exp\lt(\max\{ \|g^n\|_{L^\infty}, \|g^{n-1}\|_{L^\infty}\}\rt)\|g^n - g^{n-1}\|_{L^2}\le C\|g^n - g^{n-1}\|_{L^2}.
\]
Finally, for $(v^{n+1} - v^n)$, we have
\begin{align*}
\frac12\frac{d}{dt}&\|v^{n+1} - v^n\|_{L^2}^2\\
&= -\int_{\T^d} v^n \cdot \nabla(v^{n+1} - v^n) \cdot (v^{n+1} - v^n)\,dx -\int_{\T^d} (v^n -v^{n-1})\cdot \nabla v^n \cdot (v^{n+1} - v^n)\,dx\\
&\quad +\int_{\T^d} \Delta(v^{n+1} -v^n)\cdot (v^{n+1}-v^n)\,dx  -\int_{\T^d}(e^{g^n}-e^{g^{n-1}})( v^{n+1}-u^n)\cdot(v^{n+1}-v^n)\,dx\\
&\quad - \int_{\T^d} e^{g^{n-1}}|v^{n+1} - v^n|^2\,dx + \int_{\T^d} e^{g^{n-1}}(u^n - u^{n-1})\cdot (v^{n+1} -v^n)\,dx\\
&\le  \|\nabla v^n\|_{L^\infty}\|v^n - v^{n-1}\|_{L^2}\|v^{n+1} - v^n\|_{L^2} + \|e^{g^n} - e^{g^{n-1}}\|_{L^2}(\|v^{n+1}\|_{L^2} + \|u^n\|_{L^2}) \|v^{n+1}-v^n\|_{L^2}\\
&\quad + e^{\|g^{n-1}\|_{L^\infty}}\|u^n - u^{n-1}\|_{L^2}\|v^{n+1} - v^n\|_{L^2}\\
&\le C\left( \|v^{n+1} - v^n\|_{L^2}^2 + \|u^n- u^{n-1}\|_{L^2}^2 +\|g^n - g^{n-1}\|_{L^2}^2 + \|v^n- v^{n-1}\|_{L^2}^2  \rt).
\end{align*}
We combine all of the above estimates to obtain
\begin{align*}
\frac{d}{dt}&\left( \|(g^{n+1} -g^n)(\cdot,t)\|_{L^2}^2 + \|(u^{n+1} - u^n)(\cdot,t)\|_{L^2}^2+ \|(v^{n+1} - v^n)(\cdot,t)\|_{L^2}^2\right)\\
&\le C\left(\|(g^{n+1} -g^n)(\cdot,t)\|_{L^2}^2 + \|(u^{n+1} - u^n)(\cdot,t)\|_{L^2}^2 + + \|(v^{n+1} - v^n)(\cdot,t)\|_{L^2}^2\right)\\
&\quad +C\left(\|(g^n -g^{n-1})(\cdot,t)\|_{L^2}^2 + \|(u^n - u^{n-1})(\cdot,t)\|_{L^2}^2 + \|(v^n - v^{n-1})(\cdot,t)\|_{L^2}^2 \right)
\end{align*}
for $0 \leq t \le T^*$, where $C>0$ is independent of $n$. Therefore, we can apply Gr\"onwall's lemma to get the desired result.
\end{proof}

\subsection{Proof of Lemma \ref{L5.1}}
Now, we prove the well-posedness of strong solutions to \eqref{E-2}. First, Lemma \ref{LA.2} implies
\[
g^n \to g, \quad  u^n \to u, \quad \mbox{and} \quad v^n \to v \quad \mbox{in } \ \mathcal{C}([0,T];L^2(\T^d))
\]
as $n \to \infty$. Moreover, we can extend the convergence in $\mathcal{C}([0,T];L^2(\T^d))$ to that in $\mathcal{C}([0,T];H^{s-1}(\T^d))$ by interpolating this with the uniform bound in $\mc([0,T];H^s(\T^d))$ from Corollary \ref{CA.1}:
\[
g^n \to g \quad \mbox{in } \ \mathcal{C}([0,T];H^{s-1}(\T^d))  \quad \mbox{and} \quad  u^n \to u. \quad v^n \to v \quad \mbox{in } \ \mathcal{C}([0,T];H^{s-1}(\T^d))
\]
as $n \to \infty$. Concerning the $H^s$-regularity of $(g,u,v)$, we can use a standard argument from functional analysis. We refer to \cite{CCZ16} for details.

For the uniqueness, let  $(g, u, v)$ and $(\hat g, \hat u, \hat v)$ be two solutions with the same initial data $(g_0, u_0,v_0)$. Then, the Cauchy estimate in Lemma \ref{LA.2} implies
\[\begin{aligned}
\|&(g - \hat g)(\cdot,t)\|_{L^2}^2 + \|(u - \hat u)(\cdot,t)\|_{L^2}^2 + \|(v - \hat v)(\cdot,t)\|_{L^2}^2  \\
&\quad \le C\int_0^t \lt(\|(g-\hat g)(\cdot,s)\|_{L^2}^2 +  \|(u - \hat u)(\cdot,s)\|_{L^2}^2 + \|(v - \hat v)(\cdot,s)\|_{L^2}^2 \rt) ds
\end{aligned}
\]
for $t \le T^*$. Thus, we apply Gr\"onwall's lemma to the above to conclude the uniqueness of solutions.


\begin{thebibliography}{10}
\bibitem{AIK14} O. Anoshchenko, S. Iegorov, and E. Khruslov, Global weak solutions of the Navier--Stokes/Fokker--Planck/Poisson linked equations, J. Math. Phys. Anal. Geo., 10, (2014), 267--299.
\bibitem{AKS10} O. Anoshchenko, E. Khruslov, and H. Stephan, Global weak solutions to the Navier--Stokes--Vlasov--Poisson system, J. Math. Phys. Anal. Geo., 6, (2010), 143--182.
\bibitem{BBJM05}  C. Baranger, L. Boudin, P.-E. Jabin, and S. Mancini, A modeling of biospray for the upper airways. CEMRACS 2004 -- mathematics and applications to biology
and medicine, ESAIM Proc., 14, (2005), 41--47.
\bibitem{BR93} J. Batt and G. Rein, A rigorous stability result for the Vlasov-Poisson system in three dimensions, Ann. Mat. Pura Appl., 164, (1993), 133--154.
\bibitem{BT06} L. Bedin and M. Thompson, Motion of a charged particle in ionized fluids, Math. Models Methods Appl. Sci., 16, (2006), 1271--1318.
\bibitem{BT13} L. Bedin and M. Thompson, Existence theory for a Poisson--Nernst--Planck model of electrophoresis, Commun. Pure. Appl. Anal., 12, (2013), 157--206.
\bibitem{BDGM09} L. Boudin, L. Desvillettes, C. Grandmont, and A. Moussa, Global existence of solution for the coupled Vlasov and Navier--Stokes equations, Differ. Integral Equ., 22, (2009), 1247--1271.
\bibitem{BGLM15} L. Boudin, C. Grandmont, A. Lorz, and A. Moussa: Modelling and Numerics for Respiratory Aerosols, Commun. Comput. Phys., 18, (2015), 723--756.
\bibitem{B49} H. C. Brinkman, A calculation of the viscous force exerted by a flowing fluid on a dense swarm of particles, Appl. Sci. Res., 1, (1949), 27--24.
\bibitem{CC20} J. A. Carrillo and Y.-P. Choi, Quantitative error estimates for the large friction limit of Vlasov equation with nonlocal forces, Ann. Inst. H. Poincar\'e Anal. Non Lin\'eaire, 37, (2020), 925--954.
\bibitem{CCJpre} J. A. Carrillo, Y.-P. Choi, and J, Jung, Quantifying the hydrodynamic limit of Vlasov-type equations with alignment and nonlocal forces, Math. Models Methods Appl. Sci., to appear.
\bibitem{CCK16} J. A. Carrillo, Y.-P. Choi, and T. K. Karper, On the analysis of a coupled kinetic-fluid model with local alignment forces, Ann. I. H. Poincar\'e -- AN., 33, (2016), 273--307.
\bibitem{CCZ16} J. A. Carrillo, Y.-P. Choi, and E. Zatorska, On the pressureless damped Euler--Poisson equations with quadratic confinement: critical thresholds and large-time behavior, Math. Models Methods Appl. Sci., 26, (2016), 2311--2340.
\bibitem{CKL11} M. Chae, K. Kang, and J. Lee, Global existence of weak and classical solutions for the Navier--Stokes--Fokker--Planck equations, J. Differential Equations, 251, (2011), 2431--2465.
\bibitem{CW96} G.-Q. Chen and D. Wang, Convergence of shock capturing schemes for the compressible Euler--Poisson equations, Comm. Math. Phys., 179, (1996), 333--364.
\bibitem{Choi15} Y.-P. Choi, Compressible Euler equations interacting with incompressible flow, Kinet. Relat. Models, 8, (2015),  335--358.
\bibitem{Choi16} Y.-P. Choi, Global classical solutions and large-time behavior of the two-phase fluid model, SIAM J. Math. Anal., 48, (2016), 3090--3122.
\bibitem{Cpre} Y.-P. Choi, Large friction limit of pressureless Euler equations with nonlocal forces, preprint, arXiv:2002.01691.
\bibitem{CK16}Y.-P. Choi and B. Kwon, The Cauchy problem for the pressureless Euler/isentropic Navier--Stokes equations, J. Differential Equations, 261, (2016), 654--711.
\bibitem{CJpre} Y.-P. Choi and J. Jung, Asymptotic analysis for Vlasov--Fokker--Planck/compressible Navier--Stokes equations with a density-dependent viscosity, Hyperbolic problems: theory, numerics, applications, 145-163, AIMS Ser. Appl. Math., 10, Am. Inst. Math. Sci. (AIMS), Springfield, MO, (2020).
\bibitem{CJpre2} Y.-P. Choi and J. Jung, Asymptotic analysis for a Vlasov--Fokker--Planck/Navier--Stokes system in a bounded domain, preprint, arXiv: 1912.13134.
\bibitem{CJpre3} Y.-P. Choi and J. Jung, On the Cauchy problem for the pressureless Euler--Navier--Stokes system in the whole space, preprint, arXiv: 2008.00467.
\bibitem{CY20} Y.-P. Choi and S.-B. Yun, Global existence of weak solutions for Navier--Stokes--BGK system, Nonlinearity, 33, (2020), 1925--1955.
\bibitem{De86} P. Degond, Global existence of smooth solutions for the Vlasov--Fokker--Planck equation in 1 and 2 space dimensions, Ann. Sci. \'Ecole Norm. Super., 19, (1986), 519--542.
\bibitem{Des10} L. Desvillettes, Some aspects of the modelling at different scales of multiphase flows, Comput. Methods Appl. Mech. Eng., 199, (2010), 1265--1267.
\bibitem{DGR08} L. Desvillettes, F. Golse, and V. Ricci, The mean-field limit for solid particles in a Navier--Stokes flow, J. Stat. Phys., 131, (2008), 941--967.
\bibitem{En98} S. Engelberg, Formation of singularities in the Euler and Euler--Poisson equations, Physica D, 98, (1996), 67--74.
\bibitem{ELT01} S. Engelberg, H. Liu, and E. Tadmor, Critical thresholds in Euler--Poisson equations, Indiana Univ. Math. J., 50, (2001), 109--157.
\bibitem{FK19} A. Figalli and M.-J. Kang, A rigorous derivation from the kinetic Cucker--Smale model to the pressureless Euler system with nonlocal alignment, Anal. PDE, 12, (2019), 843--866.
\bibitem{GJV04} T. Goudon, P.-E. Jabin, and A. Vasseur, Hydrodynamic limit for the Vlasov--Navier--Stokes equations: I. Light particles regime, Indiana Univ. Math. J., 53, (2004), 1495--1515.
\bibitem{GJV04_2} T. Goudon, P.-E. Jabin, and A. Vasseur, Hydrodynamic limit for the Vlasov--Navier--Stokes equations: II. Fine particles regime, Indiana Univ. Math. J., 53, (2004), 1517--1536.
\bibitem{HKK14} S.-Y. Ha, M.-J. Kang, and B. Kwon, A hydrodynamic model for the interaction of Cucker--Smale particles and incompressible fluid, Math. Models Methods Appl. Sci., 24, (2014), 2311--2359.
\bibitem{HDRD13} M. R. Hossan, R. Dillon, A. K. Roy, and P. Dutta, Modeling and simulation of dielectrophoretic particle-particle interactions and assembly, J. Colloid Interface Sci., 394, (2013), 619--629.
\bibitem{KMT13} T. Karper, A. Mellet, and K. Trivisa, Existence of weak solutions to kinetic flocking models, SIAM Math. Anal., 45, (2013), 215--243.
\bibitem{KMT14} T. Karper, A. Mellet and K. Trivisa, On strong local alignment in the kinetic Cucker--Smale model, Hyperbolic conservation laws and related analysis with applications, 227--242, Springer Proc. Math. Stat., 49, Springer, Heidelberg, 2014.
\bibitem{KMT15} T. K. Karper, A. Mellet, and K. Trivisa, Hydrodynamic limit of the kinetic Cucker--Smale flocking model, Math. Models Methods Appl. Sci., 25, (2015), 131--163.
\bibitem{KC93} H. J. Keh and J. B. Chen, Particle interactions in electrophoresis: V. Motion of multiple spheres with thin but finite electrical double layers, J. Colloid Interface Sci., 158, (1993), 199--222.
\bibitem{LT17} C. Lattanzio and A. E. Tzavaras, From gas dynamics with large friction to gradient flows describing diffusion theories, Comm. Partial Differential Equations, 42, (2017), 261--290.
\bibitem{LL13} Y. K. Lee and H. L. Liu, Thresholds in three-dimensional restricted Euler--Poisson equations, Physica D, 262, (2013), 59--70.
\bibitem{LT02} H. Liu and  E. Tadmor, Spectral dynamics of the velocity gradient field in restricted fluid flows. Comm. Math. Phys., 228, (2002), 435--466.
\bibitem{LT03} H. Liu and E. Tadmor, Critical thresholds in 2D restricted Euler--Poisson equations, SIAM J. Appl. Math., 63, (2003), 1889--1910.
\bibitem{MV07} A. Mellet and A. Vasseur, Global weak solutions for a Vlasov--Fokker--Planck/Navier--Stokes system of equations, Math. Models Methods Appl. Sci., 17, (2007), 1039--1063.
\bibitem{MV08} A. Mellet and A. Vasseur, Asymptotic analysis for a Vlasov--Fokker--Planck/compressible Navier--Stokes equations, Comm. Math. Phys., 281, (2008), 573--596.
\bibitem{OR81} P. J. O'Rourke, Collective drop effects on vaporizing liquid sprays, PhD thesis, Los Alamos National Laboratory, (1981).
\bibitem{Pe90} B. Perthame, Nonexistence of global solutions to Euler--Poisson equations for repulsive forces, Japan J. Appl. Math., 7, (1990), 363--367.
\bibitem{PS98}B. Perthame and P. E. Souganidis, A limiting case for velocity averaging, Ann. Sci. \'Ecole Norm. Sup., 31, (1998), 591--598.
\bibitem{Tit58} E. Titchmarsh, Eigenfunction expansions associated with second-order differential equations, Part II, Oxford University Press, Oxford, UK, 1958.
\bibitem{Vill02} C. Villani, A review of mathematical topics in collisional kinetic theory, Handbook Mathematical Fluid Dynamics, 1, (2002), 71--305.
\bibitem{Wi58} F.A. Williams, Spray combustion and atomization, Phys. Fluids, 1, (1958), 
541--555.
\bibitem{YY18} L. Yao and C. Yu, Existence of global weak solutions for the Navier--Stokes--Vlasov--Boltzmann equations, J. Differential Equations, 265, (2018), 5575--5603.
\bibitem{Yu13} C. Yu, Global weak solutions to the incompressible Navier--Stokes--Vlasov equations, J. Math. Pures Appl., 100, (2013), 275--293.
\end{thebibliography}
\end{document}